\documentclass[a4paper,onecolumn,english]{article}

\usepackage[T1]{fontenc}
\usepackage[utf8]{inputenc}
\usepackage{babel}
\usepackage{pdflscape}
\usepackage{enumitem}
\usepackage{amsmath}
\usepackage{amsfonts}
\usepackage{amssymb}
\usepackage{amsthm}
\usepackage{thmtools}
\usepackage{dsfont}
\usepackage{bm}
\usepackage{relsize}
\usepackage[title]{appendix}
\usepackage[left=2.5cm,right=2.5cm,top=2.5cm,bottom=2.5cm]{geometry}
\usepackage[colorlinks=true,
            linkcolor = blue,
            urlcolor  = blue,
            citecolor = red,
			breaklinks]{hyperref}

\allowdisplaybreaks

\numberwithin{equation}{section}

\newcommand{\eps}{\varepsilon}
\newcommand{\mb}[1]{\mathbf{#1}}
\newcommand*{\ccdot}{\kern-.12em\cdot\kern-.12em}
\newcommand\restrict[1]{\raisebox{-.5ex}{$|$}_{#1}}
\newcommand{\supp}{\mathrm{supp}}
\newcommand{\warrow}{\overset{w}{\longrightarrow}}
\newcommand{\varrow}{\overset{v}{\longrightarrow}}

\newtheoremstyle{slplain}
  {0.4cm}
  {0.4cm}
  {\upshape}
  {}
  {\bfseries}
  {.}
  { }
  {}
  
\newtheoremstyle{itplain}
    {0.4cm}
    {0.4cm}
    {\itshape}
    {}
    {\bfseries}
    {.}
    { }
    {}
  
\declaretheorem[style=slplain,numberwithin=section]{definition}
\declaretheorem[style=slplain,sibling=definition]{example}
\declaretheorem[style=slplain,sibling=definition]{remark}
\declaretheorem[style=slplain,sibling=definition]{assumption}

\declaretheorem[style=itplain,sibling=definition]{theorem}

\declaretheorem[style=itplain,sibling=definition]{proposition}
\declaretheorem[style=itplain,sibling=definition]{lemma}
\declaretheorem[style=itplain,sibling=definition]{corollary}

\AtBeginDocument{
	\addtolength{\abovedisplayskip}{0.125ex}
	\addtolength{\abovedisplayshortskip}{0.125ex}
	\addtolength{\belowdisplayskip}{0.125ex}
	\addtolength{\belowdisplayshortskip}{0.125ex}
}

\renewenvironment{abstract}{%
\noindent\hfill\begin{minipage}{0.92\textwidth}
\rule{\textwidth}{1pt}}
{\par\noindent\rule{\textwidth}{1pt}\end{minipage}\hfill}

\let\OLDthebibliography\thebibliography
\renewcommand\thebibliography[1]{
  \OLDthebibliography{#1}
  \setlength{\parskip}{0pt}
  \setlength{\itemsep}{3pt plus 0.3ex}
}

\title{\bfseries Sturm-Liouville hypergroups without the\linebreak compactness axiom\\[0.3cm]}

\author{Rúben Sousa
\thanks{Corresponding author. CMUP, Departamento de Matemática, Faculdade de Ciências, Universidade do Porto, Rua do Campo Alegre 687, 4169-007 Porto, Portugal. Email: \texttt{rubensousa@fc.up.pt}}
\and
Manuel Guerra \thanks{CEMAPRE and ISEG (School of Economics and Management), Universidade de Lisboa, Rua do Quelhas 6, 1200-781 Lisbon, Portugal. Email: \texttt{mguerra@iseg.ulisboa.pt}}
\and
Semyon Yakubovich \thanks{CMUP, Departamento de Matemática, Faculdade de Ciências, Universidade do Porto, Rua do Campo Alegre 687, 4169-007 Porto, Portugal. Email: \texttt{syakubov@fc.up.pt}}
\\[0.3cm]
}
\date{\today\\}

\begin{document}

\maketitle

\begin{abstract}
	 \small
	 \parbox{\linewidth}{\vspace{-2pt}
\begin{center} \bfseries Abstract \vspace{-8pt} \end{center}
	 
	 \-\ \quad We establish a positive product formula for the solutions of the Sturm-Liouville equation $\ell(u) = \lambda u$, where $\ell$ belongs to a general class which includes singular and degenerate Sturm-Liouville operators. Our technique relies on a positivity theorem for possibly degenerate hyperbolic Cauchy problems and on a regularization method which makes use of the properties of the diffusion semigroup generated by the Sturm-Liouville operator.
	 	 
	 \-\ \quad We show that the product formula gives rise to a convolution algebra structure on the space of finite measures, and we discuss whether this structure satisfies the basic axioms of the theory of hypergroups. We introduce the notion of a degenerate hypergroup of full support and improve the known existence theorems for Sturm-Liouville hypergroups. Convolution-type integral equations on weighted Lebesgue spaces are also introduced, and a solvability condition is established. \vspace{5pt}
	 
     \-\ \quad \textbf{Keywords:} Product formula, hypergroup, generalized convolution, Sturm-Liouville spectral theory, degenerate hyperbolic equation, convolution integral equations. \vspace{6.5pt}
     }
\end{abstract}

\vspace{8pt}

\begingroup
\let\clearforchapter\relax
\section{Introduction}
\endgroup

A hypergroup is a binary operation $*$ on the space $\mathcal{M}_\mathbb{C}(K)$ of finite complex measures on an underlying space $K$, which preserves the subset of probability measures on $K$ and gives rise to a structure of Banach algebra with unit on $\mathcal{M}_\mathbb{C}(K)$. In the axiomatic definition of hypergroup introduced by Jewett in \cite{jewett1975}, the operation $*$, known as generalized convolution, is required to satisfy axioms of continuity and compactness of support; the compactness axiom requires, in particular, that the convolution of Dirac measures is a measure of compact support. An extensive theory of (probabilistic) harmonic analysis has been developed in the context of hypergroups, see the monographs \cite{bloomheyer1994,berezansky1998} and references therein. See also \cite{bakryzribi2017,fruchtl2018,heyeretal2017,szekelyhidietal2019} for recent work on hypergroup structures.

In this paper we describe a class of hypergroups which does not satisfy the compactness axiom, and yet allows one to develop harmonic analysis in the sense indicated above.

Starting from the seminal works of Delsarte \cite{delsarte1938} and Levitan \cite{levitan1940} on generalized translation operators, the development of the theory of hypergroups was largely motivated by the study of Sturm-Liouville differential operators on an interval $(a,b)$ of the real line. The key idea here is the following: it is well known that the eigenfunction expansion of a Sturm-Liouville operator, say, of the form
\[
\ell = -{1 \over r} {d \over dx} \Bigl( p \, {d \over dx}\Bigr), \qquad a < x < b
\]
gives rise (under certain conditions) to an integral transform $(\mathcal{F}h)(\lambda) := \int_a^b h(x) \, w_\lambda(x) \, r(x) dx$ ($\lambda \in \mathbb{R}$) which is an isometry between weighted $L_2$-spaces; here $\{w_\lambda\}$ is a family of solutions of the Sturm-Liouville equation $\ell(u) = \lambda u$. Now, one may extend the transformation $\mathcal{F}$ to measures $\mu \in \mathcal{M}_\mathbb{C}[a,b)$ by defining
\begin{equation} \label{eq:intro_SLtransfmeas}
(\mathcal{F}\mu)(\lambda) \equiv \widehat{\mu}(\lambda) := \int_{[a,b)} w_\lambda(x) \, \mu(dx),
\end{equation}
and then it is natural to ask: \emph{does there exist a (generalized) convolution operator $*$ which is trivialized by the transformation \eqref{eq:intro_SLtransfmeas}, in the sense that the property $\widehat{\mu * \nu} = \widehat{\mu} \ccdot \widehat{\nu}$ holds for all $\mu,\nu \in \mathcal{M}_\mathbb{C}[a,b)$?} If $\mu$ and $\nu$ are taken to be Dirac measures at the points $x,y \in [a,b)$, then the trivialization property reads \begin{equation} \label{eq:intro_prodform}
w_\lambda(x) \, w_\lambda(y) = \int_{[a,b)} w_\lambda \, d\bm{\nu}_{x,y}
\end{equation}
where $\bm{\nu}_{x,y} = \delta_x * \delta_y$. The construction of generalized convolutions is therefore closely related to the problem of existence of a so-called \emph{product formula} for the solutions of the Sturm-Liouville equation $\ell(u) = \lambda u$; in this problem, the goal is to determine a family $\{\bm{\nu}_{x,y}\} \subset \mathcal{M}_\mathbb{C}[a,b)$ such that \eqref{eq:intro_prodform} holds. For the hypergroup axioms to hold we actually need the $\bm{\nu}_{x,y}$ to be probability measures; in this case we say that \eqref{eq:intro_prodform} is a \emph{hypergroup-like product formula}.

The Bessel operator $- {d \over dx^2} - {2\alpha + 1 \over x} {d \over dx}$ and the Jacobi operator $- {d \over dx^2} - [(2\alpha + 1) \coth x + (2\beta + 1) \tanh x] {d \over dx}$ are standard examples of Sturm-Liouville operators on the half-line $[0,\infty)$ for which the kernel of the associated Sturm-Liouville integral transform (the Hankel and Jacobi transform, respectively) admits a hypergroup-like product formula where the measures $\bm{\nu}_{x,y}$ have been computed in closed form (see \cite{hirschman1960} and \cite{koornwinder1984}, respectively). More generally, it was shown by Zeuner \cite{zeuner1992} that any Sturm-Liouville operator of the form $- {1 \over A} {d \over dx} \bigl(A {d \over dx}\bigr)$, where $A \in \mathrm{C}^1(0,\infty)$ satisfies suitable assumptions (see Subsection \ref{sub:SLhyp_nondegen} below) also admits a hypergroup-like product formula; in general, the measures $\bm{\nu}_{x,y}$ of the product formula \eqref{eq:intro_prodform} are not known in closed form, but their existence can be proved using classical results on positivity properties for hyperbolic partial differential equations, cf.\ \cite{weinberger1956}.

A general property of the Sturm-Liouville operators considered by Zeuner is that the support $\supp(\bm{\nu}_{x,y})$ of the measures in the product formula is contained in $[|x-y|,x+y]$; in particular, $\supp(\bm{\nu}_{x,y})$ is compact, as required by the usual hypergroup axioms. However, the situation is quite different for the Whittaker convolution, generated by the normalized Whittaker differential operator $x^2 {d^2 \over dx^2} + (1+2(1-\alpha) x){d \over dx}$ on the half-line $[0,\infty)$. In fact, in this case the measures in the product formula, whose closed form expression was recently determined by the authors in \cite{sousaetal2018a,sousaetal2018b}, satisfy $\supp(\bm{\nu}_{x,y}) = [0,\infty)$ for all $x,y > 0$. But it turns out that the probability-preserving and continuity axioms are satisfied by the Whittaker convolution, and therefore it is still possible to develop harmonic analysis on the measure algebra $\bigl( \mathcal{M}_\mathbb{C}[0,\infty), * \bigr)$ (see \cite{sousaetal2018a,sousaetal2018b}). Moreover, one can show that the Whittaker operator restricted to any interval $[c,\infty)$, $c > 0$, can be reduced by a change of variable to an operator belonging to the class introduced by Zeuner, and therefore determines a convolution satisfying the compact support axiom. Hence it is natural to interpret the measure algebra associated to the Whittaker convolution as a degenerate hypergroup and to wonder if it is possible to construct degenerate hypergroup structures for other Sturm-Liouville operators. The use of the term ``degenerate" is further justified by the fact that in the limit $c=0$ the hyperbolic Cauchy problem associated with $\ell$ (defined in Subsection \ref{sub:prodform_hypCP}) becomes parabolically degenerate at the initial line.

In this paper, our purpose is to introduce a new technique for proving the existence of a hypergroup-like product formula for Sturm-Liouville operators whose associated hyperbolic Cauchy problem is possibly parabolically degenerate at the initial line. Our technique is based on a regularization method which we now briefly sketch. The inversion formula for the integral transform $\mathcal{F}$ generated by $\ell$ provides a formal candidate for the measure $\bm{\nu}_{x,y}$, namely the inverse transform $\mathcal{F}^{-1}[w_{(\cdot)}(x) \, w_{(\cdot)}(y)]$. However, the inversion integral is, in general, divergent. To get around this, the idea is to consider instead the regularized inverse transform $\mathcal{F}^{-1}[e^{-t(\cdot)} w_{(\cdot)}(x) \, w_{(\cdot)}(y)]$, where $t > 0$, and to prove that the presence of the exponential term ensures the convergence of the inversion formula. This regularization effect is closely related with the smoothing properties of the one-parameter semigroup generated by the Sturm-Liouville operator, which are well-known from the theory of one-dimensional (Feller) diffusion semigroups \cite{feller1954,mckean1956}. It will then be seen that the measure $\bm{\nu}_{x,y}$ can be recovered from the measures $\bm{\nu}_{t,x,y}$ of the product formula for $e^{-t\lambda} w_{\lambda}(x) \, w_{\lambda}(y)$ as the weak limit as $t \downarrow 0$. The weak convergence argument relies on the nontrivial fact that the $\bm{\nu}_{t,x,y}$ (and therefore also $\bm{\nu}_{x,y}$) are probability measures; to justify this, we use a partial differential equation approach based on the maximum principle for hyperbolic equations.

Since we deal with hyperbolic Cauchy problems which may be parabolically degenerate, the classical theory of hyperbolic problems in two variables (cf.\ e.g.\ \cite[Chapter V]{courant1952}) is, in general, not applicable. To overcome this, we use the spectral theory of Sturm-Liouville operators (cf.\ \cite[Chapter 9]{teschl2014}) to deduce existence, uniqueness and positivity results for a general class of possibly degenerate Cauchy problems. This class also includes many (uniformly) hyperbolic equations with singularities which fall outside the scope of the classical methods. In such singular cases it will be seen that the uniform hyperbolicity yields a product formula where the $\bm{\nu}_{x,y}$ have compact support and the resulting generalized convolution operator satisfies all the hypergroup axioms, leading to an existence theorem for Sturm-Liouville hypergroups which generalizes previous results in the literature. On the other hand, as we will see, in the presence of parabolic degeneracy the product formula is such that the measures $\bm{\nu}_{x,y}$ are supported on the full interval $[a,b)$. This allows us to interpret the Whittaker convolution as a particular case of a general family of degenerate Sturm-Liouville hypergroups of full support; this is relevant because, to the best of our knowledge, no full support convolution structures generated by Sturm-Liouville operators other than the Whittaker operator were known to exist prior to this work.

As mentioned above, the main facts of harmonic analysis can be generalized to Sturm-Liouville hypergroups; in fact, most of the results of \cite{bloomheyer1994,berezansky1998} on standard Sturm-Liouville hypergroups extend (with some modifications) to degenerate hypergroups. However, the concept of convolution integral equations on hypergroups seems to remain unexplored in existing literature. Let us recall that the integral equation of the second kind on the real line with difference kernel, given by
\begin{equation} \label{eq:intro_convinteq}
u(x) + \int_{-\infty}^\infty u(y) \, K(x-y)\, dy = f(x),
\end{equation}
can be equivalently written as $u + u \star K = f$, where $\star$ is the ordinary convolution; therefore, \eqref{eq:intro_convinteq} is usually called a convolution integral equation. It was proved by Wiener (see \cite[p.\ 164]{krein1962}) that 
\[
1+(\mathfrak{F}K)(\lambda) \neq 0 \quad (\lambda \in \mathbb{R}), \qquad \text{where } (\mathfrak{F}K)(\lambda) = \int_{-\infty}^\infty e^{i\lambda x} K(x) dx,
\]
is a necessary and sufficient condition for the existence and uniqueness of solution for the convolution equation \eqref{eq:intro_convinteq}. In this paper, these ideas are developed in the context of a general Sturm-Liouville hypergroup: we define a Sturm-Liouville convolution equation as an integral equation which can be written as $u + u * K = f$, where $*$ is the generalized convolution, and we establish an analogous existence and uniqueness theorem, with the Fourier transform $(\mathfrak{F}K)(\lambda)$ being replaced by the Sturm-Liouville integral transform $(\mathcal{F}K)(\lambda)$. As we will see, this is a consequence of the fact that the generalized convolution of functions, seen as an operator on certain weighted $L_1$ spaces, gives rise to a Banach algebra structure which admits an analogue of the Wiener-Lévy theorem. As far as we are aware, similar results were only known to hold for very special cases, namely the Kontorovich-Lebedev convolution \cite[Chapter 17]{yakubovichluchko1994}, the Whittaker convolution \cite{sousaetal2018a} and some other cases which are reducible to the ordinary convolution (e.g.\ \cite{thaoetal2014,khoaetal2017}).

The paper is organized as follows. Section \ref{sec:prelim} collects some preliminary facts about the solutions of Sturm-Liouville boundary value problems and the related eigenfunction expansions. In Section \ref{sec:laplacerep} we prove that the kernel $w_\lambda(x)$ of the integral transform generated by $\ell$ can be written as the Fourier transform of a probability measure, thereby generalizing a result which is known to hold for Sturm-Liouville hypergroups; this so-called Laplace-type representation is later used in the proof of the product formula. The proof of the hypergroup-like product formula for the functions $w_\lambda(x)$ is given in Section \ref{sec:prodform}. In Section \ref{sec:genconv_hypergr} we show, using the properties of the transformation \eqref{eq:intro_SLtransfmeas}, that the convolution determined by the product formula is continuous in the weak topology and yields a positivity-preserving Banach algebra structure on $\mathcal{M}_\mathbb{C}[a,b)$; we then discuss the support of the convolution of Dirac measures in the nondegenerate and degenerate cases, and relate our results with the axioms of hypergroups. In Section \ref{sec:convalgL1} each Sturm-Liouville convolution is shown to lead to a Wiener-Lévy type theorem on a family of weighted Lebesgue spaces, and this theorem is applied to the study of convolution integral equations on hypergroups. The Appendix contains the proofs of some properties of the associated hyperbolic Cauchy problem. \vspace{7pt}

\section{Preliminaries} \label{sec:prelim}

The following notations will be used throughout the paper. For a subset $E \subset \mathbb{R}^d$, $\mathrm{C}(E)$ is the space of continuous complex-valued functions on $E$; $\mathrm{C}_\mathrm{b}(E)$, $\mathrm{C}_0(E)$ and $\mathrm{C}_\mathrm{c}(E)$ are, respectively, its subspaces of bounded continuous functions, of continuous functions vanishing at infinity and of continuous functions with compact support; $\mathrm{C}^k(E)$ stands for the subspace of $k$ times continuously differentiable functions. $\mathrm{B}_\mathrm{b}(E)$ is the space of complex-valued bounded and Borel measurable functions. The corresponding spaces of real-valued functions are denoted by $\mathrm{C}(E,\mathbb{R})$, $\mathrm{C}_\mathrm{b}(E,\mathbb{R})$, etc. For a given measure $\mu$ on $E$, $L_p(E;\mu)$ ($1 \leq p \leq \infty$) denotes the Lebesgue space of complex-valued $p$-integrable functions with respect to $\mu$. The restriction of a function $f:E \longrightarrow \mathbb{C}$ to a subset $B \subset E$ is denoted by $f\restrict{B}$. The space of probability (respectively, finite positive, finite complex) Borel measures on $E$ will be denoted by $\mathcal{P}(E)$ (respectively, $\mathcal{M}_+(E)$, $\mathcal{M}_\mathbb{C}(E)$). The total variation of $\mu \in \mathcal{M}_\mathbb{C}(E)$ is denoted by $\|\mu\|$, and $\delta_x$ denotes the Dirac measure at a point $x$.

In all that follows we consider a Sturm-Liouville differential expression of the form 
\begin{equation} \label{eq:shypPDE_elldiffexpr}
\ell = -{1 \over r} {d \over dx} \Bigl( p \, {d \over dx}\Bigr), \qquad x \in (a,b)
\end{equation}
($-\infty \leq a < b \leq \infty$), where $p$ and $r$ are (real-valued) coefficients such that $p(x), r(x) > 0$ for all $x \in (a,b)$ and $p, p', r$ and $r'$ are locally absolutely continuous on $(a,b)$. Concerning the behavior of the coefficients at the boundary $x=a$, we will always assume that the boundary condition
\begin{equation}
\label{eq:shypPDE_Lop_leftBC}
\int_a^c \int_y^c {dx \over p(x)} \, r(y) dy < \infty
\end{equation}
(where $c \in (a,b)$ is an arbitrary point) is satisfied.

Some important properties of the solutions of the Sturm-Liouville equation $\ell(u) = \lambda u$ ($\lambda \in \mathbb{C}$) are given in the following three lemmas. The notation $u^{[1]} := p u'$ is used in the sequel.

\begin{lemma} \label{lem:entrpf_ode_wsol}
For each $\lambda \in \mathbb{C}$, there exists a unique solution $w_\lambda(\cdot)$ of the boundary value problem
\begin{equation} \label{eq:entrpf_ode_wsol}
\ell(w) = \lambda w \quad (a < x < b), \qquad\;\; w(a) = 1, \qquad\;\; w^{[1]}(a) = 0.
\end{equation}
Moreover, $\lambda \mapsto w_\lambda(x)$ is, for each fixed $x$, an entire function of exponential type.
\end{lemma}

\begin{proof}
(Adapted from \cite[Lemma 3]{kac1967}.) Let
\[
\eta_0(x) = 1, \qquad \eta_j(x) = \int_a^x \bigl(\mathfrak{s}(x) - \mathfrak{s}(\xi)\bigr) \eta_{j-1}(\xi) r(\xi) d\xi \quad (j=1,2,\ldots).
\]
Pick an arbitrary $\beta \in (a,b)$ and define $\mathcal{S}(x) = \int_a^x \bigl(\mathfrak{s}(\beta) - \mathfrak{s}(\xi)\bigr) r(\xi) d\xi$, where $\mathfrak{s}(x) := \int_c^x {d\xi \over p(\xi)}$. From the boundary assumption \eqref{eq:shypPDE_Lop_leftBC} it follows that $0 \leq \mathcal{S}(x) \leq \mathcal{S}(\beta) < \infty$ for $x \in (a,\beta]$. Furthermore, it is easy to show (using induction) that $|\eta_j(x)| \leq {1 \over j!} (\mathcal{S}(x))^j$ for all $j$. Therefore, the function
\[
w_\lambda(x) = \sum_{j=0}^\infty (-\lambda)^j \eta_j(x) \qquad (a < x \leq \beta, \; \lambda \in \mathbb{C})
\]
is well-defined as an absolutely convergent series. The estimate
\[
|w_\lambda(x)| \leq \sum_{j=0}^\infty |\lambda|^j {(\mathcal{S}(x))^j \over j!} = e^{|\lambda| \mathcal{S}(x)} \leq e^{|\lambda| \mathcal{S}(\beta)} \qquad (a < x \leq \beta)
\]
shows that $\lambda \mapsto w_\lambda(x)$ is entire and of exponential type. In addition, for $a < x \leq \beta$ we have
\begin{align*}
1 - \lambda \int_a^x {1 \over p(y)} \int_a^y w_\lambda(\xi) \, r(\xi) d\xi\, dy & = 1 - \lambda \int_a^x (\mathfrak{s}(x) - \mathfrak{s}(\xi)) w_\lambda(\xi)\, r(\xi) d\xi \\
& = 1-\lambda \int_a^x (\mathfrak{s}(x) - \mathfrak{s}(\xi)) \biggl( \sum_{j=0}^\infty (-\lambda)^j \eta_j(\xi) \biggr) r(\xi) d\xi \\
& = 1 + \sum_{j=0}^\infty (-\lambda)^{j+1} \int_a^x (\mathfrak{s}(x) - \mathfrak{s}(\xi)) \eta_j(\xi)\, r(\xi) d\xi \\
& = 1 + \sum_{j=0}^\infty (-\lambda)^{j+1} \eta_{j+1}(x) \, = \, w_\lambda(x),
\end{align*}
i.e., $w_\lambda(x)$ satisfies
\begin{equation} \label{eq:shypPDE_wsol_inteq}
w_\lambda(x) = 1 - \lambda \int_a^x {1 \over p(y)} \int_a^y w_\lambda(\xi) \, r(\xi)d\xi\, dy
\end{equation}
This integral equation is equivalent to \eqref{eq:entrpf_ode_wsol}, so the proof is complete.
\end{proof}

We note also that the following converse of Lemma \ref{lem:entrpf_ode_wsol} holds: \emph{if $\int_a^c \int_y^c {dx \over p(x)} \, r(y) dy = \infty$ (so that \eqref{eq:shypPDE_Lop_leftBC} fails to hold) then for $\lambda < 0$ there exists no solution of $\ell(w) = \lambda w$ satisfying the boundary conditions $w(a) = 1$ and $w^{[1]}(a) = 0$.}

Indeed, if the integral $\int_a^c \int_y^c {dx \over p(x)} \, r(y) dy$ diverges (which means, according to the Feller boundary classification given in \cite[Section 5.11]{ito2006}, that $a$ is an exit boundary or a natural boundary for the operator $\ell$) then it follows from \cite[Sections 5.13--5.14]{ito2006} that any solution $w$ of $\ell(w) = \lambda w$\, ($\lambda < 0$) either satisfies $w(a) = 0$ or $w^{[1]}(a) = +\infty$, so in particular \eqref{eq:entrpf_ode_wsol} cannot hold.

\begin{lemma} \label{lem:entrpf_ode_wepslimit}
Let $\{a_m\}_{m \in \mathbb{N}}$ be a sequence $b > a_1 > a_2 > \ldots$ with $\lim a_m = a$. For $m \in \mathbb{N}$ and $\lambda \in \mathbb{C}$, let $w_{\lambda,m}(x)$ be the unique solution of the boundary value problem
\begin{equation} \label{eq:entrpf_ode_wsoleps}
\ell(w) = \lambda w \quad (a_m < x < b), \qquad\;\; w(a_m) = 1, \qquad\;\; w^{[1]}(a_m) = 0.
\end{equation}
Then
\begin{equation} \label{eq:shypPDE_ode_wepslimit}
\lim_{m \to \infty} w_{\lambda,m}(x) = w_\lambda(x) \quad\;\: \text{and} \quad\;\; \lim_{m \to \infty} w_{\lambda,m}^{[1]}(x) = w_\lambda^{[1]}(x)
\end{equation}
pointwise for each $a < x < b$ and $\lambda \in \mathbb{C}$.
\end{lemma}

\begin{proof}
In the same way as in the proof of Lemma \ref{lem:entrpf_ode_wsol} we can check that the solution of \eqref{eq:entrpf_ode_wsoleps} is given by
\[
w_{\lambda,m}(x) = \sum_{j=0}^\infty (-\lambda)^j \eta_{j,m}(x) \qquad (a_m < x < b, \; \lambda \in \mathbb{C})
\]
where $\eta_{0,m}(x) = 1$ and $\eta_{j,m}(x) = \int_{a_m}^x \bigl(\mathfrak{s}(x) - \mathfrak{s}(\xi)\bigr) \eta_{j-1,m}(\xi) r(\xi) d\xi$. As before we have $|\eta_{j,m}(x)| \leq {1 \over j!} (\mathcal{S}(x))^j$ for $a_m < x \leq \beta$ (where $\mathcal{S}$ is the function from the proof of Lemma \ref{lem:entrpf_ode_wsol}). Using this estimate and induction on $j$, it is easy to see that $\eta_{j,m}(x) \to \eta_j(x)$ as $m \to \infty$\, ($a < x \leq \beta$, $j=0,1,\ldots$). Noting that the estimate on $|\eta_{j,m}(x)|$ allows us to take the limit under the summation sign, we conclude that $w_{\lambda,m}(x) \to w_\lambda(x)$ as $m \to \infty$ ($a < x \leq \beta$). Finally, by \eqref{eq:shypPDE_wsol_inteq} we have 
\[
\lim_{m \to \infty} w_{\lambda,m}^{[1]}(x) = - \lambda\lim_{m \to \infty} \int_{a_m}^x w_{\lambda,m}(\xi) \, r(\xi) d\xi = -\lambda \int_a^x w_\lambda(\xi) \, r(\xi) d\xi = w_\lambda^{[1]}(x) \qquad (a < x \leq \beta). \qedhere
\]
using dominated convergence and the estimates $|w_{\lambda,m}(x)| \leq e^{|\lambda| \mathcal{S}(\beta)}$,\, $|w_\lambda(x)| \leq e^{|\lambda| \mathcal{S}(\beta)}$.
\end{proof}

\begin{lemma} \label{lem:shypPDE_wsolbound}
If $x \mapsto p(x)r(x)$ is an increasing function, then the solution of \eqref{eq:entrpf_ode_wsol} is bounded:
\begin{equation} \label{eq:shypPDE_wsolbound}
|w_\lambda(x)| \leq 1 \qquad \text{for all } \, a < x < b, \; \lambda \geq 0.
\end{equation}
\end{lemma}

\begin{proof}
(Adapted from \cite[Proposition 4.3]{zeuner1992}.) Let us start by assuming that $p(a)r(a) > 0$. For $\lambda = 0$ the result is trivial because $w_0(x) \equiv 1$. Fix $\lambda > 0$. Multiplying both sides of the differential equation $\ell(w_\lambda) = \lambda w_\lambda$ by $2w_\lambda^{[1]}$, we obtain $-{1 \over pr} [(w_\lambda^{[1]})^2]' = \lambda (w_\lambda^2)'$. Integrating the differential equation and then using integration by parts, we get
\begin{align*}
\lambda\bigl(1-w_\lambda(x)^2\bigr) & = \int_a^x {1 \over p(\xi) r(\xi)} \bigl(w_\lambda^{[1]}(\xi)^2\bigr)' d\xi \\
& = {w_\lambda^{[1]}(x)^2 \over p(x) r(x)} + \int_a^x \bigl(p(\xi)r(\xi)\bigr)' \biggl({w_\lambda^{[1]}(\xi) \over p(\xi)r(\xi)}\biggr)^{\!2} d\xi, \qquad a < x < b
\end{align*}
where we also used the fact that $w_\lambda^{[1]}(a) = 0$ and the assumption that $p(a)r(a) > 0$. The right hand side is nonnegative, because $x \mapsto p(x) r(x)$ is increasing and therefore $(p(\xi)r(\xi))' \geq 0$. Given that $\lambda > 0$, it follows that $1 - w_\lambda(x)^2 \geq 0$, so that $|w_\lambda(x)| \leq 1$.

If $p(a)r(a) = 0$, the above proof can be used to show that the solution of \eqref{eq:entrpf_ode_wsoleps} is such that $|w_{\lambda,m}(x)| \leq 1$ for all $a < x < b$, $\lambda \geq 0$ and $m \in \mathbb{N}$; then Lemma \ref{lem:entrpf_ode_wepslimit} yields the desired result.
\end{proof}

\begin{remark} \label{rmk:shypPDE_tildeell}
We shall make extensive use of the fact that the differential expression \eqref{eq:shypPDE_elldiffexpr} can be transformed into the standard form
\[
\widetilde{\ell} = - {1 \over A} {d \over d\xi} \Bigl(A {d \over d\xi} \Bigr) = -{d^2 \over d\xi^2} - {A' \over A} {d \over d\xi}.
\]
This is achieved by setting 
\begin{equation} \label{eq:shypPDE_tildeell_A}
A(\xi) := \sqrt{p(\gamma^{-1}(\xi)) \, r(\gamma^{-1}(\xi))},
\end{equation}
where $\gamma^{-1}$ is the inverse of the increasing function 
\[
\gamma(x) = \int_c^x\! \smash{\sqrt{r(y) \over p(y)}} dy,
\]
$c \in (a,b)$ being a fixed point (if $\smash{\sqrt{r(y) \over p(y)}}$ is integrable near $a$, we may also take $c=a$). Indeed, it is straightforward to check that a given function $\omega_\lambda: (a,b) \to \mathbb{C}$ satisfies $\ell(\omega_\lambda) = \lambda \omega_\lambda$ if and only if $\widetilde{\omega}_\lambda(\xi) := \omega_\lambda(\gamma^{-1}(\xi))$ satisfies $\widetilde{\ell}(\widetilde{\omega}_\lambda) = \lambda \widetilde{\omega}_\lambda$. It is interesting to note that the assumption of the previous lemma ($x \mapsto p(x) r(x)$ is increasing) is equivalent to requiring that the first-order coefficient $A' \over A$ of the transformed operator $\widetilde{\ell}$ is nonnegative.
\end{remark}

As it is well-known, the spectral expansion of self-adjoint realizations of the differential operator \eqref{eq:shypPDE_elldiffexpr} in the space $L_2\bigl((a,b); r(x)dx\bigr)$ give rise to a Sturm-Liouville type integral transform. The next proposition collects some basic facts from the theory of eigenfunction expansions of Sturm-Liouville operators. For brevity we write $L_p(r) := L_p\bigl((a,b); r(x)dx\bigr)$ and $\|\cdot\|_p := \|\cdot\|_{L_p(r)}$.

\begin{proposition} \label{prop:entrpf_Ffourier}
Suppose that $b$ is (according to Feller's boundary classification) a natural boundary for the differential expression $\ell$, that is, the coefficients of $\ell$ satisfy
\[
\int_c^b \int_y^b {dx \over p(x)} \, r(y) dy = \int_c^b \int_c^y {dx \over p(x)} \, r(y) dy = \infty.
\]
Then the operator
\[
\mathcal{L}: \mathcal{D}_\mathcal{L}^{(2)} \subset L_2(r) \longrightarrow L_2(r), \qquad\quad \mathcal{L} u = \ell(u)
\]
where
\begin{equation} \label{eq:shypPDE_Lop_L2domain}
\mathcal{D}_\mathcal{L}^{(2)} := \Bigl\{ u \in L_2(r) \Bigm| u \text{ and } u' \text{ locally abs.\ continuous on } (a,b), \; \ell(u) \in L_2(r), \; \lim_{x \downarrow a} u^{[1]}(x) = 0 \Bigr\}
\end{equation}
is self-adjoint. There exists a unique locally finite positive Borel measure $\bm{\rho}_\mathcal{L}$ on $\mathbb{R}$ such that the map $h \mapsto \mathcal{F} h$, where
\begin{equation} \label{eq:entrpf_Ffourierdef}
(\mathcal{F} h)(\lambda) := \int_a^b h(x) \, w_\lambda(x) \, r(x) dx \qquad \bigl(h \in \mathrm{C}_\mathrm{c}[a,b), \; \lambda \geq 0\bigr),
\end{equation}
induces an isometric isomorphism $\mathcal{F}: L_2(r) \longrightarrow L_2(\mathbb{R}; \bm{\rho}_\mathcal{L})$ whose inverse is given by
\begin{equation} \label{eq:entrpf_Ffourierinverse}
(\mathcal{F}^{-1} \varphi)(x) = \int_\mathbb{R} \varphi(\lambda) \, w_\lambda(x) \, \bm{\rho}_\mathcal{L}(d\lambda),
\end{equation}
the convergence of the latter integral being understood with respect to the norm of $L_2(r)$. The spectral measure $\bm{\rho}_\mathcal{L}$ is supported on $[0,\infty)$. Moreover, the differential operator $\mathcal{L}$ is connected with the integral transform \eqref{eq:entrpf_Ffourierdef} via the identity
\begin{equation} \label{eq:shypPDE_Ltransfidentity}
[\mathcal{F} (\mathcal{L} h)] (\lambda) = \lambda \ccdot (\mathcal{F} h)(\lambda), \qquad h \in \mathcal{D}_\mathcal{L}^{(2)}
\end{equation}
and the domain $\mathcal{D}_\mathcal{L}^{(2)}$ defined by \eqref{eq:shypPDE_Lop_L2domain} can be written as
\begin{equation} \label{eq:shypPDE_LtransfidentD2}
\mathcal{D}_\mathcal{L}^{(2)} = \Bigl\{ u \in L_2(r) \Bigm| \lambda \ccdot (\mathcal{F} f)(\lambda) \in L_2\bigl([0,\infty); \bm{\rho}_\mathcal{L}\bigr) \Bigr\}.
\end{equation}
\end{proposition}

\begin{proof}
The fact that $(\mathcal{L}, \mathcal{D}_\mathcal{L}^{(2)})$ is self-adjoint is well-known, see \cite{mckean1956,linetsky2004}. The existence of a generalized Fourier transform associated with the operator $\mathcal{L}$ is a consequence of the standard Weyl-Titchmarsh-Kodaira theory of eigenfunction expansions of Sturm-Liouville operators (cf.\ \cite[Section 3.1]{sousayakubovich2018} and \cite[Section 8]{weidmann1987}).

In the general case the eigenfunction expansion is written in terms of two linearly independent eigenfunctions and a $2 \times 2$ matrix measure. However, from the regular/entrance boundary assumption \eqref{eq:shypPDE_Lop_leftBC} it follows that the function $w_\lambda(x)$ is square-integrable near $x = 0$ with respect to the measure $r(x)dx$; moreover, by Lemma \ref{lem:entrpf_ode_wsol}, $w_\lambda(x)$ is (for fixed $x$) an entire function of $\lambda$. Therefore, the possibility of writing the expansion in terms only of the eigenfunction $w_\lambda(x)$ follows from the results of \cite[Sections 9 and 10]{eckhardt2013}.
\end{proof}

The integral transform $(\mathcal{F} h)(\lambda) = \int_a^b h(x) \, w_\lambda(x) \, r(x) dx$ is the so-called \emph{$\mathcal{L}$-transform}. It is often important to know whether the inversion integral for the $\mathcal{L}$-transform is absolutely convergent. A sufficient condition is provided by the following lemma:

\begin{lemma} \label{lem:shypPDE_Lfourier_D2prop}
\textbf{(a)} For each $\mu \in \mathbb{C} \setminus\mathbb{R}$, the integrals
\begin{equation} \label{eq:shypPDE_Lresolv_specunif}
\int_{[0,\infty)} {w_\lambda(x)\, w_\lambda(y) \over |\lambda - \mu|^2} \bm{\rho}_\mathcal{L}(d\lambda) \qquad\; \text{and} \qquad\; \int_{[0,\infty)} {w_\lambda^{[1]}(x)\, w_\lambda^{[1]}(y) \over |\lambda - \mu|^2} \bm{\rho}_\mathcal{L}(d\lambda)
\end{equation}
converge uniformly on compact squares in $(a,b)^2$. \\[-8pt]

\textbf{(b)} If $h \in \mathcal{D}_{\mathcal{L}}^{(2)}$, then 
\begin{align} 
\label{eq:shypPDE_Lfourier_D2prop}
h(x) & = \int_{[0,\infty)}\! (\mathcal{F}h)(\lambda) \, w_\lambda(x) \, \bm{\rho}_\mathcal{L}(d\lambda)\\
\label{eq:shypPDE_Lfourier_D2propderiv}
h^{[1]}(x) & = \int_{[0,\infty)}\! (\mathcal{F}h)(\lambda) \, w_\lambda^{[1]}(x) \, \bm{\rho}_\mathcal{L}(d\lambda)
\end{align}
where the right-hand side integrals converge absolutely and uniformly on compact subsets of $(a,b)$.
\end{lemma}

\begin{proof}
\textbf{(a)} By \cite[Lemma 10.6]{eckhardt2013} and \cite[p.\ 229]{teschl2014},
\[
\int_{[0,\infty)} {w_\lambda(x) w_\lambda(y) \over |\lambda - \mu|^2} \bm{\rho}_\mathcal{L}(d\lambda) = \int_a^b G(x,\xi,\mu)G(y,\xi,\mu)\, r(\xi) d\xi = {1 \over \mathrm{Im}(\mu)}\, \mathrm{Im}\bigl(G(x,y,\mu)\bigr)
\]
where $G(x,y,\mu)$ is the resolvent kernel (or Green function) of the operator $(\mathcal{L}, \mathcal{D}_\mathcal{L}^{(2)})$. Moreover, according to \cite[Theorems 8.3 and 9.6]{eckhardt2013}, the resolvent kernel is given by
\[
G(x,y,\mu) = \begin{cases}
w_\mu(x) \bm{u}_\mu(y), & x < y \\
w_\mu(y) \bm{u}_\mu(x), & x \geq y
\end{cases}
\]
where $\bm{u}_\lambda(\cdot)$ is a solution of $\ell(u) = \lambda u$ which is square-integrable near $\infty$ with respect to the measure $r(x)dx$ and verifies the identity $w_\lambda(x) \bm{u}_\lambda^{[1]}(x) - w_\lambda^{[1]}(x) \bm{u}_\lambda(x) \equiv 1$. It is easily seen (cf.\ \cite[p.\ 125]{naimark1968}) that the functions $\mathrm{Im}\bigl(G(x,y,\mu)\bigr)$ and $\partial_x^{[1]} \partial_y^{[1]} \mathrm{Im}\bigl(G(x,y,\mu)\bigr)$ are continuous in $0 < x,y < \infty$. Essentially the same proof as that of \cite[Corollary 3]{naimark1968} now yields that
\[
\int_{[0,\infty)} {w_\lambda^{[1]}(x) \, w_\lambda^{[1]}(y) \over |\lambda - \mu|^2} \bm{\rho}_\mathcal{L}(d\lambda) = {1 \over \mathrm{Im}(\mu)}\, \partial_x^{[1]} \partial_y^{[1]} \mathrm{Im}\bigl(G(x,y,\mu)\bigr)
\]
and that the integrals \eqref{eq:shypPDE_Lresolv_specunif} converge uniformly for $x,y$ in compacts.
\\[-8pt]

\textbf{(b)} By Proposition \ref{prop:entrpf_Ffourier} and the classical theorem on differentiation under the integral sign for Riemann-Stieltjes integrals, to prove \eqref{eq:shypPDE_Lfourier_D2prop}--\eqref{eq:shypPDE_Lfourier_D2propderiv} it only remains to justify the absolute and uniform convergence of the integrals in the right-hand sides.

Recall from Proposition \ref{prop:entrpf_Ffourier} that the condition $h \in \mathcal{D}_\mathcal{L}^{(2)}$ implies that $\mathcal{F}h \in L_2\bigl([0,\infty); \bm{\rho}_\mathcal{L}\bigr)$ and also $\lambda \,(\mathcal{F}h)(\lambda) \in L_2\bigl([0,\infty); \bm{\rho}_\mathcal{L}\bigr)$. As a consequence, we obtain
\begin{align*}
& \int_{[0,\infty)} \bigl|(\mathcal{F}h)(\lambda) w_\lambda(x)\bigr| \bm{\rho}_\mathcal{L}(d\lambda) \\
& \qquad\qquad \leq \!\int_{[0,\infty)} \! \lambda \, \bigl|(\mathcal{F}h)(\lambda)\bigr| \biggl|{w_\lambda(x) \over \lambda + i}\biggr| \bm{\rho}_\mathcal{L}(d\lambda) + \! \int_{[0,\infty)} \bigl|(\mathcal{F}h)(\lambda)\bigr| \biggl| {w_\lambda(x) \over \lambda + i} \biggr| \bm{\rho}_\mathcal{L}(d\lambda) \\
& \qquad\qquad \leq \bigl(\|\lambda \, (\mathcal{F}h)(\lambda)\|_\rho + \|(\mathcal{F}h)(\lambda)\|_\rho\bigr) \biggl\| {w_\lambda(x) \over \lambda + i} \biggr\|_\rho \\
& \qquad\qquad < \infty
\end{align*}
where $\| \cdot \|_\rho$ denotes the norm of the space $L_2\bigl(\mathbb{R}; \bm{\rho}_\mathcal{L}\bigr)$, and similarly
\[
\int_{[0,\infty)}\! \bigl|(\mathcal{F}h)(\lambda) \, w_\lambda^{[1]}(x) \bigr| \bm{\rho}_\mathcal{L}(d\lambda) \leq \bigl(\|\lambda \, (\mathcal{F}h)(\lambda)\|_\rho + \|(\mathcal{F}h)(\lambda)\|_\rho\bigr) \biggl\| {w_\lambda^{[1]}(x) \over \lambda + i} \biggr\|_\rho < \infty.
\]
We know from part (a) that the integrals which define $\bigl\| {w_\lambda(x) \over \lambda + i} \bigr\|_\rho$ and $\bigl\| {w_\lambda^{[1]}(x) \over \lambda + i} \bigr\|_\rho$ converge uniformly, hence the integrals in \eqref{eq:shypPDE_Lfourier_D2prop}--\eqref{eq:shypPDE_Lfourier_D2propderiv} converge absolutely and uniformly for $x$ in compact subsets.
\end{proof}

It is also useful to know that, according to a standard result from the theory of diffusion processes and semigroups which we state below, Sturm-Liouville differential expressions of the form \eqref{eq:shypPDE_elldiffexpr} generate positivity-preserving contraction semigroups acting on the space of bounded continuous functions. We recall from \cite{bottcher2013} that, for a subset $E \subset \mathbb{R}^d$, a \emph{Feller semigroup} $\{T_t\}_{t \geq 0}$ on $\mathrm{C}_0(E,\mathbb{R})$ is a strongly continuous, positivity-preserving contraction semigroup on $\mathrm{C}_0(E,\mathbb{R})$, and that a Feller semigroup is \emph{conservative} if its extension to $\mathrm{B}_\mathrm{b}(E)$ satisfies $T_t \mathds{1} = \mathds{1}$ (here $\mathds{1}$ denotes the function identically equal to one).

\begin{proposition} \label{prop:shypPDE_L_fellergen_tpdf}
Suppose that $b$ is a natural boundary for the differential expression $\ell$. Then the operator
\[
\mathcal{L}^{(0)}: \mathcal{D}_{\mathcal{L}}^{(0)} \cap \mathrm{C}_0([a,b),\mathbb{R}) \longrightarrow \mathrm{C}_0([a,b),\mathbb{R}), \qquad\quad \mathcal{L}^{(0)} u = -\ell(u)
\]
where
\[
\mathcal{D}_\mathcal{L}^{(0)} = \bigl\{ u \in \mathrm{C}_0[a,b) \bigm| \ell(u) \in \mathrm{C}_0[a,b),\; \lim_{x \downarrow a} u^{[1]}(x) = 0 \bigr\}
\]
is the generator of a conservative Feller semigroup $\{T_t\}_{t \geq 0}$ on $\mathrm{C}_0([a,b),\mathbb{R})$. The semigroup admits the representation
\begin{equation} \label{eq:shypPDE_fellersgp_fundsolrep}
(T_t h)(x) = \int_a^b h(y)\, p(t,x,y)\, r(y)dy, \qquad \bigl(h \in \mathrm{B}_\mathrm{b}\bigl([a,b),\mathbb{R}\bigr), \; t > 0, \; x \in (a,b)\bigr)
\end{equation}
where the (nonnegative) transition kernel $p(t,x,y)$ is given by
\begin{equation} \label{eq:shypPDE_fellersgp_fundsol_spectr}
p(t,x,y) = \int_{[0,\infty)} e^{-t\lambda} \, w_\lambda(x) \, w_\lambda(y)\, \bm{\rho}_\mathcal{L}(d\lambda) \qquad \bigl(t > 0, \; x,y \in (a,b)\bigr)
\end{equation}
with the integral converging absolutely and uniformly on compact squares of $(a,b) \times (a,b)$ for each fixed $t > 0$. If $h \in L_2(r) \cap \mathrm{B}_\mathrm{b}\bigl([a,b),\mathbb{R}\bigr)$, then \eqref{eq:shypPDE_fellersgp_fundsolrep} can also be written as
\begin{equation} \label{eq:shypPDE_fellersgp_fundsolL2rep}
(T_t h)(x) = \int_{[0,\infty)} e^{-t\lambda} w_\lambda(x) \, (\mathcal{F} h)(\lambda)\, \bm{\rho}_\mathcal{L}(d\lambda) \qquad \bigl(t > 0, \; x \in (a,b)\bigr)
\end{equation}
where the integral converges with respect to the norm of $L_2(r)$. 
\end{proposition}

\begin{proof}
The first assertion is proved in \cite[Sections 4 and 6]{fukushima2014} (see also \cite[Section II.5]{mandl1968}). The claimed representation for the transition semigroup and kernel follows from \cite[Sections 2--3]{linetsky2004}.
\end{proof}

\section{Laplace-type representation} \label{sec:laplacerep}

As mentioned in the introduction, the existence of a hypergroup-like product formula for the kernel of the $\mathcal{L}$-transform is strongly connected with the positivity of the associated Cauchy problem. We now introduce an assumption which, as we will see in Subsection \ref{sub:prodform_hypCP}, is sufficient for the Cauchy problem to be positivity preserving. Recall that the function $A$, defined in \eqref{eq:shypPDE_tildeell_A}, is the coefficient associated with the transformation of $\ell$ into the standard form (Remark \ref{rmk:shypPDE_tildeell}).

\renewcommand\theassumption{MP}
\begin{assumption} \label{asmp:shypPDE_SLhyperg}
We have $\gamma(b) = \int_c^b\! \smash{\sqrt{r(y) \over p(y)}} dy = \infty$, and there exists $\eta \in \mathrm{C}^1(\gamma(a),\infty)$ such that $\eta \geq 0$, the functions $\bm{\phi}_\eta := {A' \over A} - \eta$, $\:\bm{\psi}_\eta := {1 \over 2} \eta' - {1 \over 4} \eta^2 + {A' \over 2A} \ccdot \eta$ are both decreasing on $(\gamma(a),\infty)$ and $\bm{\phi}_\eta$ satisfies $\lim_{\xi \to \infty} \bm{\phi}_\eta(\xi) = 0$.
\end{assumption}

This assumption will be held throughout the remainder of the paper. 

Having in mind the product formula that we shall establish for a general Sturm-Liouville operator satisfying Assumption \ref{asmp:shypPDE_SLhyperg}, in this section we prove the related fact that the kernel $w_\lambda(x)$ of the $\mathcal{L}$-transform admits a representation as the Laplace transform of a subprobability measure. We start by stating a basic property which holds for all Sturm-Liouville operators \eqref{eq:shypPDE_elldiffexpr} satisfying this assumption:

\begin{lemma}
The function $A' \over A$ is nonnegative, and there exists a finite limit $\sigma := \lim_{\xi \to \infty} {A'(\xi) \over 2 A(\xi)} \in [0,\infty)$.
\end{lemma}

\begin{proof}
See \cite[Section 2]{zeuner1992}.
\end{proof}

The existence of a Laplace-type representation for the kernel of the $\mathcal{L}$-transform is already known to hold for a Sturm-Liouville operator of the form $- {1 \over A} {d \over dx} \bigl(A {d \over dx}\bigr)$ where the coefficient $A$ satisfies the assumptions of the existence theorem of \cite{zeuner1992} for Sturm-Liouville hypergroups (see the discussion in Subsection \ref{sub:SLhyp_nondegen}). In particular, the following result is proved in \cite[Theorem 3.5.58]{bloomheyer1994}:

\begin{proposition} \label{prop:shypPDE_Wlaplacerep_hyperg}
Let $A \in \mathrm{C}^1[0,\infty)$ with $A(x) > 0$ for all $x \geq 0$. Suppose that there exists $\eta \in \mathrm{C}^1[0,\infty)$ such that $\eta \geq 0$, the functions $\bm{\phi}_\eta$, $\bm{\psi}_\eta$ are both decreasing on $(0,\infty)$ and $\lim_{x \to \infty} \bm{\phi}_\eta(x) = 0$ ($\bm{\phi}_\eta$, $\bm{\psi}_\eta$ are defined as in Assumption \ref{asmp:shypPDE_SLhyperg}). For $\lambda \in \mathbb{C}$, let $\bm{\omega}_\lambda(\cdot)$ be the unique solution of the boundary value problem
\[
-{1 \over A} (A \bm{\omega}')' = \lambda \bm{\omega} \quad (0 < x < \infty), \qquad\;\; \bm{\omega}(0) = 1, \qquad\;\; \bm{\omega}'(0) = 0.
\]
Then for each $x \geq 0$ there exists a subprobability measure $\pi_x$ on $\mathbb{R}$ such that
\[
\bm{\omega}_{\tau^2 + \sigma^2}(x) = \int_\mathbb{R} e^{i\tau s} \pi_x(ds) = \int_\mathbb{R} \cos(\tau s) \, \pi_x(ds) \qquad (\tau \in \mathbb{C})
\]
where $\sigma = \lim_{\xi \to \infty} {A'(\xi) \over 2 A(\xi)}$.
\end{proposition}

The following theorem generalizes the proposition above to the class of operators $\ell$ of the form \eqref{eq:shypPDE_elldiffexpr} satisfying \eqref{eq:shypPDE_Lop_leftBC} and Assumption \ref{asmp:shypPDE_SLhyperg}:

\begin{theorem}[Laplace-type representation] \label{thm:shypPDE_Wlaplacerep}
For each $x \in [a,b)$ there exists a subprobability measure $\nu_x$ on $\mathbb{R}$ such that
\begin{equation} \label{eq:shypPDE_Wlaplacerep}
w_{\tau^2 + \sigma^2}(x) = \int_\mathbb{R} e^{i\tau s} \nu_x(ds) = \int_\mathbb{R} \cos(\tau s) \, \nu_x(ds) \qquad (\tau \in \mathbb{C})
\end{equation}
where $\sigma = \lim_{\xi \to \infty} {A'(\xi) \over 2 A(\xi)}$. In particular, the boundedness property \eqref{eq:shypPDE_wsolbound} extends to
\begin{equation} \label{eq:shypPDE_Wbound}
|w_{\tau^2 + \sigma^2}(x)| \leq 1 \quad\; \text{on the strip } \, |\mathrm{Im}(\tau)| \leq \sigma \quad\; (a \leq x < b).
\end{equation}
\end{theorem}

\begin{proof}
For $m \in \mathbb{N}$ and $\lambda \in \mathbb{C}$, let $w_{\lambda,m}$ be the solution of \eqref{eq:entrpf_ode_wsoleps}. The function $\widetilde{w}_{\lambda,m}(\xi) = w_{\lambda,m}(\gamma^{-1}(\xi))$ is the solution of
\[
\widetilde{\ell}(u) = \lambda u \quad\; (\tilde{a}_m < \xi < \infty), \qquad\;\; u(\tilde{a}_m) = 1, \qquad\;\; u^{[1]}(\tilde{a}_m) = 0
\]
where $\tilde{a}_m = \gamma(a_m)$. By Assumption \ref{asmp:shypPDE_SLhyperg}, the function $\bm{A}(y) := A(y+\tilde{a}_m)$ satisfies the assumption of Proposition \ref{prop:shypPDE_Wlaplacerep_hyperg}. It follows that for each $\xi > \tilde{a}_m$ there exists a subprobability measure $\pi_{\xi,m}$ such that
\[
\widetilde{w}_{\tau^2 + \sigma^2,m}(\xi) = \int_\mathbb{R} e^{i \tau s} \pi_{\xi,m}(ds) = \int_\mathbb{R} \cos(\tau s) \, \pi_{\xi,m}(ds) \qquad (\tau \in \mathbb{C}).
\]
In particular, $\tau \mapsto \widetilde{w}_{\tau^2 + \sigma^2,m}(\xi)$ ($\tau \in \mathbb{R}$) is the Fourier transform of the measure $\pi_{\xi,m}$. We know (from Lemma \ref{lem:entrpf_ode_wepslimit}) that $\widetilde{w}_{\tau^2 + \sigma^2,m}(\xi) \longrightarrow \widetilde{w}_{\tau^2 + \sigma^2}(\xi) := w_{\tau^2 + \sigma^2}(\gamma^{-1}(\xi))$ pointwise as $m \to \infty$, the limit function being continuous in $\tau$ (cf.\ Lemma \ref{lem:entrpf_ode_wsol}). Applying the Lévy continuity theorem \cite[Theorem 23.8]{bauer1996}, we conclude that $\widetilde{w}_{\tau^2 + \sigma^2}(\xi)$ is the Fourier transform of a subprobability measure $\pi_\xi$ and, in addition, the measures $\pi_{\xi,m}$ converge weakly to $\pi_\xi$ as $m \to \infty$. Therefore, for $\xi > \gamma(a)$ we have
\begin{equation} \label{eq:shypPDE_Wlaplacerep_pf1}
\widetilde{w}_{\tau^2 + \sigma^2}(\xi) = \int_\mathbb{R} e^{i \tau s} \pi_{\xi}(ds) = \int_\mathbb{R} \cos(\tau s) \, \pi_{\xi}(ds) \qquad ( \tau \in \mathbb{R}).
\end{equation}

In order to extend \eqref{eq:shypPDE_Wlaplacerep_pf1} to $\tau \in \mathbb{C}$, we let $0 \leq \phi_1 \leq \phi_2 \leq \ldots$ be functions with compact support such that $\phi_n \uparrow 1$ pointwise, and for fixed $\xi > \gamma(a)$, $\kappa > 0$ we compute
\begin{align*}
\int_\mathbb{R} \cosh(\kappa s) \, \pi_\xi(ds) & = \lim_{n \to \infty} \int_\mathbb{R} \phi_n(s) \cosh(\kappa s) \, \pi_\xi(ds) \\
& = \lim_{n \to \infty} \lim_{m \to \infty} \int_\mathbb{R} \phi_n(s) \cosh(\kappa s) \, \pi_{\xi,m}(ds) \\
& \leq \lim_{m \to \infty} \int_\mathbb{R} \cosh(\kappa s) \, \pi_{\xi,m}(ds) \: = \: \lim_{m \to \infty} \widetilde{w}_{\sigma^2-\kappa^2,m}(\xi) \: = \: \widetilde{w}_{\sigma^2-\kappa^2}(\xi) < \infty
\end{align*}
From this estimate we easily see that the right-hand side of \eqref{eq:shypPDE_Wlaplacerep_pf1} is an entire function of $\tau$; therefore, by analytic continuation, \eqref{eq:shypPDE_Wlaplacerep_pf1} holds for all $\tau \in \mathbb{C}$. Setting $\nu_x = \pi_{\gamma(x)}$ gives \eqref{eq:shypPDE_Wlaplacerep}.

Finally, if $|\mathrm{Im}(\tau)| \leq \sigma$ then
\[
|w_{\tau^2 + \sigma^2}(x)| \leq \int_\mathbb{R} |\cos(\tau s)| \nu_{x}(ds) \leq \int_\mathbb{R} \cosh(\sigma s) \, \nu_{x}(ds) = w_0(x) = 1
\]
and therefore \eqref{eq:shypPDE_Wbound} is true.
\end{proof}

The rest of this section provides some additional properties of the solutions of $\ell(u) = \lambda u$ which will be needed later.

\begin{proposition}
If $\lambda > \sigma^2$, then the equation $\ell(u) = \lambda u$ is oscillatory at $b$, that is, all solutions of $\ell(u) = \lambda u$ have infinitely many zeros clustering at $b$. Consequently, $b$ is a natural boundary for $\ell$.
\end{proposition}

\begin{proof}
The results of \cite[Lemma 3.7]{fruchtl2018} on the asymptotic behavior of the solutions of the standardized equation $\widetilde{\ell}(u) = (\tau^2 + \sigma^2) u$ show that for $\tau > 0$ this equation has a linearly independent pair of solutions with infinitely many zeros clustering at infinity; hence any solution of $\widetilde{\ell}(u) = (\tau^2 + \sigma^2) u$ has this property (cf.\ \cite[Section 8.1]{coddingtonlevinson1955}). It immediately follows that the same is true for any solution of $\ell(u) = (\tau^2 + \sigma^2) u$ ($\tau > 0$).

According to \cite[p.\ 348]{linetsky2004}, if $\widetilde{\ell}(u) = \lambda u$ is oscillatory at $b$ for some $\lambda > 0$ then $b$ is a natural (Feller) boundary for the operator $\ell$, so the final assertion holds.
\end{proof}

\begin{proposition} \label{prop:shypPDE_spectralsupp}
The spectral measure from Proposition \ref{prop:entrpf_Ffourier} is such that $\supp(\bm{\rho}_\mathcal{L}) = [\sigma^2,\infty)$. In addition, $\mathcal{L}$ has purely absolutely continuous spectrum in $(\sigma^2,\infty)$.
\end{proposition}

\begin{proof}
To show that the essential spectrum of $\mathcal{L}$ equals $[\sigma^2,\infty)$, we may assume that the differential expression $\ell$ is regular at the endpoint $a$: this is so because, by a well-known result \cite[Theorem 9.11]{teschl2014}, the essential spectrum of $\mathcal{L}$ is the union of the essential spectrums of self-adjoint realizations of $\ell$ restricted to the intervals $(a,c)$ and $(c,b)$ (where $c \in (a,b)$), and because it is known from \cite[Theorem 3.1]{mckean1956} that the spectrum is purely discrete whenever there are no natural boundaries.

The equation $\ell(u) = \lambda u$ is clearly non-oscillatory at $a$; it is oscillatory at $b$ for $\lambda > \sigma^2$ and (by the Laplace representation \eqref{eq:shypPDE_Wlaplacerep}) non-oscillatory at $b$ for $\lambda < \sigma^2$. Hence it follows from \cite[Theorem 2]{linetsky2004} that the essential spectrum of $\mathcal{L}$ is contained in $[\sigma^2, \infty)$. Now, the operator $\mathcal{L}$ is unitarily equivalent, via the Liouville transformation (see e.g.\ \cite[Section 4.3]{everitt1982} and \cite[Section 4]{linetsky2004}), to a self-adjoint realization of the differential expression $-{d^2 \over d\xi^2} + \mathfrak{q}$, where
\begin{equation} \label{eq:shypPDE_liouvtrans_frakq}
\mathfrak{q}(\xi) = \Bigl({A'(\xi) \over 2A(\xi)}\Bigr)^{\!2} + \Bigl({A'(\xi) \over 2A(\xi)}\Bigr)' = {1 \over 4} \bm{\phi}_\eta^2(\xi) + \bm{\psi}_\eta(\xi) + {1 \over 2} \bm{\phi}_\eta'(\xi), \qquad \xi \in (\gamma(a), \infty).
\end{equation}
We know from Assumption \ref{asmp:shypPDE_SLhyperg} and \cite[Lemma 2.9]{zeuner1992} that $\lim_{\xi \to \infty} \bm{\phi}_\eta(\xi) = 0$ and $\lim_{\xi \to \infty} \eta'(\xi) = 0$. Consequently, $\lim_{\xi \to \infty} {1 \over 4} \bm{\phi}_\eta^2(\xi) + \bm{\psi}_\eta(\xi) = \sigma^2$. In turn, the fact that $\bm{\phi}_\eta$ is positive and decreasing clearly implies that $\bm{\phi}_\eta' \in L_1([c,\infty),d\xi)$ for $c > \gamma(a)$. Using \cite[Theorem 15.3]{weidmann1987}, we conclude that the spectrum of $\mathcal{L}$ is purely absolutely continuous on $(\sigma^2,\infty)$ and the essential spectrum equals $[\sigma^2,\infty)$.

It remains to show that $\mathcal{L}$ has no eigenvalues on $[0,\sigma^2]$. Indeed, if we assume that $0 \leq \lambda_0 \leq \sigma^2$ is an eigenvalue of $\mathcal{L}$, then $w_{\lambda_0}$ belongs to $\mathcal{D}_{\mathcal{L}}^{(2)}$ and therefore, by the Laplace representation \eqref{eq:shypPDE_Wlaplacerep}, $w_{\lambda}$ belongs to $\mathcal{D}_{\mathcal{L}}^{(2)}$ for all $\lambda \geq \sigma^2$; since the eigenvalues are discrete, this is a contradiction.
\end{proof}

\begin{proposition} \label{prop:hypPDE_Wasym_infty}
We have
\[
\lim_{x \uparrow b} w_\lambda(x) = 0 \qquad \text{for all } \lambda > 0
\]
if and only if $\,\lim_{x \uparrow b} p(x)r(x) = \infty$.
\end{proposition}

\begin{proof}
After transforming $\ell$ into the standard form (Remark \ref{rmk:shypPDE_tildeell}), the result follows easily from \cite[Lemma 3.7]{fruchtl2018}.
\end{proof}

\section{Product formula} \label{sec:prodform}

The goal of this section is to prove one of the key results of the paper, which is stated as follows:

\begin{theorem}[Product formula for $w_\lambda$] \label{thm:shypPDE_prodform_weaklim}
For each $x, y \in [a,b)$ there exists a measure $\bm{\nu}_{x,y} \in \mathcal{P}[a,b)$ such that the product $w_\lambda(x) \, w_\lambda(y)$ admits the integral representation
\[
w_\lambda(x) \, w_\lambda(y) = \int_{[a,b)} w_\lambda(\xi)\, \bm{\nu}_{x,y}(d\xi), \qquad x, y \in [a,b), \; \lambda \in \mathbb{C}.
\]
\end{theorem}

\subsection{The associated hyperbolic Cauchy problem} \label{sub:prodform_hypCP}

The proof of Theorem \ref{thm:shypPDE_prodform_weaklim} relies crucially on the basic properties of the hyperbolic Cauchy problem associated with $\ell$, i.e., of the boundary value problem defined by
\begin{equation} \label{eq:shypPDE_Lcauchy}
(\ell_x f)(x,y) = (\ell_y f)(x,y) \quad\; (x,y \in (a,b)), \qquad\quad
f(x,a) = h(x), \qquad\quad
(\partial_y^{[1]}\!f)(x,a) = 0
\end{equation}
where $h$ is a given sufficiently regular function, $\partial^{[1]} u = pu'$ and the subscripts indicate the variable in which the operators act. 

The assumptions on the coefficients of $\ell$ introduced in the previous sections allow for the higher order coefficient of $\ell$ to vanish at the endpoint $a$, in which case the hyperbolic Cauchy problem \eqref{eq:shypPDE_Lcauchy} is parabolically degenerate at the initial line. In general, such hyperbolic problems cannot be dealt with using the classical theory of hyperbolic equations in two variables. But, as we will show, the existence, uniqueness and positivity properties for the Cauchy problem \eqref{eq:shypPDE_Lcauchy} can be deduced by making use of the eigenfunction expansion of the Sturm-Liouville operator $\ell$.

\begin{theorem}[Existence and uniqueness of solution] \label{thm:shypPDE_Lexistuniq}
If $h \in \mathcal{D}_{\mathcal{L}}^{(2)\!}$ and\, $\ell(h) \in \mathcal{D}_{\mathcal{L}}^{(2)\!}$, then there exists a unique solution $f \in \mathrm{C}^2\bigl((a,b)^2\bigr)$ of the Cauchy problem \eqref{eq:shypPDE_Lcauchy} satisfying the conditions
\begin{enumerate}[itemsep=0pt,topsep=4pt]
\item[\textbf{(i)}] $f(\cdot, y) \in \mathcal{D}_\mathcal{L}^{(2)}$\, for all $a < y < b$;
\item[\textbf{(ii)}] There exists a zero $\bm{\rho}_\mathcal{L}$-measure set $\Lambda_0 \subset [\sigma^2,\infty)$ such that for each $\lambda \in [\sigma^2,\infty) \setminus \Lambda_0$ we have
\begin{gather}
\label{eq:shypPDE_uniq_v2_cond1} \mathcal{F}[\ell_y f(\cdot,y)](\lambda) = \ell_y [\mathcal{F}f(\cdot,y)](\lambda) \quad \text{for all } \, a < y < b, \\[2pt]
\label{eq:shypPDE_uniq_v2_cond2} \smash{\displaystyle \lim_{y \downarrow a} [\mathcal{F}f(\cdot,y)](\lambda) = (\mathcal{F}h)(\lambda), \qquad\; \lim_{y \downarrow a} \partial_y^{[1]\!} \mathcal{F}[f(\cdot,y)](\lambda) = 0. }
\end{gather}
\end{enumerate}
This unique solution is given by
\begin{equation} \label{eq:shypPDE_Lexistence}
f(x,y) = \int_{[\sigma^2,\infty)\!} w_\lambda(x) \, w_\lambda(y) \, (\mathcal{F} h)(\lambda) \, \bm{\rho}_{\mathcal{L}}(d\lambda).
\end{equation}
\end{theorem}

\begin{proof}
We start by proving that there exists at most one solution of \eqref{eq:shypPDE_Lcauchy} satisfying the given conditions. Let $f_1, f_2 \in \mathrm{C}^2\bigl((a,b)^2\bigr)$ be two solutions of $\ell_x f = \ell_y f$ such that \emph{(i)}--\emph{(ii)} hold for $f \in \{f_1, f_2\}$. Fix $\lambda \in [0,\infty) \setminus \Lambda_0$ and let $\Psi_{\!j}(y,\lambda) := [\mathcal{F}f_j(\cdot,y)](\lambda)$. We have
\[
\ell_y \Psi_{\!j}(y,\lambda) = \mathcal{F}[\ell_y f_j(\cdot,y)](\lambda) = \mathcal{F}[\ell_x f_j(\cdot,y)](\lambda) = \lambda \Psi_{\!j}(y,\lambda), \qquad a < y < b
\]
where the first equality is due to \eqref{eq:shypPDE_uniq_v2_cond1} and the last step follows from \eqref{eq:shypPDE_Ltransfidentity}. Moreover, 
\[
\lim_{y \downarrow a}\Psi_{\!j}(y,\lambda) = (\mathcal{F}h)(\lambda) \quad \text{ and } \quad \lim_{y \downarrow a} \partial_y^{[1]}\Psi_{\!j}(y,\lambda) = 0
\]
by \eqref{eq:shypPDE_uniq_v2_cond2}. It thus follows from Lemma \ref{lem:entrpf_ode_wsol} that
\[
[\mathcal{F}f_j(\cdot,y)](\lambda) = \Psi_{\!j}(y,\lambda) = (\mathcal{F}h)(\lambda)\, w_\lambda(y), \qquad a < y < b.
\]
This equality takes place for $\bm{\rho}_\mathcal{L}$-almost every $\lambda$, so the isometric property of $\mathcal{F}$ gives $f_1(\cdot,y) = f_2(\cdot,y)$ Lebesgue-almost everywhere; since the $f_j$ are continuous, we conclude that $f_1(x,y) \equiv f_2(x,y)$ for all $x,y \in (a,b)$.

In order to prove that \eqref{eq:shypPDE_Lexistence} is the (unique) solution, we need to justify that $\ell_x f$ can be computed via differentiation under the integral sign. It follows from \eqref{eq:entrpf_ode_wsol} that $w_\lambda^{[1]}(x) = - \lambda \int_a^x w_\lambda(\xi) \,r(\xi) d\xi$ and therefore (by Lemma \ref{lem:shypPDE_wsolbound}) $|w_\lambda^{[1]}(x)| \leq \lambda \int_a^x r(\xi) d\xi$. Hence
\begin{equation} \label{eq:shypPDE_Gsol_spectrep_pf1}
\int_{[\sigma^2,\infty)\!} \bigl| (\mathcal{F} h)(\lambda) \, w_\lambda^{[1]}(x) \, w_\lambda(y)\bigr| \bm{\rho}_{\mathcal{L}}(d\lambda) \leq \int_a^x r(\xi) d\xi \ccdot \int_{[\sigma^2,\infty)\!} \lambda \, \bigl| (\mathcal{F} h)(\lambda)\, w_\lambda(y) \bigr| \bm{\rho}_{\mathcal{L}}(d\lambda) < \infty,
\end{equation}
where the convergence (which is uniform on compacts) follows from \eqref{eq:shypPDE_Ltransfidentity} and Lemma \ref{lem:shypPDE_Lfourier_D2prop}(b). Due to the convergence of the differentiated integral, we have $\partial_x^{[1]}\!f(x,y) = \int_{[\sigma^2,\infty)\!} (\mathcal{F} h)(\lambda) \, w_\lambda^{[1]}(x) \, w_\lambda(y) \, \bm{\rho}_{\mathcal{L}}(d\lambda)$. Since $(\ell w_\lambda)(x) = \lambda w_\lambda(x)$, in the same way we check that $\int_{[\sigma^2,\infty)} (\mathcal{F} h)(\lambda) \, (\ell w_\lambda)(x)\, w_\lambda(y)\, \bm{\rho}_{\mathcal{L}}(d\lambda)$ converges absolutely and uniformly on compacts and is therefore equal to $(\ell_x f)(x,y)$. Consequently,
\[
(\ell_x f)(x,y) = (\ell_y f)(x,y) = \int_{[\sigma^2,\infty)\!} \lambda\, (\mathcal{F} h)(\lambda) \, w_\lambda(x) \, w_\lambda(y) \, \bm{\rho}_{\mathcal{L}}(d\lambda).
\]
Concerning the boundary conditions, Lemma \ref{lem:shypPDE_Lfourier_D2prop}(b) together with the fact that $w_\lambda(a) = 1$ imply that $f(x,a) = h(x)$, and from \eqref{eq:shypPDE_Gsol_spectrep_pf1} we easily see that $\lim_{y \downarrow a} \partial_y^{[1]}\!f(x,y) = 0$. This shows that $f$ is a solution of the Cauchy problem \eqref{eq:shypPDE_Lcauchy}.
\end{proof}

\begin{proposition}[Pointwise approximation by solutions of problems with shifted boundary] \label{prop:shypPDE_Gexistence_eps}
Let $\{a_m\}_{m \in \mathbb{N}}$ be a sequence $b > a_1 > a_2 > \ldots$ with $\lim a_m = a$. If $h \in \mathcal{D}_{\mathcal{L}}^{(2)}$ and\, $\ell(h) \in \mathcal{D}_{\mathcal{L}}^{(2)\!}$, then for each $m \in \mathbb{N}$ the function
\begin{equation} \label{eq:shypPDE_Lexistence_eps}
f_m(x,y) = \int_{[\sigma^2,\infty)\!} w_\lambda(x) \, w_{\lambda,m}(y) \, (\mathcal{F} h)(\lambda) \, \bm{\rho}_{\mathcal{L}}(d\lambda) \qquad \bigl(x \in (a,b), \: y \in (a_m,b)\bigr)
\end{equation}
is a solution of the Cauchy problem
\begin{equation} \label{eq:shypPDE_Leps_fepscauchy}
\begin{gathered}
(\ell_x f_m)(x,y) = (\ell_y f_m)(x,y), \qquad\; f_m(x,a_m) = h(x), \qquad\; (\partial_y^{[1]}\!f_m)(x,a_m) = 0.
\end{gathered}
\end{equation}
Moreover, we have 
\begin{equation} \label{eq:shypPDE_Leps_limit}
\lim_{m \to \infty} f_m(x,y) = f(x,y) \qquad \text{pointwise for each } x, y \in (a,b).
\end{equation}
where $f(x,y)$ is the solution \eqref{eq:shypPDE_Lexistence} of the Cauchy problem \eqref{eq:shypPDE_Lcauchy}.
\end{proposition}

\begin{proof}
Let us begin by justifying that $\partial_x^{[1]}\!f_m(x,y)$ and $(\ell_x f_m)(x,y)$ can be computed via differentiation under the integral sign. The differentiated integrals are given by
\begin{gather}
\label{eq:shypPDE_Gepssol_spectrep_pf1} \int_{[\sigma^2,\infty)\!} w_\lambda^{[1]}(x) \, w_{\lambda,m}(y) \, (\mathcal{F} h)(\lambda) \, \bm{\rho}_{\mathcal{L}}(d\lambda) \\
\label{eq:shypPDE_Gepssol_spectrep_pf2} \int_{[\sigma^2,\infty)\!} w_\lambda(x) \, w_{\lambda,m}(y) \, [\mathcal{F} (\ell (h))](\lambda) \, \bm{\rho}_{\mathcal{L}}(d\lambda)
\end{gather}
(for the latter, we used the identities $(\ell w_\lambda)(x) = \lambda w_\lambda(x)$ and \eqref{eq:shypPDE_Ltransfidentity}), and their absolute and uniform convergence on compacts follows from the fact that $h, \ell (h) \in \mathcal{D}_\mathcal{L}^{(2)}$, together with Lemma \ref{lem:shypPDE_Lfourier_D2prop}(b) and the inequality $|w_{\lambda,m}(\cdot)|\leq 1$ (which follows from Lemma \ref{lem:shypPDE_wsolbound} if we replace $a$ by $a_m$). This justifies that $\partial_x^{[1]}\!f_m(x,y)$ and $(\ell_x f_m)(x,y)$ are given by \eqref{eq:shypPDE_Gepssol_spectrep_pf1}, \eqref{eq:shypPDE_Gepssol_spectrep_pf2} respectively.

We also need to ensure that $\partial_y^{[1]}\!f_m(x,y)$ and $(\ell_y f_m)(x,y)$ are given by the corresponding differentiated integrals, and to that end we must check that
\[
\int_{[\sigma^2,\infty)\!} w_\lambda(x) \, w_{\lambda,m}^{[1]}(y) \, (\mathcal{F} h)(\lambda) \, \bm{\rho}_{\mathcal{L}}(d\lambda)
\]
converges absolutely and uniformly. Indeed, it follows from \eqref{eq:entrpf_ode_wsoleps} that for $y \geq a_m$ we have $w_{\lambda,m}^{[1]}(y) = \lambda \int_{a_m}^y w_{\lambda,m}(\xi) \,r(\xi) d\xi$ and consequently $|w_{\lambda,m}^{[1]}(y)| \leq \lambda \int_{a_m}^y \,r(\xi) d\xi$; hence
\begin{equation} \label{eq:shypPDE_Gepssol_spectrep_pf3}
\int_{[\sigma^2,\infty)\!} \bigl| w_\lambda(x) \, w_{\lambda,m}^{[1]}(y)\, (\mathcal{F} h)(\lambda)\bigr| \bm{\rho}_{\mathcal{L}}(d\lambda) \leq \int_{a_m}^y r(\xi) d\xi \ccdot\! \int_{[\sigma^2,\infty)\!} \lambda \bigl| w_\lambda(x) (\mathcal{F} h)(\lambda)\bigr| \bm{\rho}_{\mathcal{L}}(d\lambda)
\end{equation}
and the uniform convergence in compacts follows from \eqref{eq:shypPDE_Ltransfidentity} and Lemma \ref{lem:shypPDE_Lfourier_D2prop}(b).

The verification of the boundary conditions is straightforward: Lemma \ref{lem:shypPDE_Lfourier_D2prop}(b) together with the fact that $w_{\lambda,m}(a_m) = 1$ imply that $f_m(x,a_m) = h(x)$, and from \eqref{eq:shypPDE_Gepssol_spectrep_pf3} we easily see that $\partial_y^{[1]}\!f_m (x,a_m) = 0$. This shows that the function $f_m$ defined by \eqref{eq:shypPDE_Lexistence_eps} is a solution of the Cauchy problem \eqref{eq:shypPDE_Leps_fepscauchy}.

Since $w_{\lambda,m}(y) \to w_\lambda(y)$ as $m \to \infty$ (Lemma \ref{lem:entrpf_ode_wepslimit}), the pointwise convergence $f_m(x,y) \to f(x,y)$ follows from the dominated convergence theorem (which is applicable due to Lemmas \ref{lem:shypPDE_wsolbound} and \ref{lem:shypPDE_Lfourier_D2prop}(b)).
\end{proof}

\begin{proposition}[Positivity of solution for the problem with shifted boundary] \label{prop:hypPDE_sol_positivity_reg}
Let $\{a_m\}_{m \in \mathbb{N}}$ as in the previous proposition and let $h \in \mathcal{D}_{\mathcal{L}}^{(2)}$ with $\ell(h) \in \mathcal{D}_{\mathcal{L}}^{(2)\!}$. If $h \geq 0$, then the function $f_m$ given by \eqref{eq:shypPDE_Lexistence_eps} is such that 
\begin{equation} \label{eq:shypPDE_feps_positivity}
f_m(x,y) \geq 0 \qquad \text{for } x \geq y > a_m.
\end{equation}
If, in addition, $h \leq C$ (where $C$ is a constant), then $f_m(x,y) \leq C$ for $x \geq y > a_m$.
\end{proposition}

The proof of this positivity result, which is an adaptation of that of \cite[Proposition 3.7]{zeuner1992}, relies on a weak maximum principle which, in turn, is a consequence of the integral identity stated in the next lemma. We recall that $A$ is the function defined in Remark \ref{rmk:shypPDE_tildeell} and that $\eta$, $\bm{\phi}_\eta$, $\bm{\psi}_\eta$ have been defined in Assumption \ref{asmp:shypPDE_SLhyperg}.

\begin{lemma} \label{lem:shypPDE_inteqtriangle}
Let $\bm{\ell}^B$ be the differential expression $\bm{\ell}^B v := - v'' - \bm{\phi}_\eta v' + \bm{\psi}_\eta v$. For $\gamma(a) < c \leq y \leq x$, consider the triangle $\Delta_{c,x,y} := \{(\xi,\zeta) \in \mathbb{R}^2 \mid \zeta \geq c, \, \xi + \zeta \leq x+y, \, \xi - \zeta \geq x-y \}$, and let $v \in \mathrm{C}^2(\Delta_{c,x,y})$. Write $B(x):=\exp({1 \over 2} \int_\beta^x \eta(\xi)d\xi)$ (with $\beta > \gamma(a)$ arbitrary) and $A_{\mathsmaller{B}}(x) = {A(x) \over B(x)^2}$. Then the following integral equation holds:
\begin{equation} \label{eq:shypPDE_inteqtriangle}
A_{\mathsmaller{B}}(x) A_{\mathsmaller{B}}(y) \, v(x,y) = H + I_0 + I_1 + I_2 + I_3 - I_4
\end{equation}
where
\begin{align}
\label{eq:shypPDE_inteqtriangle_H} H & := \tfrac{1}{2} A_{\mathsmaller{B}}(c) \bigl[A_{\mathsmaller{B}}(x-y+c) \, v(x-y+c,c) + A_{\mathsmaller{B}}(x+y-c) \, v(x+y-c,c)] \\
\label{eq:shypPDE_inteqtriangle_I0} I_0 & := \tfrac{1}{2} A_{\mathsmaller{B}}(c) \int_{x-y+c}^{x+y-c} A_{\mathsmaller{B}}(s) (\partial_y v)(s,c)\, ds \\
\label{eq:shypPDE_inteqtriangle_I1}I_1 & := \tfrac{1}{2} \int_c^y A_{\mathsmaller{B}}(s) A_{\mathsmaller{B}}(x-y+s) \bigl[ \bm{\phi}_\eta(s) + \bm{\phi}_\eta(x-y+s) \bigr] v(x-y+s,s)\, ds \\
\label{eq:shypPDE_inteqtriangle_I2}I_2 & := \tfrac{1}{2} \int_c^y A_{\mathsmaller{B}}(s) A_{\mathsmaller{B}}(x+y-s) \bigl[ \bm{\phi}_\eta(s) - \bm{\phi}_\eta(x+y-s) \bigr] v(x+y-s,s)\, ds \\
\label{eq:shypPDE_inteqtriangle_I3}I_3 & := \tfrac{1}{2} \int_{\Delta_{c,x,y}\!} A_{\mathsmaller{B}}(\xi) A_{\mathsmaller{B}}(\zeta) \bigl[\bm{\psi}_\eta(\zeta) - \bm{\psi}_\eta(\xi)\bigr] v(\xi,\zeta)\, d\xi d\zeta \\
\label{eq:shypPDE_inteqtriangle_I4}I_4 & := \tfrac{1}{2} \int_{\Delta_{c,x,y}\!} A_{\mathsmaller{B}}(\xi) A_{\mathsmaller{B}}(\zeta) \, (\bm{\ell}_\zeta^B v - \bm{\ell}_\xi^B v)(\xi,\zeta)\, d\xi d\zeta.
\end{align}
\end{lemma}

\begin{proof}
Just compute 
\begin{align*}
I_4 - I_3 & = \tfrac{1}{2} \int_{\Delta_{c,x,y}}\biggl({\partial \over \partial \xi} \bigl[ A_{\mathsmaller{B}}(\xi) A_{\mathsmaller{B}}(\zeta) \, (\partial_\xi v)(\xi,\zeta) \bigr] - {\partial \over \partial \zeta} \bigl[ A_{\mathsmaller{B}}(\xi) A_{\mathsmaller{B}}(\zeta) \, (\partial_\zeta v)(\xi,\zeta) \bigr] \biggr) d\xi d\zeta \\
& = I_0 - \tfrac{1}{2} \int_0^y A_{\mathsmaller{B}}(s) A_{\mathsmaller{B}}(x-y+s) \, (\partial_\zeta v + \partial_\xi v)(x-y+s,s) \, ds \\
& \quad\; - \tfrac{1}{2} \int_0^y A_{\mathsmaller{B}}(s) A_{\mathsmaller{B}}(x+y-s) \, (\partial_\zeta v - \partial_\xi v)(x+y-s,s) \, ds \\
& = I_0 + I_1 - \int_c^y {d \over ds} \bigl[ A_{\mathsmaller{B}}(s) A_{\mathsmaller{B}}(x-y+s) \, v(x-y+s,s) \bigr] ds \\
& \quad\; + I_2 - \int_c^y {d \over ds} \bigl[ A_{\mathsmaller{B}}(s) A_{\mathsmaller{B}}(x+y-s) \, v(x+y-s,s) \bigr] ds
\end{align*}
where in the second equality we used Green's theorem, and the third equality follows easily from the fact that $(A_{\mathsmaller{B}})' = \bm{\phi}_\eta A_{\mathsmaller{B}}$.
\end{proof}

\begin{corollary}[Weak maximum principle] \label{cor:shypPDE_tildeell_maxprinc}
Suppose Assumption \ref{asmp:shypPDE_SLhyperg} holds, and let $\gamma(a) < c \leq y_0 \leq x_0$. If $u \in \mathrm{C}^2(\Delta_{c,x_0,y_0})$ satisfies
\begin{equation} \label{eq:shypPDE_tildeell_maxprinc_ineq}
\begin{aligned}
(\widetilde{\ell}_x u - \widetilde{\ell}_y u)(x,y) \leq 0, & \qquad (x,y) \in \Delta_{c,x_0,y_0} \\
u(x,c) \geq 0, & \qquad x \in [x_0-y_0+c,x_0+y_0-c] \\
(\partial_y u)(x,c) + \tfrac{1}{2} \eta(c) u(x,c) \geq 0, & \qquad x \in [x_0-y_0+c,x_0+y_0-c]
\end{aligned}
\end{equation}
then $u \geq 0$ in $\Delta_{c,x_0,y_0}$.
\end{corollary}

\begin{proof}
Pick a function $\omega \in \mathrm{C}^2[c,\infty)$ such that $\bm{\ell}^B \omega < 0$, $\omega(c) > 0$ and $\omega'(c) \geq 0$ (where $\bm{\ell}_x^B$ is the differential operator defined in Lemma \ref{lem:shypPDE_inteqtriangle}). Clearly, it is enough to show that for all $\delta > 0$ we have $v(x,y) := B(x) B(y) u(x,y) + \delta \omega(y) > 0$ for $(x, y) \in \Delta_{c,x_0,y_0}$.

By Lemma \ref{lem:shypPDE_inteqtriangle}, the integral equation \eqref{eq:shypPDE_inteqtriangle} holds for the function $v$. Assume by contradiction that there exist $\delta > 0$, $(x, y) \in \Delta_{c,x_0,y_0}$ for which we have $v(x,y) = 0$ and $v(\xi,\zeta) \geq 0$ for all $(\xi,\zeta) \in \Delta_{c,x,y} \subset \Delta_{c,x_0,y_0}$. It is clear from the choice of $\omega$ that $v(\cdot,c) > 0$, thus we have $H \geq 0$ in the right hand side of \eqref{eq:shypPDE_inteqtriangle}. Similarly, $(\partial_y v)(\cdot,c) = B(x) B(y) \bigl[(\partial_y u)(\cdot,c) + \tfrac{1}{2} \eta(c) u(\cdot,c)\bigr] + \delta \omega'(c) \geq 0$, hence $I_0 \geq 0$. Since $\bm{\phi}_\eta$ is positive and decreasing and $\bm{\psi}_\eta$ is decreasing (cf.\ Assumption \ref{asmp:shypPDE_SLhyperg}) and we are assuming that $u \geq 0$ on $\Delta_{c,x,y}$, it follows that $I_1 \geq 0$, $I_2 \geq 0$ and $I_3 \geq 0$. In addition, $I_4 < 0$ because $(\bm{\ell}_\zeta^B v - \bm{\ell}_\xi^B v)(\xi,\zeta) = B(x) B(y) (\widetilde{\ell}_\zeta u - \widetilde{\ell}_\xi u)(\xi,\zeta) + (\bm{\ell}^B\omega)(\zeta) < 0$. Consequently, \eqref{eq:shypPDE_inteqtriangle} yields $0 = A_{\mathsmaller{B}}(x) A_{\mathsmaller{B}}(y) v(x,y) \geq -I_4 > 0$. This contradiction shows that $v(x,y) > 0$ for all $(x,y) \in \Delta_{c,x_0,y_0}$.
\end{proof}

\begin{proof}[Proof of Proposition \ref{prop:hypPDE_sol_positivity_reg}]
It follows from Proposition \ref{prop:shypPDE_Gexistence_eps} that the function $u_m(x,y) := f_m(\gamma^{-1\!}(x), \gamma^{-1\!}(y))$ is a solution of the Cauchy problem
\begin{align}
\label{eq:shypPDE_Leps_tildefepscauchy1} (\widetilde{\ell}_x u_m)(x,y) = (\widetilde{\ell}_y u_m)(x,y), & \qquad x, y > \tilde{a}_m \\
\label{eq:shypPDE_Leps_tildefepscauchy2} u_m(x,\tilde{a}_m) = h(\gamma^{-1}(x)), & \qquad x > \tilde{a}_m \\
\label{eq:shypPDE_Leps_tildefepscauchy3} (\partial_y u_m)(x,\tilde{a}_m) = 0, & \qquad x > \tilde{a}_m
\end{align}
where $\tilde{a}_m = \gamma(a_m)$. Clearly, $u_m$ satisfies the inequalities \eqref{eq:shypPDE_tildeell_maxprinc_ineq} for arbitrary $x_0 \geq y_0 \geq \tilde{a}_m$ (here $c = \tilde{a}_m$). By Corollary \ref{cor:shypPDE_tildeell_maxprinc}, $u_m(x_0,y_0) \geq 0$ for all $x_0 \geq y_0 > \tilde{a}_m$; consequently, \eqref{eq:shypPDE_feps_positivity} holds.

The proof that $h \leq C$ implies $f_m \leq C$ is straightforward: if we have $h \leq C$, then $\widetilde{u}_m(x,y) = C - u_m(x,y)$ is a solution of \eqref{eq:shypPDE_Leps_tildefepscauchy1} with initial conditions $\widetilde{u}_m(x,\tilde{a}_m) = C - h(\gamma^{-1}(x)) \geq 0$ and \eqref{eq:shypPDE_Leps_tildefepscauchy3}, thus the reasoning of the previous paragraph yields that $C - u_m \geq 0$ for $x \geq y > \tilde{a}_m$.
\end{proof}

\begin{corollary}[Positivity of solution for the Cauchy problem \eqref{eq:shypPDE_Lcauchy}] \label{cor:hypPDE_sol_positivity}
Let $h \in \mathcal{D}_\mathcal{L}^{(2)}$ with $\ell(h) \in \mathcal{D}_{\mathcal{L}}^{(2)\!}$. If $h \geq 0$, then the function $f$ given by \eqref{eq:shypPDE_Lexistence} is such that 
\[
f(x,y) \geq 0 \qquad \text{for } x, y \in (a,b).
\]
If, in addition, $h \leq C$, then $f(x,y) \leq C$ for $x, y \in (a,b)$.
\end{corollary}

\begin{proof}
This is an immediate consequence of Proposition \ref{prop:hypPDE_sol_positivity_reg} together with the pointwise convergence property \eqref{eq:shypPDE_Leps_limit} (note that the conclusion holds for all $x,y \in (a,b)$ because by \eqref{eq:shypPDE_Lexistence} we have $f(x,y) = f(y,x)$).
\end{proof}

\subsection{The time-shifted product formula}

Before proving the product formula for the kernels $\{w_\lambda(\cdot)\}$ themselves, we will establish a product formula of the form \eqref{eq:intro_prodform} for the functions $\{e^{-t\lambda} w_\lambda(\cdot)\}$. This auxiliary result will be called the \emph{time-shifted product formula} because of the transition identity $(T_t w_\lambda)(x) = e^{-t\lambda} w_\lambda(x)$, which means that $e^{-t\lambda} w_\lambda(x)$ is the $\mathcal{L}$-transform of the transition kernel $p(t,x,y)$ of the transition (Feller) semigroup $\{T_t\}$ generated by the Sturm-Liouville operator $\ell$, cf.\ Proposition \ref{prop:shypPDE_L_fellergen_tpdf}.

By the inversion formula \eqref{eq:entrpf_Ffourierinverse} for the $\mathcal{L}$-transform, a natural candidate for the measure of the product formula for $\{w_\lambda(\cdot)\}$ is
\[
\bm{\nu}_{x,y}(d\xi) = \int_{[\sigma^2,\infty)} w_\lambda(x) w_\lambda(y) w_\lambda(\xi) \bm{\rho}_\mathcal{L}(d\lambda) \, r(\xi) d\xi.
\]
This is only a formal solution, because in general the integral does not converge. However, the uniform convergence of this integral always holds (under the present assumptions on $\ell$) if the exponential term $e^{-t\lambda}$ is included in the integrand:

\begin{lemma}
Let $t_0 > 0$ and $K_1, K_2$ compact subsets of $(a,b)$. The integral
\[
\int_{[\sigma^2,\infty)\!} e^{-t\lambda\,} w_\lambda(x) \, w_\lambda(y) \, w_\lambda(\xi) \, \bm{\rho}_{\mathcal{L}}(d\lambda)
\]
converges absolutely and uniformly on $(t,x,y,\xi) \in [t_0,\infty) \times K_1 \times K_2 \times [a,b)$.
\end{lemma}

\begin{proof}
This follows from Lemma \ref{lem:shypPDE_wsolbound} and the uniform convergence property of the integral representation of the transition kernel of the Feller semigroup generated by $\ell$ (Proposition \ref{prop:shypPDE_L_fellergen_tpdf}).
\end{proof}

In what follows we write
\[
q_t(x,y,\xi) := \int_{[\sigma^2,\infty)\!} e^{-t\lambda\,} w_\lambda(x) \, w_\lambda(y) \, w_\lambda(\xi) \, \bm{\rho}_{\mathcal{L}}(d\lambda).
\]
This function, which is (at least formally) the density of the measure of the time-shifted product formula, is for fixed $t,x,y$ the density (with respect to $r(\xi) d\xi$) of a subprobability measure:

\begin{lemma} \label{lem:shypPDE_exptriplepf_probmeas}
The function $q_t(x,y,\xi)$ is nonnegative and such that $\int_a^b q_t(x,y,\xi)\, r(\xi) d\xi \leq 1$ for all $(t,x,y) \in (0,\infty) \times (a,b) \times (a,b)$.
\end{lemma}

Throughout the proof (and in the sequel) we write $\mathcal{D}_\mathcal{L}^{(2,0)} := \mathcal{D}_\mathcal{L}^{(2)} \cap \mathcal{D}_\mathcal{L}^{(0)}$. Note that if $g \in \mathrm{C}_\mathrm{c}^2[a,b)$ with $g' \in \mathrm{C}_\mathrm{c}(a,b)$, then $g \in \mathcal{D}_\mathcal{L}^{(2,0)}$; consequently, any indicator function of an interval $I \subset [a,b)$ is the pointwise limit of functions $g_n \in \mathcal{D}_\mathcal{L}^{(2,0)}$.

\begin{proof}
Since $q_t(x,y,\cdot) \in \mathrm{C}_\mathrm{b}[a,b)$, it suffices to show that for all $g \in \mathcal{D}_\mathcal{L}^{(2,0)}$ with $0 \leq g \leq 1$ we have
\[
0 \leq \mathcal{Q}_{t,g}(x,y) \leq 1 \qquad \bigl(t > 0, \; x, y \in (a,b)\bigr)
\]
where $\mathcal{Q}_{t,g}(x,y) := \int_a^b g(\xi) \, q_t(x,y,\xi)\, r(\xi) d\xi$.

Fix $t > 0$ and $g \in \mathcal{D}_\mathcal{L}^{(2,0)}$ with $0 \leq g \leq 1$. Since $\mathcal{F}[q_t(x,y,\cdot)] = e^{-t\lambda\,} w_\lambda(x) \, w_\lambda(y)$, it follows from the isometric property of the $\mathcal{L}$-transform (Proposition \ref{prop:entrpf_Ffourier}) that
\[
\mathcal{Q}_{t,g}(x,y) = \int_{[\sigma^2,\infty)\!} e^{-t\lambda\,} w_\lambda(x) \, w_\lambda(y) \, (\mathcal{F}g)(\lambda) \, \bm{\rho}_{\mathcal{L}}(d\lambda).
\]
Differentiating under the integral sign we easily check (by dominated convergence and using Lemma \ref{lem:shypPDE_Lfourier_D2prop}(b)) that $\ell_x \mathcal{Q}_{t,g} = \ell_y	 \mathcal{Q}_{t,g}$, $(\partial_y^{[1]} \mathcal{Q}_{t,g})(x,a) = 0$ and
\[
\mathcal{Q}_{t,g}(x,a) = \int_{[\sigma^2,\infty)\!} e^{-t\lambda\,} w_\lambda(x) \, (\mathcal{F}g)(\lambda) \, \bm{\rho}_{\mathcal{L}}(d\lambda) = (T_t g)(x)
\]
where the last equality follows from \eqref{eq:shypPDE_fellersgp_fundsolL2rep} (here $\{T_t\}$ is the Feller semigroup generated by $\ell$, cf.\ Proposition \ref{prop:shypPDE_L_fellergen_tpdf}). The fact that $0 \leq g \leq 1$ clearly implies that $0 \leq (T_t g)(x) \leq 1$ for $x \in (a,b)$. One can verify via \eqref{eq:shypPDE_LtransfidentD2} that the function $h(x) = (T_t g)(x)$ is such that $h \in \mathcal{D}_\mathcal{L}^{(2)}$ and $\ell(h) \in \mathcal{D}_{\mathcal{L}}^{(2)\!}$. It then follows from the positivity property of the hyperbolic Cauchy problem (Corollary \ref{cor:hypPDE_sol_positivity}) that $0 \leq \mathcal{Q}_{t,g}(x,y) \leq 1$ for all $x,y \in (a,b)$, as claimed.
\end{proof}

\begin{proposition}[Time-shifted product formula] \label{prop:shypPDE_expprodform}
The product $e^{-t\lambda\,} w_\lambda(x) \, w_\lambda(y)$ admits the integral representation
\begin{equation} \label{eq:shypPDE_expprodform}
e^{-t\lambda\,} w_\lambda(x) \, w_\lambda(y) = \int_a^b w_\lambda(\xi)\, q_t(x,y,\xi)\, r(\xi) d\xi, \qquad t > 0, \; x, y \in (a,b), \; \lambda \geq 0
\end{equation}
where the integral in the right hand side is absolutely convergent.

In particular, $\int_a^b q_t(x,y,\xi)\, r(\xi) d\xi = 1$ for all $t > 0$, $x,y \in (a,b)$.
\end{proposition}

\begin{proof}
The absolute convergence of the integral in the right hand side is immediate from Lemmas \ref{lem:shypPDE_wsolbound} and \ref{lem:shypPDE_exptriplepf_probmeas}.

By Proposition \ref{prop:entrpf_Ffourier}, the equality in \eqref{eq:shypPDE_expprodform} holds $\bm{\rho}_\mathcal{L}$-almost everywhere. Since $\supp(\bm{\rho}_\mathcal{L}) = [\sigma^2, \infty)$ (Lemma \ref{prop:shypPDE_spectralsupp}), the fact that both sides of \eqref{eq:shypPDE_expprodform} are continuous functions of $\lambda \geq 0$ allows us to extend by continuity the equality \eqref{eq:shypPDE_expprodform} to all $\lambda \geq \sigma^2$. If $\sigma = 0$, we are done.

Suppose that $\sigma > 0$. By \eqref{eq:shypPDE_Wbound} and Lemma \ref{lem:shypPDE_exptriplepf_probmeas}, together with standard results on the analyticity of parameter-dependent integrals, the function $\tau \mapsto \int_a^b w_{\tau^2+\sigma^2}(\xi)\, q_t(x,y,\xi)\, r(\xi) d\xi$ is an analytic function of $\tau$ in the strip $|\mathrm{Im}(\tau)| < \sigma$. It is also clear that $\tau \mapsto e^{-t(\tau^2 + \sigma^2)\,} w_{\tau^2 + \sigma^2}(x) \, w_{\tau^2 + \sigma^2}(y)$ is an entire function. By analytic continuation we see that these two functions are equal for all $\tau$ in the strip $|\mathrm{Im}(\tau)| < \sigma$; consequently, \eqref{eq:shypPDE_expprodform} holds.

The last statement is obtained by setting $\lambda = 0$.
\end{proof}

\subsection{The product formula for $w_\lambda$ as the limit case}

As one would expect, the product formula \eqref{eq:shypPDE_prodform_weaklim} will be deduced by taking (in a suitable way) the limit as $t \downarrow 0$ in the time-shifted product formula \eqref{eq:shypPDE_expprodform}. First we present a lemma which will be needed to handle the case where the functions $w_\lambda(x)$ do not vanish at the limit $x \uparrow b$ (cf.\ Proposition \ref{prop:hypPDE_Wasym_infty}). The lemma is based on a known result on changes of spectral functions for Sturm-Liouville operators and Krein strings (\cite{krein1953}, see also \cite[Section 6.9]{dymmckean1976}).

\begin{lemma} \label{lem:shypPDE_modification}
For $-\infty < \kappa \leq \sigma^2$, consider the modified differential expression
\[
\ell^{\langle\kappa\rangle\!} = -{1 \over r^{\langle\kappa\rangle\!}} {d \over dx} \Bigl( p^{\langle\kappa\rangle\!} \, {d \over dx}\Bigr), \qquad x \in (a,b)
\]
where $p^{\langle\kappa\rangle\!} = w_\kappa^2 \ccdot p$ and $r^{\langle\kappa\rangle\!} = w_\kappa^2 \ccdot r$. Then Assumption \ref{asmp:shypPDE_SLhyperg} also holds for $\ell^{\langle\kappa\rangle\!}$, and the function
\[
w_{\lambda}^{\langle\kappa\rangle\!}(x) := {w_{\kappa + \lambda}(x) \over w_\kappa(x)}
\]
is, for each $\lambda \in \mathbb{C}$, the unique solution of $\ell^{\langle\kappa\rangle\!}(w) = \lambda w$, $w(a) = 1$ and $(p^{\langle\kappa\rangle\!} w')(a) = 0$. Moreover, the spectral measure associated with $\ell^{\langle\kappa\rangle}$ is given by
\[
\bm{\rho}_{\mathcal{L}}^{\langle\kappa\rangle\!} (\lambda_1,\lambda_2] = \bm{\rho}_{\mathcal{L}} (\lambda_1+\kappa,\lambda_2+\kappa] \qquad (-\infty < \lambda_1 \leq \lambda_2 < \infty).
\]
\end{lemma}

\begin{proof}
Fix $-\infty < \kappa \leq \sigma^2$. The functions $A$ and $A^{\langle\kappa\rangle\!}$ associated to the operators $\ell$ and $\ell^{\langle\kappa\rangle\!}$ respectively (defined as in \eqref{eq:shypPDE_tildeell_A}) are connected by $A^{\langle\kappa\rangle\!} = \widetilde{w}_\kappa^2 \ccdot A$, where $\widetilde{w}_\kappa(\xi) = w_\kappa(\gamma^{-1}(\xi))$.

In order to show that Assumption \ref{asmp:shypPDE_SLhyperg} holds for $\ell^{\langle \kappa \rangle}$, consider first the function $A^{\langle \kappa,m \rangle\!}(\xi) := \widetilde{w}_{\kappa,m}^2(\xi) \ccdot A(\xi)$, where $\tilde{a}_m \leq \xi < \infty$ and $\widetilde{w}_{\kappa,m}$ is defined as in the proof of Theorem \ref{thm:shypPDE_Wlaplacerep}. Let $\eta^{\langle \kappa,m \rangle\!} := \eta + 2{\widetilde{w}_{\kappa,m}' \over \widetilde{w}_{\kappa,m}}$, where $\eta$ satisfies the conditions of Assumption \ref{asmp:shypPDE_SLhyperg}. It is easily seen (cf.\ \cite[Example 4.6]{zeuner1992}) that 
\[
\bm{\phi}_{\eta^{\langle \kappa,m \rangle\!}} := {(A^{\langle \kappa,m \rangle\!})' \over A^{\langle \kappa,m \rangle\!}} - \eta^{\langle \kappa,m \rangle\!} = \bm{\phi}_\eta, \qquad \bm{\psi}_{\eta^{\langle \kappa,m \rangle\!}} = \bm{\psi}_\eta - \kappa, \qquad \eta^{\langle \kappa,m \rangle\!}(\tilde{a}_m) = \eta(\tilde{a}_m) \geq 0
\]
and then it follows from \cite[Remark 2.12]{zeuner1992} that $\eta^{\langle \kappa,m \rangle\!} \geq 0$, hence $\eta^{\langle \kappa,m \rangle\!}$ satisfies Assumption \ref{asmp:shypPDE_SLhyperg} for the function $A^{\langle \kappa,m \rangle\!}$. If we now let $\eta^{\langle\kappa\rangle\!}(\xi) := \eta(\xi) + 2{\widetilde{w}_\kappa'(\xi) \over \widetilde{w}_\kappa(\xi)} = \lim_{m \to \infty} \eta^{\langle \kappa,m \rangle\!}(\xi)$ (where $\gamma(a) < \xi < \infty$; the second equality is due to \eqref{eq:shypPDE_ode_wepslimit}), then it is clear that the limit function $\eta^{\langle \kappa \rangle\!}$ satisfies Assumption \ref{asmp:shypPDE_SLhyperg} for the function $A^{\langle \kappa \rangle\!}$ associated with the operator $\ell^{\langle\kappa\rangle\!}$.

A simple computation gives
\begin{align*}
- {1 \over r^{\langle\kappa\rangle\!}} \biggl[p^{\langle\kappa\rangle\!} \Bigl({w_{\kappa + \lambda} \over w_\kappa}\Bigr) '\biggr]' & = - {1 \over w_\kappa^2 \ccdot r}\bigl[p \, w_{\kappa + \lambda}' w_{\kappa} - p \, w_{\kappa + \lambda} w_{\kappa}'\bigr]' \\
& = -{1 \over w_\kappa^2} \bigl[\ell(w_{\kappa + \lambda}) \, w_\kappa - w_{\kappa + \lambda} \, \ell(w_\kappa)\bigr] = \lambda {w_{\kappa + \lambda}(x) \over w_\kappa(x)}
\end{align*}
so that $\ell^{\langle\kappa\rangle\!}(w_{\lambda}^{\langle\kappa\rangle\!}) = \lambda w_{\lambda}^{\langle\kappa\rangle\!}$. The boundary conditions at $a$ are also straightforwardly checked. To prove the last assertion, notice that the Fourier transforms associated with $\ell$ and $\ell^{\langle\kappa\rangle\!}$ are related through the identity
\[
\Bigl(\mathcal{F}^{\langle\kappa\rangle\!} {h \over w_\kappa}\Bigr) (\lambda) = (\mathcal{F}h) (\kappa + \lambda), \qquad h \in L_2(r)
\]
and therefore
\[
\|(\mathcal{F}h)\|_{L_2(\mathbb{R},\bm{\rho}_{\mathcal{L}})} = \|h\|_{L_2(r)} = \biggl\| {h \over w_\kappa} \biggr\|_{L_2(r^{\langle\kappa\rangle\!})} \!\! = \bigl\| (\mathcal{F}h)(\kappa + \cdot) \bigr\|_{L_2(\mathbb{R},\bm{\rho}_{\mathcal{L}}^{\langle\kappa\rangle\!})}.
\]
Recalling the uniqueness of the spectral measure for which the isometric property in Proposition \ref{prop:entrpf_Ffourier} holds, we deduce that $\bm{\rho}_{\mathcal{L}}^{\langle\kappa\rangle\!} (\lambda_1,\lambda_2] = \bm{\rho}_{\mathcal{L}} (\lambda_1+\kappa,\lambda_2+\kappa]$.
\end{proof}

We are finally ready to prove the product formula for the $\mathcal{L}$-transform kernels $\{w_\lambda(\cdot)\}$. Recall that, by definition \cite[\S30]{bauer2001}, the complex measures $\mu_n$ converge weakly (respectively, vaguely) to $\mu \in \mathcal{M}_\mathbb{C}[a,b)$ if $\lim_n \int_{[a,b)} g(\xi) \mu_n(d\xi) = \int_{[a,b)} g(\xi) \mu(d\xi)$ for all $g \in \mathrm{C}_\mathrm{b}[a,b)$ (respectively, for all $g \in \mathrm{C}_0[a,b)$). We use the notations $\warrow$ and $\varrow$ to denote weak and vague convergence, respectively.

\begin{theorem}[Product formula for $w_\lambda$] \label{thm:shypPDE_prodform_weaklim_full}
For $x,y \in (a,b)$ and $t > 0$, let $\bm{\nu}_{t,x,y} \in \mathcal{P}[a,b)$ be the measure defined by $\bm{\nu}_{t,x,y}(d\xi) = q_t(x,y,\xi)\, r(\xi) d\xi$. Then for each $x,y \in (a,b)$ there exists a measure $\bm{\nu}_{x,y} \in \mathcal{P}[a,b)$ such that $\bm{\nu}_{t,x,y} \warrow \bm{\nu}_{x,y}$ as $t \downarrow 0$. Moreover, the product $w_\lambda(x) \, w_\lambda(y)$ admits the integral representation
\begin{equation} \label{eq:shypPDE_prodform_weaklim}
w_\lambda(x) \, w_\lambda(y) = \int_{[a,b)} w_\lambda(\xi)\, \bm{\nu}_{x,y}(d\xi), \qquad x, y \in (a,b), \; \lambda \in \mathbb{C}.
\end{equation}
In particular, Theorem \ref{thm:shypPDE_prodform_weaklim} holds.
\end{theorem}

\begin{proof}
Let $\{t_n\}_{n \in \mathbb{N}}$ be an arbitrary decreasing sequence with $t_n \downarrow 0$. Since any sequence of probability measures contains a vaguely convergent subsequence \cite[p.\ 213]{bauer2001}, there exists a subsequence $\{t_{n_k}\}$ and a measure $\bm{\nu}_{x,y} \in \mathcal{M}_+[a,b)$ such that $\bm{\nu}_{t_{n_k},x,y}\! \varrow \bm{\nu}_{x,y}$ as $k \to \infty$. Let us show that all such subsequences $\{\bm{\nu}_{t_{n_k},x,y}\}$ have the same vague limit. Suppose that $t_k^{1}$, $t_k^{2}$ are two different sequences with $t_k^{j} \downarrow 0$ and that $\bm{\nu}_{t_k^{j},x,y}\! \varrow \bm{\nu}_{x,y}^{j}$ as $k \to \infty$ ($j=1,2$). For $g \in \mathcal{D}_\mathcal{L}^{(2,0)}$ we have
\begin{align*}
\int_{[a,b)} g(\xi) \, \bm{\nu}_{x,y}^{j}(d\xi) & = \lim_{k \to \infty} \int_{[a,b)} g(\xi) \, \bm{\nu}_{t_k^{j},x,y}(d\xi) \\
& = \lim_{k \to \infty} \int_{[\sigma^2,\infty)\!} e^{-t_k^{j}\lambda\,} w_\lambda(x) \, w_\lambda(y) \, (\mathcal{F}g)(\lambda) \, \bm{\rho}_{\mathcal{L}}(d\lambda) \\
& = \int_{[\sigma^2,\infty)\!} w_\lambda(x) \, w_\lambda(y) \, (\mathcal{F}g)(\lambda) \, \bm{\rho}_{\mathcal{L}}(d\lambda)
\end{align*}
(the second equality was justified in the proof of Lemma \ref{lem:shypPDE_exptriplepf_probmeas}, and dominated convergence yields the last equality). In particular, $\int_{[a,b)} g(\xi) \, \bm{\nu}_{x,y}^{1}(\xi) = \int_{[a,b)} g(\xi) \, \bm{\nu}_{x,y}^{2}(\xi)$ for all $g \in \mathcal{D}_\mathcal{L}^{(2,0)}$, and this implies that $\bm{\nu}_{x,y}^{1} = \bm{\nu}_{x,y}^{2}$. Since all subsequences have the same vague limit, we conclude that $\bm{\nu}_{t,x,y} \varrow \bm{\nu}_{x,y}$ as $t \downarrow 0$.

Suppose first that $\lim_{x \uparrow b} p(x)r(x) = \infty$. Then, by Proposition \ref{prop:hypPDE_Wasym_infty}, we have $w_\lambda \in \mathrm{C}_0[a,b)$ for $\lambda > 0$. Accordingly, by taking the limit as $t \downarrow 0$ of both sides of \eqref{eq:shypPDE_expprodform} we deduce that the product formula \eqref{eq:shypPDE_prodform_weaklim} holds for all $\lambda > 0$.

To prove that \eqref{eq:shypPDE_prodform_weaklim} is valid in the general case, let $\kappa < 0$ be arbitrary. We know that the operator $\ell^{\langle\kappa\rangle\!}$ from Lemma \ref{lem:shypPDE_modification} satisfies Assumption \ref{asmp:shypPDE_SLhyperg}; moreover, since $w_\kappa$ is increasing and unbounded (this follows from classical results on the solutions of Sturm-Liouville type equations, e.g.\ \cite[Sections 5.13--5.14]{ito2006}), we have $\lim_{x \uparrow b} p^{\langle\kappa\rangle\!}(x) r^{\langle\kappa\rangle\!}(x) = \infty$. From the previous part of the proof,
\begin{equation} \label{eq:shypPDE_prodform_weaklim_pf1}
w_{\lambda}^{\langle\kappa\rangle\!}(x) \, w_{\lambda}^{\langle\kappa\rangle\!}(y) = \int_a^b w_{\lambda}^{\langle\kappa\rangle\!}(\xi)\, \bm{\nu}_{x,y}^{\langle\kappa\rangle\!}(d\xi), \qquad x, y \in (a,b), \; \lambda > 0
\end{equation}
with $\bm{\nu}_{x,y}^{\langle\kappa\rangle\!}$ constructed as before. We easily verify that $q_t^{\langle\kappa\rangle\!}(x,y,\xi) r^{\langle\kappa\rangle\!}(\xi) = {e^{t\kappa} w_\kappa(\xi) \over w_\kappa(x) w_\kappa(y)} q_t(x,y,\xi) r(\xi)$ and, consequently, $\bm{\nu}_{x,y}^{\langle\kappa\rangle\!}(d\xi) = {w_\kappa(\xi) \over w_\kappa(x) w_\kappa(y)} \bm{\nu}_{x,y}(d\xi)$. It thus follows from \eqref{eq:shypPDE_prodform_weaklim_pf1} that
\[
w_{\kappa + \lambda}(x) \, w_{\kappa + \lambda}(y) = \int_a^b w_{\kappa + \lambda}(\xi)\, \bm{\nu}_{x,y}(d\xi), \qquad x, y \in (a,b), \; \lambda > 0,
\]
where $\kappa < 0$ is arbitrary; hence \eqref{eq:shypPDE_prodform_weaklim} holds for all $\lambda \in \mathbb{R}$. If we then set $\lambda = \tau^2 + \sigma^2$ in \eqref{eq:shypPDE_prodform_weaklim}, we straightforwardly verify that both sides are entire functions of $\tau$ (for the right hand side, this follows from the Laplace-type representation \eqref{eq:shypPDE_Wlaplacerep} and the fact that the integral converges for all $\lambda < 0$), so by analytic continuation the product formula holds for all $\lambda \in \mathbb{C}$.

Given that $w_0(x) \equiv 1$, setting $\lambda = 0$ in \eqref{eq:shypPDE_prodform_weaklim} shows that $\bm{\nu}_{x,y} \in \mathcal{P}[a,b)$; consequently, the measures $\bm{\nu}_{t,x,y}$ converge to $\bm{\nu}_{x,y}$ in the weak topology (cf.\ \cite[Theorem 30.8]{bauer2001}). Clearly, the product formula \eqref{eq:shypPDE_prodform_weaklim} can be extended to $x,y \in [a,b)$ by setting $\bm{\nu}_{x,a} := \delta_x$ and $\bm{\nu}_{a,y} := \delta_y$, hence Theorem \ref{thm:shypPDE_prodform_weaklim} holds.
\end{proof}

It is worth commenting that the reasoning used in this proof also allows us to justify that the time-shifted product formula \eqref{eq:shypPDE_expprodform} is valid for all $\lambda \in \mathbb{C}$.

As shown in the proof above, the measure $\bm{\nu}_{x,y}$ of the product formula \eqref{eq:shypPDE_prodform_weaklim} is characterized by the identity
\begin{equation} \label{eq:shypPDE_prodformmeas_charact}
\int_{[a,b)} h(\xi) \, \bm{\nu}_{x,y} (d\xi) = \int_{[\sigma^2,\infty)\!} w_\lambda(x) \, w_\lambda(y) \, (\mathcal{F}h)(\lambda) \, \bm{\rho}_{\mathcal{L}}(d\lambda), \qquad h \in \mathcal{D}_\mathcal{L}^{(2,0)}.
\end{equation}
Furthermore, the relation between this measure and the measure $\bm{\nu}_{t,x,y}(d\xi) = q_t(x,y,\xi)\, r(\xi) d\xi$ of the time-shifted product formula \eqref{eq:shypPDE_expprodform} can be written explicitly:

\begin{corollary} \label{cor:shypPDE_prodformreg_measrep}
The measure $\bm{\nu}_{t,x,y}$ can be written in terms of the measure $\bm{\nu}_{x,y}$ and the transition kernel $p(t,x,y)$ of the Feller semigroup generated by the Sturm-Liouville operator $\ell$ as
\[
\bm{\nu}_{t,x,y}(d\xi) = \int_a^b \bm{\nu}_{z,y}(d\xi)\, p(t,x,z)\, r(z) dz \qquad \bigl( t > 0, \; x, y \in (a,b) \bigr).
\]
\end{corollary}

\begin{proof}
Recalling \eqref{eq:shypPDE_fellersgp_fundsolL2rep} and the proof of the previous proposition, we find that for $g \in \mathcal{D}_\mathcal{L}^{(2,0)}$ we have
\begin{align*}
\int_a^b \int_{[a,b)} g(\xi) \bm{\nu}_{z,y}(d\xi)\, p(t,x,z)\, r(z) dz & = \int_a^b \int_{[\sigma^2,\infty)\!} w_\lambda(z) \, w_\lambda(y) \, (\mathcal{F}g)(\lambda) \, \bm{\rho}_{\mathcal{L}}(d\lambda) \, p(t,x,z)\, r(z) dz \\
& = \int_{[\sigma^2,\infty)\!} e^{-t\lambda\,} w_\lambda(x) \, w_\lambda(y) \, (\mathcal{F}g)(\lambda) \, \bm{\rho}_{\mathcal{L}}(d\lambda) \\
& = \int_a^b g(\xi) \, q_t(x,y,\xi)\, r(\xi) d\xi,
\end{align*}
hence the measures $\bm{\nu}_{t,x,y}(d\xi)$ and $\int_a^b \bm{\nu}_{z,y}(d\xi)\, p(t,x,z)\, r(z) dz$ are the same.
\end{proof}

\section{Generalized convolutions and hypergroups} \label{sec:genconv_hypergr}

\subsection{The convolution measure algebra} \label{sub:convmeasalgebra}

As usual in the theory of generalized convolutions, we define the convolution $*: \mathcal{M}_\mathbb{C}[a,b) \times \mathcal{M}_\mathbb{C}[a,b) \longrightarrow \mathcal{M}_\mathbb{C}[a,b)$ as the natural extension of the mapping $(x,y) \mapsto \delta_x * \delta_y := \bm{\nu}_{x,y}$, where $\bm{\nu}_{x,y}$ is the measure of the product formula \eqref{eq:shypPDE_prodform_weaklim}:

\begin{definition} \label{def:shypPDEconv_def}
Let $\mu, \nu \in \mathcal{M}_\mathbb{C}[a,b)$. The complex measure
\[
(\mu * \nu)(d\xi) = \int_{[a,b)} \int_{[a,b)} \bm{\nu}_{x,y}(d\xi) \, \mu(dx) \, \nu(dy)
\]
is called the \emph{$\mathcal{L}$-convolution} of the measures $\mu$ and $\nu$.
\end{definition}

The key tool for studying the properties of the $\mathcal{L}$-convolution is the extension of the $\mathcal{L}$-transform \eqref{eq:entrpf_Ffourierdef} to complex measures, defined by
\[
\widehat{\mu}(\lambda) := \int_{[a,b)} w_\lambda(x)\, \mu(dx), \qquad \lambda \geq 0.
\]

It is immediate from Lemmas \ref{lem:entrpf_ode_wsol} and \ref{lem:shypPDE_wsolbound} that $|\widehat{\mu}(\lambda)| \leq \widehat{\mu}(0) = \| \mu \|$ for all $\mu \in \mathcal{M}_+[a,b)$. In addition, this transformation has various properties which resemble those of the Fourier transform of complex measures:

\begin{proposition} \label{prop:shypPDE_Lfourmeas_props}
The $\mathcal{L}$-transform $\widehat{\mu}$ of $\mu \in \mathcal{M}_\mathbb{C}[a,b)$ has the following properties:
\begin{enumerate}[itemsep=0pt,topsep=4pt]

\item[\textbf{(i)}] $\widehat{\mu}$ is continuous on $[0,\infty)$. Moreover, if a family of measures $\{\mu_j\} \subset \mathcal{M}_\mathbb{C}[a,b)$ is tight and uniformly bounded, then $\{\widehat{\mu_j}\}$ is equicontinuous on $[0,\infty)$.

\item[\textbf{(ii)}] Each measure $\mu \in \mathcal{M}_\mathbb{C}[a,b)$ is uniquely determined by $\widehat{\mu}\restrict{[\sigma^2, \infty)}$.

\item[\textbf{(iii)}]
If $\{\mu_n\}$ is a sequence of measures belonging to $\mathcal{M}_+[a,b)$, $\mu \in \mathcal{M}_+[a,b)$, and $\mu_n \warrow \mu$, then
\[
\widehat{\mu_n} \xrightarrow[\,n \to \infty\,]{} \widehat{\mu} \qquad \text{uniformly for } \lambda \text{ in compact sets.}
\]

\item[\textbf{(iv)}] Suppose that $\lim_{x \uparrow b} w_\lambda(x) = 0$ for all $\lambda > 0$. If $\{\mu_n\}$ is a sequence of measures belonging to $\mathcal{M}_+[a,b)$ whose $\mathcal{L}$-transforms are such that
\begin{equation} \label{eq:shypPDE_Gfour_continuity_hyp}
\widehat{\mu_n}(\lambda) \xrightarrow[\,n \to \infty\,]{} f(\lambda) \qquad \text{pointwise in } \lambda \geq 0
\end{equation}
for some real-valued function $f$ which is continuous at a neighborhood of zero, then $\mu_n \warrow \mu$ for some measure $\mu \in \mathcal{M}_+[a,b)$ such that $\widehat{\mu} \equiv f$.
\end{enumerate}
\end{proposition}

\begin{proof}
\textbf{\emph{(i)}} Let us prove the second statement, which implies the first. Set $C = \sup_j \|\mu_j\|$. Fix $\lambda_0 \geq 0$ and $\eps > 0$. By the tightness assumption, we can choose $\beta \in (a,b)$ such that $|\mu_j|(\beta,b) < \eps$ for all $j$. Since the family $\{w_{(\cdot)}(x)\}_{x \in (a, \beta]}$ is equicontinuous on $[0,\infty)$ (this follows easily from the power series representation of $w_{(\cdot)}(x)$, cf.\ proof of Lemma \ref{lem:entrpf_ode_wsol}), we can choose $\delta > 0$ such that
\[
|\lambda - \lambda_0| < \delta \quad \implies \quad |w_\lambda(x) - w_{\lambda_0}(x)| < \eps \; \text{ for all } a < x \leq \beta.
\]
Consequently,
\begin{align*}
& \bigl|\widehat{\mu_j}(\lambda) - \widehat{\mu_j}(\lambda_0)\bigr| = \biggl| \int_{(a,b)} \bigl(w_\lambda(x) - w_{\lambda_0}(x)\bigr) \mu_j(dx) \biggr| \\
& \quad \leq\int_{(\beta,b)\!} \bigl|w_\lambda(x) - w_{\lambda_0}(x)\bigr| |\mu_j|(dx) + \int_{(a,\beta]\!} \bigl|w_\lambda(x) - w_{\lambda_0}(x)\bigr| |\mu_j|(dx) \\
& \quad \leq 2\eps + C\eps = (C+2)\eps
\end{align*}
for all $j$, provided that $|\lambda - \lambda_0| < \delta$, which means that $\{\widehat{\mu_j}\}$ is equicontinuous at $\lambda_0$. \\[-8pt]

\textbf{\emph{(ii)}} Let $\mu \in \mathcal{M}_\mathbb{C}[a,b)$ be such that $\widehat{\mu}(\lambda) = 0$ for all $\lambda \geq 0$. We need to show that $\mu$ is the zero measure. For each $g \in \mathcal{D}_\mathcal{L}^{(2,0)}$, by \eqref{eq:shypPDE_prodformmeas_charact} we have
\[
\int_{[a,b)} g \: d(\delta_x * \mu) = \int_{[\sigma^2,\infty)} (\mathcal{F} g)(\lambda) \, w_\lambda(x) \, \widehat{\mu}(\lambda) \, \bm{\rho}_{\mathcal{L}}(d\lambda) \, = \, 0.
\]
Since $g \in \mathcal{D}_\mathcal{L}^{(2,0)}$, by Lemma \ref{lem:shypPDE_Lfourier_D2prop} and dominated convergence we have $\lim_{x \downarrow a} \, \int_{[a,b)} g \: d\bm{\nu}_{x,y} = g(y)$ for $y \geq a$; therefore, again by dominated convergence,
\[
0 = \lim_{x \downarrow a} \int_{[a,b)} g \: d(\delta_x * \mu) = \lim_{x \downarrow a} \int_{[a,b)\!} \biggl( \int_{[a,b)} g \: d\bm{\nu}_{x,y} \biggr) \mu(dy) \\
= \int_{[a,b)\!} g(y)\, \mu(dy)
\]
This shows that $\int_{[a,b)\!} g(y)\, \mu(dy) = 0$ for all $g \in \mathcal{D}_\mathcal{L}^{(2,0)}$ and, consequently, $\mu$ is the zero measure. \\[-8pt]

\textbf{\emph{(iii)}} Since $w_\lambda(\cdot)$ is continuous and bounded, the pointwise convergence $\widehat{\mu_n}(\lambda) \to \widehat{\mu}(\lambda)$ follows from the definition of weak convergence of measures. By Prokhorov's theorem \cite[Theorem 8.6.2]{bogachev2007}, $\{\mu_n\}$ is tight and uniformly bounded, thus (by part (i)) $\{\widehat{\mu_n}\}$ is equicontinuous on $[0,\infty)$. Invoking \cite[Lemma 15.22]{klenke2014}, we conclude that the convergence $\widehat{\mu_n} \to \widehat{\mu}$ is uniform for $\lambda$ in compact sets. \\[-8pt]

\textbf{\emph{(iv)}} We only need to show that the sequence $\{\mu_n\}$ is tight and uniformly bounded. (Recall that a family $\{\mu_j\} \subset \mathcal{M}_\mathbb{C}[a,b)$ is said to be uniformly bounded if $\sup_j\|\mu_j\| < \infty$, and $\{\mu_j\}$ is said to be tight if for each $\eps > 0$ there exists a compact $K_\eps \subset [a,b)$ such that $\sup_j \,|\mu_j|([a,b) \setminus K_\eps) < \eps$; these definitions are taken from \cite{bogachev2007}.) Indeed, if $\{\mu_n\}$ is tight and uniformly bounded, then Prokhorov's theorem yields that for any subsequence $\{\mu_{n_k}\}$ there exists a further subsequence $\{\mu_{n_{k_j}}\!\}$ and a measure $\mu \in \mathcal{M}_+[a,b)$ such that $\mu_{n_{k_j}}\!\! \warrow \mu$. Then, due to part (iii) and to \eqref{eq:shypPDE_Gfour_continuity_hyp}, we have $\widehat{\mu}(\lambda) = f(\lambda)$ for all $\lambda \geq 0$, which implies (by part (ii)) that all such subsequences have the same weak limit; consequently, the sequence $\mu_n$ itself converges weakly to $\mu$.

The uniform boundedness of $\{\mu_n\}$ follows immediately from the fact that $\widehat{\mu_n}(0) = \mu_n[a,b)$ converges. To prove the tightness, take $\eps > 0$. Since $f$ is continuous at a neighborhood of zero, we have ${1 \over \delta} \int_0^{2\delta} \bigl(f(0) - f(\lambda)\bigr)d\lambda \longrightarrow 0$ as $\delta \downarrow 0$; therefore, we can choose $\delta > 0$ such that
\[
\biggl| {1 \over \delta} \int_0^{2\delta} \bigl(f(0) - f(\lambda)\bigr)d\lambda \biggr| < \eps.
\] 
Next we observe that, due to the assumption that $\lim_{x \uparrow b} w_\lambda(x) = 0$ for all $\lambda > 0$, we have $\int_0^{2\delta} \bigl( 1-w_\lambda(x) \bigr) d\lambda \longrightarrow 2\delta$ as $x \uparrow b$, meaning that we can pick $\beta \in (a,b)$ such that
\[
\int_0^{2\delta} \bigl( 1-w_\lambda(x) \bigr) d\lambda \geq \delta \qquad \text{for all } \beta < x < b.
\]
By our choice of $\beta$ and Fubini's theorem,
\begin{align*}
\mu_n\bigl[\beta,b) & = {1 \over \delta} \int_{[\beta,b)} \delta\, \mu_n(dx) \\
& \leq {1 \over \delta} \int_{[\beta,b)} \int_0^{2\delta} \bigl( 1-w_\lambda(x) \bigr) d\lambda\, \mu_n(dx) \\
& \leq {1 \over \delta} \int_{[a,b)} \int_0^{2\delta} \bigl( 1-w_\lambda(x) \bigr) d\lambda\, \mu_n(dx) \\
& = {1 \over \delta} \int_0^{2\delta} \bigl(\widehat{\mu_n}(0) - \widehat{\mu_n}(\lambda)\bigr) d\lambda.
\end{align*}
Hence, using the dominated convergence theorem,
\begin{align*}
\limsup_{n \to \infty} \mu_n[\beta,b) & \leq {1 \over \delta} \limsup_{n \to \infty}\! \int_0^{2\delta} \bigl(\widehat{\mu_n}(0) - \widehat{\mu_n}(\lambda)\bigr) d\lambda \\
& = {1 \over \delta} \int_0^{2\delta}\!\! \lim_{n \to \infty} \bigl(\widehat{\mu_n}(0) - \widehat{\mu_n}(\lambda)\bigr) d\lambda = {1 \over \delta} \int_0^{2\delta} \bigl(f(0) -
f(\lambda)\bigr) d\lambda < \eps
\end{align*}
due to the choice of $\delta$. Since $\eps$ is arbitrary, we conclude that $\{\mu_n\}$ is tight, as desired.
\end{proof}

An unsurprising consequence of the construction of the $\mathcal{L}$-convolution is that it is trivialized by the $\mathcal{L}$-transform of measures. Indeed:

\begin{proposition} \label{prop:shypPDE_Ltransf_convtrivial}
Let $\mu, \nu, \pi \in \mathcal{M}_\mathbb{C}[a,b)$. We have $\pi = \mu * \nu$ if and only if
\[
\widehat{\pi}(\lambda) = \widehat{\mu}(\lambda)\, \widehat{\nu}(\lambda) \qquad \text{for all } \lambda \geq 0.
\]
\end{proposition}

\begin{proof}
Using the product formula \eqref{eq:shypPDE_prodform_weaklim}, we compute
\begin{align*}
\widehat{\mu * \nu}(\lambda) & = \int_{[a,b)\!\!} w_\lambda(x) \, (\mu * \nu)(dx) \\
& = \int_{[a,b)\!} \int_{[a,b)\!} \int_{[a,b)\!} w_\lambda(\xi)\, \bm{\nu}_{x,y}(d\xi) \, \mu(dx) \nu(dy)\\
& = \int_{[a,b)\!} \int_{[a,b)\!\!} w_\lambda(x) \, w_\lambda(y) \, \mu(dx) \nu(dy) \: = \: \widehat{\mu}(\lambda) \, \widehat{\nu}(\lambda), \qquad\;\; \lambda \geq 0.
\end{align*}
This proves the ``only if" part, and the converse follows from the uniqueness property in Proposition \ref{prop:shypPDE_Lfourmeas_props}(ii).
\end{proof}

The next result summarizes the properties of the measure algebra determined by the $\mathcal{L}$-convolution:

\begin{proposition} \label{prop:shypPDE_conv_Mbanachalg}
The space $(\mathcal{M}_\mathbb{C}[a,b),*)$, equipped with the total variation norm, is a commutative Banach algebra over $\mathbb{C}$ whose identity element is the Dirac measure $\delta_a$. The subset $\mathcal{P}[a,b)$ is closed under the $\mathcal{L}$-convolution. Moreover, the map $(\mu,\nu) \mapsto \mu*\nu$ is continuous (in the weak topology) from $\mathcal{M}_\mathbb{C}[a,b) \times \mathcal{M}_\mathbb{C}[a,b)$	to $\mathcal{M}_\mathbb{C}[a,b)$.
\end{proposition}

\begin{proof}
Since $\widehat{\mu * \nu} = \widehat{\mu} \ccdot \widehat{\nu}$ (Proposition \ref{prop:shypPDE_Ltransf_convtrivial}), the commutativity, associativity and bilinearity of the $\mathcal{L}$-convolution follow at once from the uniqueness property of the $\mathcal{L}$-transform (Proposition \ref{prop:shypPDE_Lfourmeas_props}(ii)). One can verify directly from the definition of the $\mathcal{L}$-convolution that the submultiplicativity property $\|\mu * \nu\| \leq \|\mu\| \ccdot \|\nu\|$ holds, and that equality holds whenever $\mu, \nu \in \mathcal{M}_+[a,b)$; it is also clear that the convolution of positive measures is a positive measure. We conclude that the Banach algebra property holds and that $\mathcal{P}[a,b)$ is closed under convolution.

If $\lim_{x \uparrow b} w_\lambda(x) = 0$ for all $\lambda > 0$, the identity $\widehat{\bm{\nu}_{x,y}}(\lambda) = w_\lambda(x) w_\lambda(y)$ implies (by Proposition \ref{prop:shypPDE_Lfourmeas_props}(iv)) that $(x,y) \mapsto \bm{\nu}_{x,y}$ is continuous in the weak topology. If the functions $w_\lambda(x)$ do not vanish at the limit $x \uparrow b$, let $\kappa < 0$ be arbitrary and let $h \in \mathrm{C}_\mathrm{b}[a,b)$. Since $w_\kappa$ is increasing and unbounded, ${h \over w_\kappa} \in \mathrm{C}_0[a,b)$. By Lemma \ref{lem:shypPDE_modification}, the map $(x,y) \mapsto \bm{\nu}_{x,y}^{\langle\kappa\rangle\!}$ (where $\bm{\nu}_{x,y}^{\langle\kappa\rangle\!}$ is the measure defined in the proof of Theorem \ref{thm:shypPDE_prodform_weaklim_full}) is continuous, hence
\[
(x,y) \: \longmapsto \, \int_{[a,b)\!} {h(\xi) \over w_\kappa(\xi)} \bm{\nu}_{x,y}^{\langle\kappa\rangle\!}(d\xi) = {1 \over w_\kappa(x) w_\kappa(y)} \int_{[a,b)\!} h(\xi) \, \bm{\nu}_{x,y}(d\xi)
\]
is continuous. This shows that $(x,y) \mapsto \int_{[a,b)\!} h(\xi) \, \bm{\nu}_{x,y}(d\xi)$ is continuous for all $h \in \mathrm{C}_\mathrm{b}[a,b)$ and therefore $(x,y) \mapsto \bm{\nu}_{x,y}$ is continuous in the weak topology. Finally, for $h \in \mathrm{C}_\mathrm{b}[a,b)$ and $\mu_n, \nu_n \in \mathcal{M}_\mathbb{C}[a,b)$ with $\mu_n \warrow \mu$ and $\nu_n \warrow \nu$ we have
\begin{align*}
\lim_n \int_{[a,b)} h(\xi) (\mu_n * \nu_n)(d\xi) & = \lim_n \int_{[a,b)\!} \int_{[a,b)\!} \biggl( \int_{[a,b)\!} h\; d\bm{\nu}_{x,y} \biggr) \mu_n(dx) \nu_n(dy) \\
& = \int_{[a,b)\!} \int_{[a,b)\!} \biggl( \int_{[a,b)\!} h\; d\bm{\nu}_{x,y} \biggr) \mu(dx) \nu(dy) \\
& = \int_{[a,b)} h(\xi) (\mu * \nu)(d\xi)
\end{align*}
due to the continuity of the function in parenthesis; this proves that $(\mu,\nu) \mapsto \mu*\nu$ is continuous.
\end{proof}

\subsection{Back to the associated hyperbolic Cauchy problem}

In this subsection our goal is to show, using the weak continuity of the $\mathcal{L}$-convolution (proved in Proposition \ref{prop:shypPDE_conv_Mbanachalg}), that the existence and uniqueness theorem for the associated hyperbolic Cauchy problem is also valid for initial conditions $h \in \mathcal{D}_\mathcal{L}^{(2,0)}$. To this end, we state a lemma which gives the boundedness property of the convolution of measures regarded as an operator on the spaces $L_p(r)$:

\begin{lemma} \label{lem:shypPDEwl_gentransl_Lpcont}
Let $1 \leq p \leq \infty$ and $\mu \in \mathcal{M}_+[a,b)$. Then the integral 
\begin{equation} \label{eq:shypPDEtransl_def}
\mathcal{T}^\mu h(x) := \int_{[a,b)} h \, d(\delta_x * \mu)
\end{equation}
is, for each $h \in L_p(r)$, a Borel measurable function of $x$, and we have
\begin{equation} \label{eq:shypPDEwl_gentransl_Lpcont}
\|\mathcal{T}^\mu h\|_p \leq \|\mu\| \ccdot \|h\|_p \qquad \text{ for all } h \in L_p(r)
\end{equation}
(consequently, $\mathcal{T}^\mu\bigl(L_p(r)\bigr) \subset L_p(r)$).
\end{lemma}

\begin{proof}
It suffices to prove the result for nonnegative $h$. The map $\nu \mapsto \mu * \nu$ is weakly continuous (Proposition \ref{prop:shypPDE_conv_Mbanachalg}) and takes $\mathcal{M}_+[a,b)$ into itself. According to \cite[Section 2.3]{jewett1975}, this implies that, for each Borel measurable $h \geq 0$, the function $x \mapsto (\mathcal{T}^\mu h)(x)$ is Borel measurable. It follows that $\int_{[a,b)} g(x) (\mu * r)(dx) := \int_a^b (\mathcal{T}^\mu g)(x) r(x) dx$ ($g \in \mathrm{C}_\mathrm{c}[a,b)$) defines a positive Borel measure. For $a \leq c_1 < c_2 < b$, let $\mathds{1}_{[c_1,c_2)}$ be the indicator function of $[c_1,c_2)$, let $h_n \in \mathcal{D}_\mathcal{L}^{(2,0)}$ be a sequence of nonnegative functions such that $h_n \to \mathds{1}_{[c_1,c_2)}$ pointwise, and write $\mathfrak{C} = \{g \in \mathrm{C}_\mathrm{c}^\infty(a,b) \mid 0 \leq g \leq 1\}$. We compute \vspace{-2pt}
\begin{align*}
(\mu * r)[c_1,c_2) & = \lim_n \int_{[a,b)} h_n(x) (\mu * r)(dx) \\
& = \lim_n \sup_{g \in \mathfrak{C}} \int_a^b (\mathcal{T}^\mu h_n)(x) \, g(x) \, r(x) dx \\
& = \lim_n \sup_{g \in \mathfrak{C}} \int_{[\sigma^2,\infty)} \! (\mathcal{F}h_n)(\lambda) \, (\mathcal{F}g)(\lambda) \, \widehat{\mu}(\lambda) \, \bm{\rho}_{\mathcal{L}}(d\lambda) \\
& = \lim_n \sup_{g \in \mathfrak{C}} \int_a^b h_n(x) \, (\mathcal{T}^\mu g)(x) \, r(x) dx \\
& \leq \|\mu\| \ccdot \lim_n \int_a^b h_n(x) \, r(x) dx \, = \, \|\mu\| \ccdot \int_{c_1}^{c_2} r(x)dx
\end{align*}
where the third and fourth equalities follow from \eqref{eq:shypPDE_prodformmeas_charact} and the isometric property of the $\mathcal{L}$-transform (Proposition \ref{prop:entrpf_Ffourier}), and the inequality holds because $\|\mathcal{T}^\mu g\|_\infty \leq \|\mu\| \ccdot \| g \|_\infty \leq \|\mu\|$. Therefore, $\|\mathcal{T}^\mu h\|_1 = \|h\|_{L_1([a,b),\mu * r)} \leq \|\mu\| \ccdot \|h\|_1$ for each Borel measurable $h \geq 0$. Since $\delta_x * \mu \in \mathcal{M}_+[a,b)$, H\"older's inequality yields that $\|\mathcal{T}^\mu h\|_p \leq \|\mu\|^{1/q} \ccdot \|\mathcal{T}^\mu |h|^p\|_1^{1/p} \leq \|\mu\| \ccdot \| h \|_p$ for $1 < p < \infty$.

Finally, if $h \in L_\infty(r)$, $h \geq 0$ then $h = h_\mb{b} + h_\mb{0}$, where $0 \leq h_\mb{b} \leq \|h\|_\infty$ and $h_\mb{0} = 0$ Lebesgue-almost everywhere. Since $\|\mathcal{T}^\mu h_\mb{0}\|_1 \leq \|\mu\| \ccdot \|h_\mb{0}\|_1 = 0$, we have $\mathcal{T}^y h_\mb{0} = 0$ Lebesgue-almost everywhere, and therefore $\|\mathcal{T}^y h\|_\infty = \|\mathcal{T}^y h_\mb{b}\|_\infty \leq \|\mu\| \ccdot \|h\|_\infty$.
\end{proof}

The bounded operator $\mathcal{T}^\mu: L_p(r) \longrightarrow L_p(r)$ is usually called the \emph{$\mathcal{L}$-translation} by the measure $\mu \in \mathcal{M}_+[a,b)$. (When $\mu = \delta_x$ is a Dirac measure, for simplicity we write $\mathcal{T}^x$ instead of $\mathcal{T}^{\delta_x}$.) As noted during the above proof, it follows from \eqref{eq:shypPDE_prodformmeas_charact} that for $g \in \mathcal{D}_\mathcal{L}^{(2,0)}$ we have
\begin{equation} \label{eq:shypPDE_transl_Fident}
\mathcal{F}(\mathcal{T}^\mu g)(\lambda) = \widehat{\mu}(\lambda) \, (\mathcal{F}g)(\lambda) \qquad \text{for } \bm{\rho}_\mathcal{L} \text{-almost every } \lambda.
\end{equation}
A consequence of Lemma \ref{lem:shypPDEwl_gentransl_Lpcont} is that this equality extends to all $g \in L_2(r)$ (this follows from the usual continuity argument, taking into account that the operators $\mathcal{T}^\mu$ and $\mathcal{F}$ are bounded on $L_2(r)$).

\begin{proposition}[Existence and uniqueness of solution with initial condition $h \in \mathcal{D}_\mathcal{L}^{(2,0)}$] \label{prop:shypPDE_Lexistuniq_ext}
If $h \in \mathcal{D}_\mathcal{L}^{(2,0)}\!$, then there exists a unique solution $f \in \mathrm{C}^2 \bigl((a,b)^2\bigr)$ of the Cauchy problem \eqref{eq:shypPDE_Lcauchy} satisfying conditions (i)--(ii) in Theorem \ref{thm:shypPDE_Lexistuniq}, and this unique solution is given by \eqref{eq:shypPDE_Lexistence}.
\end{proposition}

\begin{proof}
The fact that there exists at most one solution of \eqref{eq:shypPDE_Lcauchy} satisfying conditions \emph{(i)--(ii)} is proved in the same way.

Let $h \in \mathcal{D}_\mathcal{L}^{(2,0)}\!$ and consider the function $f(x,y)$ defined by \eqref{eq:shypPDE_Lexistence}. The limit $\lim_{y \downarrow a} f(x,y) = h(x)$ follows from Lemma \ref{lem:shypPDE_Lfourier_D2prop}(b) and dominated convergence. Similarly, we have
\[
\lim_{y \downarrow a} \partial_y^{[1]} f(x,y) = \lim_{y \downarrow a} \int_{[0,\infty)} (\mathcal{F} h)(\lambda) \, w_\lambda(x) \, w_\lambda^{[1]}(y)\, \bm{\rho}_{\mathcal{L}}(d\lambda) = 0
\]
(the absolute and uniform convergence of the differentiated integral justifies the differentiation under the integral sign). Now fix $y \in (a,b)$. By \eqref{eq:shypPDE_prodformmeas_charact}, we have $f(\cdot,y) = \mathcal{T}^y h$. Using the identities \eqref{eq:shypPDE_Ltransfidentity} and \eqref{eq:shypPDE_transl_Fident}, we get
\[
\mathcal{F}(\ell_x (\mathcal{T}^y h))(\lambda) = \lambda \, \mathcal{F}(\mathcal{T}^y h)(\lambda) = \lambda \, w_\lambda(y) \, (\mathcal{F}h)(\lambda) = w_\lambda(y) \, \mathcal{F}(\ell(h))(\lambda) = \mathcal{F}(\mathcal{T}^y \ell(h))(\lambda),
\]
hence $\ell_x (\mathcal{T}^y h)(x) = (\mathcal{T}^y \ell(h))(x)$ for almost every $x$. Since (by the weak continuity of $(x,y) \mapsto \bm{\nu}_{x,y}$, see Proposition \ref{prop:shypPDE_conv_Mbanachalg}) $(x,y) \mapsto (\mathcal{T}^y\ell(h))(x)$ is continuous, it follows that
\[
\ell_x f(x,y) = (\mathcal{T}^y \ell(h))(x), \qquad \text{for all } x,y \in (a,b).
\]
Exactly the same reasoning shows that $\ell_y f(x,y) = (\mathcal{T}^y \ell(h))(x)$, hence $f \in \mathrm{C}^2 \bigl((a,b)^2\bigr)$ is a solution of $\ell_x u = \ell_y u$. 

It remains to check that conditions \emph{(i)--(ii)} hold. As seen above we have $\mathcal{F}(f(\cdot,y))(\lambda) = w_\lambda(y) \, (\mathcal{F}h)(\lambda)$ and $\mathcal{F}[\ell_y f(\cdot,y)](\lambda) = \mathcal{F}[\mathcal{T}^y \ell(h)](\lambda) = \lambda \, w_\lambda(y) \, (\mathcal{F}h)(\lambda)$, hence condition \emph{(ii)} holds. Moreover, it is immediate from \eqref{eq:shypPDE_LtransfidentD2} that $f(\cdot,y) \in \mathcal{D}_\mathcal{L}^{(2)}$, and therefore condition \emph{(i)} holds.
\end{proof}

\subsection{The nondegenerate case: Sturm-Liouville hypergroups} \label{sub:SLhyp_nondegen}

The goal of this section is to determine sufficient conditions in order that the $\mathcal{L}$-convolution defines a hypergroup structure on the interval $[a,b)$.

Let us recall the definition of a hypergroup, which was introduced in \cite{jewett1975} (see also \cite{bloomheyer1994}). Let $K$ be a locally compact space and $*$ a bilinear operator on $\mathcal{M}_\mathbb{C}(K)$. The pair $(K,*)$ is said to be a \emph{hypergroup} if the following axioms are satisfied:
\begin{enumerate}[itemsep=0pt,topsep=4pt]
\item[\textbf{H1.}] If $\mu, \nu \in \mathcal{P}(K)$, then $\mu * \nu \in \mathcal{P}(K)$;
\item[\textbf{H2.}] $\mu * (\nu * \pi) = (\mu * \nu) * \pi\,$ for all $\mu, \nu, \pi \in \mathcal{M}_\mathbb{C}(K)$;
\item[\textbf{H3.}] The map $(\mu,\nu) \mapsto \mu*\nu$ is continuous (in the weak topology) from $\mathcal{M}_\mathbb{C}(K) \times \mathcal{M}_\mathbb{C}(K)$	to $\mathcal{M}_\mathbb{C}(K)$;
\item[\textbf{H4.}] There exists an element $\mathrm{e} \in K$ such that $\delta_\mathrm{e} * \mu = \mu * \delta_\mathrm{e} = \mu$ for all $\mu \in \mathcal{M}_\mathbb{C}(K)$;
\item[\textbf{H5.}] There exists a homeomorphism (called \emph{involution}) $x \mapsto \check{x}$ of $K$ onto itself such that $(\check{x})\check{\vrule height1.35ex width0pt \,} = x$ and $\mathrm{e} \in \supp(\delta_x * \delta_y)$ if and only if $y = \check{x}$;
\item[\textbf{H6.}] $(\mu * \nu) \check{\vrule height1.35ex width0pt \,} = \check{\nu} * \check{\mu}$, where $\check{\mu}$ is defined via $\int f(x) \check{\mu}(dx) = \int f(\check{x}) \mu(dx)$;
\item[\textbf{H7.}] $(x,y) \mapsto \supp(\delta_x * \delta_y)$ is continuous from $K \times K$ into the space of compact subsets of $K$ (endowed with the Michael topology, see \cite{jewett1975}).
\end{enumerate}

We saw in Proposition \ref{prop:shypPDE_conv_Mbanachalg} that $\mathcal{L}$-convolution satisfies the axioms H1, H2, H3, H4 and H6 (with $K = [a,b)$ and $\mathrm{e}= a$ as the identity element; H6 holds for the identity involution $\check{x} = x$). In order to verify conditions H5 and H7, one needs to determine the support of $\bm{\nu}_{x,y} = \delta_x * \delta_y$.

A detailed study of $\supp(\bm{\nu}_{x,y})$ was carried out by Zeuner in \cite{zeuner1992}. The next proposition shows that the results of Zeuner can be applied to the $\mathcal{L}$-convolution, provided that the differential operator \eqref{eq:shypPDE_elldiffexpr} has coefficients $p = r = A$ defined on $(0,\infty)$, and there exists $\eta \in \mathrm{C}^1[0,\infty)$ satisfying the conditions given in Assumption \ref{asmp:shypPDE_SLhyperg}.

\begin{proposition} \label{prop:shypPDE_zeunersupp}
Let
\[
\ell = -{1 \over A} {d \over dx} \Bigl(A {d \over dx}\Bigr), \qquad x \in (0,\infty)
\]
where $A(x) > 0$ for all $x \geq 0$. Suppose that there exists $\eta \in \mathrm{C}^1[0,\infty)$ such that $\eta \geq 0$, the functions $\bm{\phi}_\eta$, $\bm{\psi}_\eta$ are both decreasing on $(0,\infty)$ and $\lim_{x \to \infty} \bm{\phi}_\eta(x) = 0$. Let $x_0 = \sup\{x \geq 0 \mid \bm{\psi}_\eta(x) = \bm{\psi}_\eta(0)\}$ and $x_1 = \inf\{x > 0 \mid \bm{\phi}_\eta(x) = 0 \}$. Then:
\begin{enumerate}[itemsep=0pt,topsep=4pt]
\item[\textbf{(a)}] If $x_0 = \infty$, $x_1 = 0$ and $\eta(0) = 0$ then $\supp(\delta_x * \delta_y) = \{|x-y|, x+y\}$ for all $x,y \geq 0$.
\item[\textbf{(b)}] If $0 < x_0 < \infty$, $x_1 = 0$ and $\eta(0) = 0$ then
\[
\supp(\delta_x * \delta_y) = \begin{cases}
\{|x-y|, x+y\}, & x+y \leq x_0 \\
\{|x-y|\} \cup [2x_0 - x - y, x+y], & x,y < x_0 < x + y \\
[|x-y|, x+y], & \max\{x,y\} \geq x_0.
\end{cases}
\]
\item[\textbf{(c)}] If $x_0 = \infty$, $0 < x_1 < \infty$ and $\eta(0) = 0$ then
\[
\supp(\delta_x * \delta_y) = \begin{cases}
[|x-y|, x+y], & \min\{x,y\} \leq 2x_1, \\
[|x-y|, 2x_1 + |x-y|] \cup [x+y-2x_1, x+y], & \min\{x,y\} > 2x_1.
\end{cases}
\]
\item[\textbf{(d)}] If $0 < 3x_1 < x_0 < \infty$ and $\eta(0) = 0$ then
\[
\supp(\delta_x * \delta_y) = \begin{cases}
[|x-y|, x+y], & \min\{x,y\} \leq 2x_1 \text{ or } \max\{x,y\} \geq x_0 - x_1, \\[2pt]
\begin{aligned}
& [|x-y|, 2x_1 + |x-y|] \cup \\[-5pt] & \qquad \cup [x+y-2x_1, x+y],
\end{aligned} & \min\{x,y\} > 2x_1 \text{ and } \max\{x,y\} < x_0 - x_1.
\end{cases}
\]
\item[\textbf{(e)}] If $x_0 \leq 3x_1$ or $\eta(0) > 0$ then $\supp(\delta_x * \delta_y) = [|x-y|, x+y]$ for all $x,y \geq 0$.
\end{enumerate}
\end{proposition}

\begin{proof}
Fix $z \geq 0$, and let $\{h_\eps\} \subset \mathcal{D}_\mathcal{L}^{(2,0)}$ be a family of functions such that
\begin{equation} \label{eq:shypPDE_zeunersupp_pf1}
h_\eps(\xi) > 0 \text{ for } z-\eps < \xi < z+\eps, \qquad h_\eps(\xi) = 0 \text{ for } \xi \leq z-\eps \text{ and } \xi \geq z+\eps.
\end{equation}
Observe that $z \in \supp(\delta_x * \delta_y)$ if and only if $\int_{[0,\infty)} h_\eps \: d(\delta_x * \delta_y) > 0$ for all $\eps > 0$. Now, we know from Proposition \ref{prop:shypPDE_Lexistuniq_ext} that the function
\begin{equation} \label{eq:shypPDE_zeunersupp_pf2}
f_{h_\eps}(x,y) := \int_{[0,\infty)} h_\eps \: d(\delta_x * \delta_y) = \int_{[\sigma^2,\infty)\!} w_\lambda(x) \, w_\lambda(y) \, (\mathcal{F}h_\eps)(\lambda) \, \bm{\rho}_{\mathcal{L}}(d\lambda)
\end{equation}
(the second equality is due to \eqref{eq:shypPDE_prodformmeas_charact}) is a nonnegative solution of the Cauchy problem \eqref{eq:shypPDE_Lcauchy} with $h \equiv h_\eps$; writing $B(x) := \exp({1 \over 2} \int_0^x \eta(\xi) d\xi)$, it follows that $v_{h_\eps}(x,y) = B(x) B(y) f_{h_\eps}(x,y)$ is a solution of $\bm{\ell}_x^B v - \bm{\ell}_y^B v = 0$, $\bm{\ell}_x^B$ being the differential operator defined in Lemma \ref{lem:shypPDE_inteqtriangle}. If we apply this lemma with $c > 0$ and then let $c \downarrow 0$, we deduce that the following integral equation holds:
\begin{equation} \label{eq:shypPDE_zeunersupp_inteq}
\begin{aligned}
A_{\mathsmaller{B}}(x) A_{\mathsmaller{B}}(y) \, v_{h_\eps}(x,y) = H + I_0 + I_1 + I_2 + I_3
\end{aligned}
\end{equation}
where $H = \tfrac{1}{2} A(0) \bigl[{A(x-y) \over B(x-y)} \, h_\eps(x-y) + {A(x+y) \over B(x+y)} \, h_\eps(x+y)\bigr]$, $I_0 = {\eta(0) \over 4} \int_{x-y}^{x+y} {A(s) \over B(s)} h_\eps(s)\, ds$ and $I_1, I_2, I_3$ are given by \eqref{eq:shypPDE_inteqtriangle_I1}--\eqref{eq:shypPDE_inteqtriangle_I3} with $c=0$ and $v = v_{h_\eps}$. Since $h_\eps$ and $f_{h_\eps}$ are nonnegative, all the terms in the right-hand side of \eqref{eq:shypPDE_zeunersupp_inteq} are nonnegative; consequently, we have $z \in \supp(\delta_x * \delta_y)$ if and only if at least one of the terms in the right-hand side of \eqref{eq:shypPDE_zeunersupp_inteq} is strictly positive for all $\eps > 0$. In order to ascertain whether this holds or not, one needs to perform a thorough analysis of the integrals $I_0$, $I_1$, $I_2$ and $I_3$. This has been done by Zeuner in \cite[Proposition 3.9]{zeuner1992}; his results lead to the conclusion stated in the proposition.
\end{proof}

\begin{theorem} \label{thm:shypPDE_hypergroupcase}
Let $\ell$ be a differential expression of the form \eqref{eq:shypPDE_elldiffexpr}. Suppose that $\gamma(a) > -\infty$ and that there exists $\eta \in \mathrm{C}^1[\gamma(a),\infty)$ satisfying the conditions given in Assumption \ref{asmp:shypPDE_SLhyperg}. Then $\bigl([a,b),*\bigr)$ is a hypergroup.
\end{theorem}

\begin{proof}
It was seen in Proposition \ref{prop:shypPDE_conv_Mbanachalg} that the axioms H1, H2, H3, H4 and H6 hold for the $\mathcal{L}$-convolution (with $\check{x} = x$); we need to check that axioms H5 and H7 are also satisfied.

Assume first that $\ell$ satisfies the assumptions of Proposition \ref{prop:shypPDE_zeunersupp}. Then the explicit expressions for $\supp(\delta_x * \delta_y)$ show that (in each of the cases \emph{(a)}--\emph{(e)}) $\supp(\delta_x * \delta_y)$ depends continuously on $(x,y)$ and contains $\mathrm{e} = 0$ if and only if $x=y$, hence axioms H5 and H7 hold. (Verifying the continuity is easy after noting that the topology in the space of compact subsets can be metrized by the Hausdorff metric, cf.\ \cite[Subsection 4.1]{koornwinderschwartz1997}.) 

Now, in the general case of an operator $\ell$ of the form \eqref{eq:shypPDE_elldiffexpr}, note that $\gamma(a) > -\infty$ means that $\smash{\sqrt{r(y) \over p(y)}}$ is integrable near $a$, so that we may assume that $\gamma(a) = 0$ (otherwise, replace the interior point $c$ by the endpoint $a$ in the definition of the function $\gamma$). By hypothesis, the transformed operator $\widetilde{\ell} = -{1 \over A} {d \over d\xi}(A {d \over d\xi})$ defined via \eqref{eq:shypPDE_tildeell_A} satisfies the assumptions of Proposition \ref{prop:shypPDE_zeunersupp}; by the above, the associated convolution, which we denote by $\widetilde{*}$, satisfies axioms H5 and H7. From the product formulas for the solutions $w_\lambda(x)$ and $\widetilde{w}_\lambda(\xi) = w_\lambda(\gamma^{-1}(\xi))$ we deduce that
\[
\int_{[a,b)} w_\lambda \: d(\delta_x * \delta_y) = w_\lambda(x) w_\lambda(y) = \widetilde{w}_\lambda(\gamma(x)) \widetilde{w}_\lambda(\gamma(y)) = \int_{[0,\infty)} w_\lambda(\gamma^{-1}(z)) \bigl(\delta_{\gamma(x)} \kern.1em \widetilde{*} \kern.12em \delta_{\gamma(y)}\bigr)(dz)
\]
and, consequently, $\delta_x * \delta_y = \gamma^{-1}(\delta_{\gamma(x)} \kern.1em \widetilde{*} \kern.12em \delta_{\gamma(y)})$ (the right hand side denoting the pushforward of the measure $\delta_{\gamma(x)} \kern.1em \widetilde{*} \kern.12em \delta_{\gamma(y)}$ under the map $\xi \mapsto \gamma^{-1}(\xi)$). In particular, $\supp(\delta_x * \delta_y) = \gamma^{-1}\bigl(\supp(\delta_{\gamma(x)} \kern.1em \widetilde{*} \kern.12em \delta_{\gamma(y)})\bigr)$; since $\gamma$ and $\gamma^{-1}$ are continuous, we immediately conclude that the convolution $*$ also satisfies H5 and H7.
\end{proof}

A \emph{hypergroup isomorphism} between $(K_1,*)$ and $(K_2,\diamond)$ is an isomorphism between the Banach algebras $(\mathcal{M}_\mathbb{C}(K_1),*)$ and $(\mathcal{M}_\mathbb{C}(K_2),\diamond)$ which preserves involution and point measures \cite[Definition 1.1.3]{bloomheyer1994}. The proof of Theorem \ref{thm:shypPDE_hypergroupcase} shows that the hypergroups $\bigl([a,b),*\bigr)$ and $\bigl([0,\infty),\widetilde{*}\bigr)$ associated with the differential operators $\ell$ and $\widetilde{\ell}$ are isomorphic, the isomorphism being the pushforward map $\mu \in \mathcal{M}_\mathbb{C}[a,b) \longmapsto \gamma^{-1}(\mu) \in \mathcal{M}_\mathbb{C}[0,\infty)$.

Let us write $\mathrm{C}_{\mathrm{c,even}}^\infty := \{ h:[0,\infty) \to \mathbb{C} \mid h \text{ is the restriction of an even } \mathrm{C}_\mathrm{c}^\infty(\mathbb{R})\text{-function} \}$. The next definition was introduced by Zeuner \cite{zeuner1989}:

\begin{definition} \label{def:shypPDE_slhyp}
A hypergroup $([0,\infty), *)$ is said to be a \emph{Sturm-Liouville hypergroup} if there exists a function $A$ on $[0,\infty)$ satisfying the condition
\begin{enumerate}[itemsep=0pt,topsep=4pt]
\item[\textbf{SL0}] $A \in \mathrm{C}[0,\infty) \cap \mathrm{C}^1(0,\infty)$ and $A(x) > 0$ for $x > 0$
\end{enumerate}
such that, for every function $h \in \mathrm{C}_{\mathrm{c,even}}^\infty$, the convolution
\begin{equation} \label{eq:shypPDE_SLhypdef_conv}
v_h(x,y) = \int_{[0,\infty)} h(\xi) (\delta_x * \delta_y)(d\xi)
\end{equation}
belongs to $\mathrm{C}^2\bigl([0,\infty)^2\bigr)$ and satisfies $(\ell_x v_h)(x,y) = (\ell_y v_h)(x,y)$,\, $(\partial_y v_h)(x,0) = 0$\, ($x > 0$), where $\ell_x = -{1 \over A} {\partial \over \partial x} (A(x) {\partial \over \partial x})$.
\end{definition}

A fundamental existence theorem for Sturm-Liouville hypergroups, which was proved by Zeuner \cite[Theorem 3.11]{zeuner1992}, states: \emph{Suppose that $A$ satisfies SL0 and is such that}
\begin{enumerate}[itemsep=0pt,topsep=4pt]
\item[\textbf{SL1}] One of the following assertions holds: \vspace{-1ex}
\begin{enumerate}[itemsep=0pt]
\item[\textbf{SL1.1}] $A(0) = 0$ and ${A'(x) \over A(x)} = {\alpha_0 \over x} + \alpha_1(x)$ for $x$ in a neighbourhood of $0$, where $\alpha_0 > 0$ and $\alpha_1 \in \mathrm{C}^\infty(\mathbb{R})$ is an odd function;
\item[\textbf{SL1.2}] $A(0) > 0$ and $A \in \mathrm{C}^1[0,\infty)$. \vspace{-1ex}
\end{enumerate}
\item[\textbf{SL2}] There exists $\eta \in \mathrm{C}^1[0,\infty)$ such that $\eta \geq 0$, the functions $\bm{\phi}_\eta$, $\bm{\psi}_\eta$ are both decreasing on $(0,\infty)$ and $\lim_{x \to \infty} \bm{\phi}_\eta(x) = 0$ ($\bm{\phi}_\eta$, $\bm{\psi}_\eta$ are defined as in Assumption \ref{asmp:shypPDE_SLhyperg}).
\end{enumerate}
\emph{Define the convolution $*$ via \eqref{eq:shypPDE_SLhypdef_conv} where, for $h \in \mathrm{C}_{\mathrm{c,even}}^\infty$, $v_h$ denotes the unique solution of $\ell_x v_h = \ell_y v_h$, $v_h(x,0) = v_h(0,x) = h(x)$, $(\partial_y v_h)(x,0) = (\partial_x v_h)(0,y) = 0$. Then $\bigl([0,\infty), *\bigr)$ is a Sturm-Liouville hypergroup.}

To the best of our knowledge, this is the most general known result giving sufficient conditions for the existence of a Sturm-Liouville hypergroup on $[0,\infty)$ associated with a given function $A$. In fact, as far as the authors are aware, all the concrete examples of hypergroup structures on $[0,\infty)$ which were known prior to this work are particular cases of Sturm-Liouville hypergroups satisfying conditions SL0, SL1 and SL2 (see \cite{bloomheyer1994,gallardotrimeche2002}). However, we can prove as a corollary of Theorem \ref{thm:shypPDE_hypergroupcase} that an existence theorem very similar to that of Zeuner continues to hold if the condition SL1 is removed:

\begin{corollary} \label{cor:shypPDE_slhypcor}
Suppose that $A$ satisfies SL0 and SL2. For $h \in \mathcal{D}_\mathcal{L}^{(2,0)}$, denote by $v_h$ the unique solution of $\ell_x v_h = \ell_y v_h$,\, $v_h(x,0) = v_h(0,x) = h(x)$,\, $(\partial_y^{[1]} v_h)(x,0) = (\partial_x^{[1]} v_h)(0,y) = 0$ such that conditions (i)--(ii) in Theorem \ref{thm:shypPDE_Lexistuniq} hold for $f = v_h$. Define the convolution $*$ via \eqref{eq:shypPDE_SLhypdef_conv}. Then $\bigl([0,\infty), *\bigr)$ is a hypergroup.
\end{corollary}

\begin{proof}
Just notice that, by \eqref{eq:shypPDE_prodformmeas_charact} and Proposition \ref{prop:shypPDE_Lexistuniq_ext}, the definition of convolution given in the statement of the corollary is equivalent to Definition \ref{def:shypPDEconv_def}.
\end{proof}

This corollary shows that it is natural to modify the definition of Sturm-Liouville hypergroup (Definition \ref{def:shypPDE_slhyp}) by replacing the space $\mathrm{C}_{\mathrm{c,even}}^\infty$ by $\mathcal{D}_\mathcal{L}^{(2,0)}$ and replacing $\partial_y$ by $\partial_y^{[1]}$ in the initial condition, because in this way we are able to extend the class of Sturm-Liouville hypergroups to all functions $A$ satisfying conditions SL0 and SL2.

We emphasize that condition SL1 imposes a great restriction on the behavior of the Sturm-Liouville operator $\ell(u) = -u'' - {A' \over A} u'$ near zero: in the singular case $A(0) = 0$, SL1 requires that ${A'(x) \over A(x)} \sim {\alpha_0 \over x}$. Therefore, as shown in the next example, Corollary \ref{cor:shypPDE_slhypcor} leads, in particular, to a considerable extension of the class of singular operators for which an associated hypergroup exists:

\begin{example}
If $A$ satisfies SL0 and the function ${A' \over A}$ is nonnegative and decreasing, then SL2 is satisfied with $\eta := 0$. Therefore, Corollary \ref{cor:shypPDE_slhypcor} assures that there exists a hypergroup associated with the operator $\ell(u) = -u'' - {A' \over A} u'$. Notice that this existence result holds without any restriction on the growth of ${A'(x) \over A(x)}$ as $x \downarrow 0$. 
\end{example}

\subsection{The degenerate case: degenerate hypergroups of full support}

The goal of this subsection is to prove that in the degenerate case $\gamma(a) = -\infty$ the pair $\bigl([a,b),*\bigr)$ is a degenerate hypergroup of full support, in the sense of the following definition:

\begin{definition}
Let $K$ be a locally compact space and $*$ a bilinear operator on $\mathcal{M}_\mathbb{C}(K)$. The pair $(K,*)$ is said to be a \emph{degenerate hypergroup of full support} if conditions H1--H4 and H6 hold, together with the following axiom:
\begin{enumerate}[itemsep=0pt,topsep=4pt]
\item[\textbf{DH.}] $\supp(\delta_x * \delta_y) = K$ for all $x, y \in K \setminus \{\mathrm{e}\}$.
\end{enumerate}
\end{definition}

As we saw in the proof of Proposition \ref{prop:shypPDE_zeunersupp}, in order to determine the support of $\delta_x * \delta_y$ we need to know when the solution of the Cauchy problem \eqref{eq:shypPDE_Lcauchy} is strictly positive. Our first step is to use Lemma \ref{lem:shypPDE_inteqtriangle} in order to derive an integral inequality which will turn out to be useful for studying the strict positivity of solution.

\begin{lemma} \label{lem:shypPDE_sol_intineqpos}
Write $R(x) := {A(x) \over B(x)}$, where $B(x) = \exp({1 \over 2} \int_\beta^x \eta(\xi)d\xi)$ (with $\beta > \gamma(a)$ arbitrary). Take $h \in \mathcal{D}_\mathcal{L}^{(2,0)\!}$ such that $h \geq 0$. Let $u(x,y) := f(\gamma^{-1}(x), \gamma^{-1}(y))$, where $f \in \mathrm{C}^2\bigl((a,b)^2\bigr)$ is the solution \eqref{eq:shypPDE_Lexistence} of the Cauchy problem (cf.\ Proposition \ref{prop:shypPDE_Lexistuniq_ext}). Then the following inequality holds:
\begin{align*}
R(x) R(y) u(x,y) & \geq \tfrac{1}{2} \int_{\gamma(a)}^y R(s) R(x-y+s) \bigl[ \bm{\phi}_\eta(s) + \bm{\phi}_\eta(x-y+s) \bigr] u(x-y+s,s)\, ds \\
& + \tfrac{1}{2} \int_{\gamma(a)}^y R(s) R(x+y-s) \bigl[ \bm{\phi}_\eta(s) - \bm{\phi}_\eta(x+y-s) \bigr] u(x+y-s,s)\, ds \\
& + \tfrac{1}{2} \int_{\Delta} R(\xi) R(\zeta) \bigl[\bm{\psi}_\eta(\zeta) - \bm{\psi}_\eta(\xi)\bigr] u(\xi,\zeta)\, d\xi d\zeta
\end{align*}
where $\Delta \equiv \Delta_{\gamma(a),x,y} = \{(\xi,\zeta) \in \mathbb{R}^2 \mid \zeta \geq \gamma(a), \, \xi + \zeta \leq x+y, \, \xi - \zeta \geq x-y \}$.
\end{lemma}

\begin{proof}
Let $\{a_m\}_{m \in \mathbb{N}}$ be a sequence $b > a_1 > a_2 > \ldots$ with $\lim a_m = a$. For $m \in \mathbb{N}$, define $u_m(x,y) := f_m(\gamma^{-1}(x), \gamma^{-1}(y))$, where $f_m$ is given by \eqref{eq:shypPDE_Lexistence_eps}. The function $v_m(x,y) = B(x) B(y) u_m(x,y)$ is a solution of
\begin{align*}
(\bm{\ell}_x^B v_m)(x,y) = (\bm{\ell}_y^B v_m)(x,y), & \qquad x, y > \tilde{a}_m \\
v_m(x,\tilde{a}_m) = B(x) B(\tilde{a}_m) h(\gamma^{-1}(x)), & \qquad x > \tilde{a}_m \\
(\partial_y v_m)(x,\tilde{a}_m) = \tfrac{1}{2} \eta(\tilde{a}_m) B(x) B(\tilde{a}_m) h(\gamma^{-1}(x)), & \qquad x > \tilde{a}_m
\end{align*}
where $\bm{\ell}^B v := - v'' - \bm{\phi}_\eta v' + \bm{\psi}_\eta v$. Clearly, $v_m(x,\tilde{a}_m),\, (\partial_y v_m)(x,\tilde{a}_m) \geq 0$. By Lemma \ref{lem:shypPDE_inteqtriangle}, the integral equation \eqref{eq:shypPDE_inteqtriangle} holds with $v = v_m$ and $c = a_m$. It is clear that we have $H \geq 0$, $I_0 \geq 0$ and $I_4 = 0$ in the right hand side of \eqref{eq:shypPDE_inteqtriangle}; moreover, it follows from Proposition \ref{prop:hypPDE_sol_positivity_reg} and Assumption \ref{asmp:shypPDE_SLhyperg} that the integrands of $I_1, I_2$ and $I_3$ are nonnegative. Consequently, for $\alpha \in [\tilde{a}_m,y]$ we have
\begin{equation} \label{eq:shypPDE_sol_intineqpos_pf1}
\begin{aligned}
R(x) R(y) u_m(x,y) & \geq \tfrac{1}{2} \int_\alpha^y R(s) R(x-y+s) \bigl[ \bm{\phi}_\eta(s) + \bm{\phi}_\eta(x-y+s) \bigr] u_m(x-y+s,s)\, ds \\
& + \tfrac{1}{2} \int_\alpha^y R(s) R(x+y-s) \bigl[ \bm{\phi}_\eta(s) - \bm{\phi}_\eta(x+y-s) \bigr] u_m(x+y-s,s)\, ds \\
& + \tfrac{1}{2} \int_{\Delta_{\alpha,x,y}} \! R(\xi) R(\zeta) \bigl[\bm{\psi}_\eta(\zeta) - \bm{\psi}_\eta(\xi)\bigr] u_m(\xi,\zeta)\, d\xi d\zeta
\end{aligned}
\end{equation}
where $\Delta_{\alpha,x,y} = \{(\xi,\zeta) \in \mathbb{R}^2 \mid \zeta \geq \alpha, \, \xi + \zeta \leq x+y, \, \xi - \zeta \geq x-y\}$. Since by Proposition \ref{prop:shypPDE_Gexistence_eps} $\lim_{m \to \infty} u_m(x,y) = u(x,y)$ pointwise for $x,y \in (\gamma(a),\infty)$, by taking the limit we deduce that for each fixed $\alpha \in (\gamma(a),y]$ the inequality \eqref{eq:shypPDE_sol_intineqpos_pf1} holds with $u_m$ replaced by $u$. If we then take the limit $\alpha \downarrow \gamma(a)$, the desired integral inequality follows.
\end{proof}

The next lemma will be helpful for verifying the strict positivity of the integrands in the above integral inequality.

\begin{lemma} \label{lem:shypPDE_degen_nonconst}
If $\gamma(a) = -\infty$, then at least one of the functions $\bm{\phi}_\eta$, $\bm{\psi}_\eta$ defined in Assumption \ref{asmp:shypPDE_SLhyperg} is non-constant on every neighbourhood of $-\infty$.
\end{lemma}

\begin{proof}
Suppose by contradiction that $\gamma(a) = -\infty$ and $\bm{\phi}_\eta$, $\bm{\psi}_\eta$ are both constant on an interval $(-\infty,\kappa] \subset \mathbb{R}$. Recall from the proof of Proposition \ref{prop:shypPDE_spectralsupp} that $\mathcal{L}$ is unitarily equivalent to a self-adjoint realization of $-{d^2 \over d\xi^2} + \mathfrak{q}$, where $\mathfrak{q}$ is given by \eqref{eq:shypPDE_liouvtrans_frakq}. Clearly, $\mathfrak{q}(\xi) = \mathfrak{q}_\infty := {1 \over 4} \bm{\phi}_\eta^2(\kappa) + \bm{\psi}_\eta(\kappa) < -\infty$ for all $\xi \in (-\infty,\kappa)$. It therefore follows from \cite[Theorem 15.3]{weidmann1987} that the essential spectrum of any self-adjoint realization of $\ell$ restricted to an interval $(a,c)$ (for $a < c < b$) contains $[\mathfrak{q}_\infty,\infty)$. However, it follows from the boundary condition \eqref{eq:shypPDE_Lop_leftBC} and \cite[Theorem 3.1]{mckean1956} that self-adjoint realizations of $\ell$ restricted to $(a,c)$ have purely discrete spectrum. This contradiction proves the lemma.
\end{proof}

We are now ready to prove that in the case $\gamma(a) = -\infty$ the solution of the (nontrivial) Cauchy problem \eqref{eq:shypPDE_Lcauchy} always has full support on $(a,b)^2$, even when the initial condition is compactly supported:

\begin{theorem}[Strict positivity of solution for the Cauchy problem \eqref{eq:shypPDE_Lcauchy}] \label{thm:shypPDE_degen_strictpos}
Suppose that $\gamma(a) = -\infty$. Take $h \in \mathcal{D}_\mathcal{L}^{(2,0)}$. If $h \geq 0$ and $h(\tau_0) > 0$ for some $\tau_0 \in (a,b)$, then the function $f$ given by \eqref{eq:shypPDE_Lexistence} is such that
\[
f(x,y) > 0 \qquad \text{for } x, y \in (a,b).
\]
\end{theorem}

\begin{proof}
Let $u(x,y) := f(\gamma^{-1}(x), \gamma^{-1}(y))$ and $\tilde{\tau}_0 = \gamma(\tau_0)$. Fix $x_0 \geq y_0 > -\infty$. Since $\lim_{y \to -\infty} u(\tilde{\tau}_0,y) = h(\tau_0) > 0$, there exists $\kappa \in (-\infty,\min\{y_0,\tau_0\})$ such that $u(\tilde{\tau}_0,y) > 0$ for all $y \leq \kappa$.

Suppose $\bm{\phi}_\eta$ is non-constant on every neighbourhood of $-\infty$. Choosing a smaller $\kappa$ if necessary, we may assume that $\bm{\phi}_\eta(\kappa) > \bm{\phi}_\eta(\xi)$ for all $\xi > \kappa$. For each $x > \tilde{\tau}_0$ and $y \leq \kappa$ we have by Lemma \ref{lem:shypPDE_sol_intineqpos}
\[
R(x) R(y) u(x,y) \geq \tfrac{1}{2} \int_{-\infty}^y R(s) R(x-y+s) \bigl[ \bm{\phi}_\eta(s) + \bm{\phi}_\eta(x-y+s) \bigr] u(x-y+s,s)\, ds
\]
and the integrand in the right hand side is continuous and strictly positive at $s = y - x + \tilde{\tau}_0$, so the integral is positive and therefore $u(x,y) > 0$ for all $x \geq \tilde{\tau}_0$ and $y \leq \kappa$. Again by Lemma \ref{lem:shypPDE_sol_intineqpos},
\[
R(x_0) R(y_0) u(x_0,y_0) \geq \tfrac{1}{2} \int_{-\infty}^{y_0} R(s) R(x_0+y_0-s) \bigl[ \bm{\phi}_\eta(s) - \bm{\phi}_\eta(x_0+y_0-s) \bigr] u(x_0+y_0-s,s)\, ds
\]
with the integrand being strictly positive for $s < \min\{\kappa, x_0 + y_0 - \tilde{\tau}_0\}$, thus $u(x_0,y_0) > 0$. 

Suppose now that $\bm{\psi}_\eta$ is non-constant on every neighbourhood of $-\infty$ and that $\kappa$ is chosen such that $\bm{\psi}_\eta(\kappa) > \bm{\psi}_\eta(\xi)$ for all $\xi > \kappa$. The integral inequality of Lemma \ref{lem:shypPDE_sol_intineqpos} yields
\[
R(x_0) R(y_0) u(x_0,y_0) \geq \tfrac{1}{2} \int_{\Delta} R(\xi) R(\zeta) \bigl[\bm{\psi}_\eta(\zeta) - \bm{\psi}_\eta(\xi)\bigr] u(\xi,\zeta)\, d\xi d\zeta.
\]
where $\Delta = \{(\xi,\zeta) \in \mathbb{R}^2 \mid \xi + \zeta \leq x_0+y_0, \, \xi - \zeta \geq x_0-y_0 \}$. Clearly, the integrand is continuous and $> 0$ on $\{(\tau_0,\zeta) \mid \zeta \leq \min(y_0 - |x_0 - \tau_0|,\, \kappa) \} \subset \Delta$, and it follows at once that $u(x_0,y_0) > 0$.

By Lemma \ref{lem:shypPDE_degen_nonconst} it follows that $u(x_0,y_0) > 0$. Since $x_0 \geq y_0 > -\infty$ are arbitrary we conclude that $f(x,y) > 0$ for $b > x \geq y > a$ and, by symmetry, for $x,y \in (a,b)$. 
\end{proof}

\begin{corollary}[Existence theorem for degenerate hypergroups of full support] \label{cor:shypPDE_deghypfullsupp}
Let $\ell$ be a differential expression of the form \eqref{eq:shypPDE_elldiffexpr} and satisfying \eqref{eq:shypPDE_Lop_leftBC}. Suppose that $\gamma(a) = -\infty$. Then $\bigl([a,b),*\bigr)$ is a degenerate hypergroup of full support.
\end{corollary}

\begin{proof}
By Proposition \ref{prop:shypPDE_conv_Mbanachalg}, the pair $\bigl([a,b),*\bigr)$ satisfies axioms H1--H4 and H6. As in the proof of Proposition \ref{prop:shypPDE_zeunersupp}, $z \in [a,b)$ belongs to $\supp(\delta_x * \delta_y)$ if and only if $\int_{[\sigma^2,\infty)\!} w_\lambda(x) \, w_\lambda(y) \, (\mathcal{F}h_\eps)(\lambda) \, \bm{\rho}_{\mathcal{L}}(d\lambda) > 0$ for all $\eps > 0$, where $\{h_\eps\} \subset \mathcal{D}_\mathcal{L}^{(2,0)}$ is a family of functions satisfying \eqref{eq:shypPDE_zeunersupp_pf1}. But it follows from Theorem \ref{thm:shypPDE_degen_strictpos} that $f_{h_\eps}(x,y) = \int_{[\sigma^2,\infty)\!} w_\lambda(x) \, w_\lambda(y) \, (\mathcal{F}h_\eps)(\lambda) \, \bm{\rho}_{\mathcal{L}}(d\lambda) > 0$ for all $x,y \in (a,b)$. Hence each $z \in [a,b)$ belongs to all the sets $\supp(\delta_x * \delta_y)$, $x,y \in (a,b)$; therefore, $\bigl([a,b),*\bigr)$ satisfies axiom DH.
\end{proof}

As discussed in the Introduction, the notion of degenerate hypergroup of full support is motivated by the example of the so-called Whittaker convolution, associated with the normalized Whittaker differential operator $\ell = - x^2 {d^2 \over dx^2} - (1+2(1-\alpha)x) {d \over dx}$ and studied by the authors in \cite{sousaetal2018a,sousaetal2018b}. Corollary \ref{cor:shypPDE_deghypfullsupp} shows that many other Sturm-Liouville differential expressions yield convolution algebras with the full support property.

\begin{example}
Let $\bm{\zeta} \in \mathrm{C}^1(0,\infty)$ be a nonnegative decreasing function and let $\kappa > 0$. The differential expression
\[
\ell = - x^2 {d^2 \over dx^2} - \bigl[\kappa + x \bigl(1 + \bm{\zeta}(x)\bigr)\bigr] {d \over dx}, \qquad 0 < x < \infty
\]
is a particular case of \eqref{eq:shypPDE_elldiffexpr}, obtained by considering $p(x) = x e^{-\kappa/x + I_{\bm{\zeta}}(x)}$ and $r(x) = {1 \over x} e^{-\kappa/x + I_{\bm{\zeta}}(x)}$, where $I_{\bm{\zeta}}(x) = \int_1^x \bm{\zeta}(y) {dy \over y}$. (If $\kappa = 1$ and $\bm{\zeta}(x) = 1-2\alpha > 0$, we recover the normalized Whittaker operator.) The change of variable $z = \log x$ transforms $\ell$ into the standard form $\widetilde{\ell} = - {d^2 \over dz^2} - {A'(z) \over A(z)}{d \over dz}$, where ${A'(z) \over A(z)} = \kappa e^{-\kappa z} + \bm{\zeta}(e^z)$. It is clear that $\gamma(a) = -\infty$ and that $\ell$ satisfies Assumption \ref{asmp:shypPDE_SLhyperg} with $\eta = 0$, and it is not difficult to show that the boundary condition \eqref{eq:shypPDE_Lop_leftBC} holds. Consequently, the Sturm-Liouville operator $\ell$ gives rise to a convolution structure such that $\supp(\delta_x * \delta_y) = [0,\infty)$ for all $x, y > 0$.
\end{example}

\section{Convolution algebras of functions} \label{sec:convalgL1}

\subsection{A family of $L_1$ spaces}

We now turn our attention the the $\mathcal{L}$-convolution of functions, defined in the following way (recall that $\mathcal{T}^y$ is the $\mathcal{L}$-translation \eqref{eq:shypPDEtransl_def}):

\begin{definition} 
Let $h, g:[a,b) \longrightarrow \mathbb{C}$. If the integral
\begin{equation} \label{eq:shypPDEconvfun_def}
(h * g)(x) = \int_a^b (\mathcal{T}^y h)(x)\, g(y)\, r(y) dy = \int_a^b \int_{[a,b)} \! h(\xi) \bm{\nu}_{x,y}(d\xi)\, g(y)\, r(y) dy
\end{equation}
exists for almost every $x \in [a,b)$, then we call it the \emph{$\mathcal{L}$-convolution} of the functions $h$ and $g$.
\end{definition}

The $\mathcal{L}$-convolution of functions is trivialized by the $\mathcal{L}$-transform \eqref{eq:entrpf_Ffourierdef}:
\begin{equation} \label{eq:shypPDEconvfun_trivializ}
\bigl(\mathcal{F}(h * g)\bigr)(\lambda) = (\mathcal{F}h)(\lambda) \, (\mathcal{F}g)(\lambda) \quad \text{ for all } \lambda \geq 0 \qquad \bigl( h \in \mathrm{C}_\mathrm{c}[a,b), \; g \in L_1(r) \bigr).
\end{equation}
Indeed, just compute
\begin{align*}
\bigl(\mathcal{F}(h * g)\bigr)(\lambda) & = \int_a^b \int_a^b (\mathcal{T}^y h)(x) g(y) \, r(y) dy \, w_\lambda(x) r(x) dx \\
& = \int_a^b \bigl(\mathcal{F}(\mathcal{T}^y h)\bigr)(\lambda) \, g(y) r(y) dy \\[3pt]
& \smash{= (\mathcal{F}h)(\lambda) \int_a^b g(y) w_\lambda(y) r(y) dy \, = \, (\mathcal{F}h)(\lambda) \, (\mathcal{F}g)(\lambda)} \\[-14pt]
\end{align*}
where the third equality is due to \eqref{eq:shypPDE_transl_Fident}.

We will study the $\mathcal{L}$-convolution as an operator acting on the family of Lebesgue spaces $\{L_{1,\kappa}\}_{-\infty < \kappa \leq \sigma^2}$, where $L_{1,\kappa} = L_1\bigl((a,b), w_{\kappa}(x) r(x)dx\bigr)$. Observe that this is an ordered family:
\begin{equation} \label{eq:shypPDEwl_L1kappa_orderrel}
L_{1,\kappa_2} \subset L_{1,\kappa_1} \qquad \text{whenever } \, -\infty < \kappa_2 \leq \kappa_1 \leq \sigma^2.
\end{equation}
This follows from the fact that (due to the Laplace-type representation \eqref{eq:shypPDE_Wlaplacerep}) we have $0 \leq w_{\kappa_1}(x) \leq w_{\kappa_2}(x)$ for all $x \in [a,b)$ whenever $-\infty < \kappa_2 \leq \kappa_1 \leq \sigma^2$. In particular, the space $L_{1,0} \equiv L_1(r)$ is contained in the spaces $L_{1,\kappa}$ with $0 \leq \kappa \leq \sigma^2$.

The basic properties of the $\mathcal{L}$-transform, translation and convolution on the spaces $L_{1,\kappa}$ are as follows (we write $\|\cdot\|_{1,\kappa} := \|\cdot\|_{L_{1,\kappa}}$):

\begin{proposition} \label{prop:shypPDEwl_gentransl_L1kappacont}
Let $-\infty < \kappa \leq \sigma^2$, let $h,g \in L_{1,\kappa}$, and fix $y \in [a,b)$. Then:
\begin{enumerate}[itemsep=0pt,topsep=4pt]
\item[\textbf{(a)}] The $\mathcal{L}$-transform $(\mathcal{F} h)(\lambda) := \int_a^b h(x) \, w_\lambda(x) \, r(x) dx$ is, for all $\lambda \geq \sigma^2$, well-defined as an absolutely convergent integral; in addition, $h$ is uniquely determined by $(\mathcal{F}h)\restrict{[\sigma^2, \infty)\!}$.
\item[\textbf{(b)}] The $\mathcal{L}$-translation $\mathcal{T}^y h(x) := \int_{[a,b)} h \, d\bm{\nu}_{x,y}$ is well-defined and it satisfies $\|\mathcal{T}^y h\|_{1,\kappa} \leq w_\kappa(y) \|h\|_{1,\kappa}$ (in particular, $\mathcal{T}^y \bigl(L_{1,\kappa}\bigr) \subset L_{1,\kappa}$).
\item[\textbf{(c)}] The $\mathcal{L}$-convolution $(h * g)(x) := \int_a^b (\mathcal{T}^y h)(x)\, g(y)\, r(y) dy$ is well-defined and it satisfies $\|h * g\|_{1,\kappa} \leq \|h\|_{1,\kappa} \ccdot \|g\|_{1,\kappa}$ (in particular, $L_{1,\kappa} * L_{1,\kappa} \subset L_{1,\kappa}$).
\end{enumerate}
\end{proposition}

\begin{proof}
\textbf{(a)} The absolute convergence of $\int_a^b h(x) \, w_\lambda(x) \, r(x) dx$ is immediate from \eqref{eq:shypPDEwl_L1kappa_orderrel}. Letting $\mu$ be the (possibly unbounded) measure $\mu(dx) = h(x)\, r(x) dx$, the same proof of Proposition \ref{prop:shypPDE_Lfourmeas_props}(ii) shows that if $\widehat{\mu}(\lambda) \equiv (\mathcal{F}h)(\lambda) = 0$ for all $\lambda \geq \sigma^2$, then $\mu$ is the zero measure; consequently, the $\mathcal{L}$-transform determines uniquely the function $h$. \\[-8pt]

\textbf{(b)} Let $h \in L_{1,\kappa}$ and let $\ell^{\langle\kappa\rangle\!}$ be the operator defined in Lemma \ref{lem:shypPDE_modification}. We saw in the proof of Theorem \ref{thm:shypPDE_prodform_weaklim_full} that $\bm{\nu}_{x,y}^{\langle\kappa\rangle\!}(d\xi) = {w_\kappa(\xi) \over w_\kappa(x) w_\kappa(y)} \bm{\nu}_{x,y}(d\xi)$, hence
\[
(\mathcal{T}^y h)(x) = \int_{[a,b)\!} h \, d\bm{\nu}_{x,y} = w_\kappa(x) w_\kappa(y) \int_{[a,b)} {h \over w_\kappa} \, d\bm{\nu}_{x,y}^{\langle \kappa \rangle} = w_\kappa(x) w_\kappa(y) \Bigl(\mathcal{T}_{\langle \kappa \rangle}^y {h \over w_\kappa}\Bigr)(x),
\]
here $\bm{\nu}_{x,y}^{\langle\kappa\rangle\!}$ and $\mathcal{T}_{\langle \kappa \rangle}$ are, respectively, the measure of the product formula and the translation operator associated with $\ell^{\langle\kappa\rangle}$. Since $\|h\|_{1,\kappa} = \bigl\| {h \over w_\kappa}\bigr \|_{L_1(r^{\langle\kappa\rangle\!})}$, it follows from Lemma \ref{lem:shypPDEwl_gentransl_Lpcont} that $\mathcal{T}^y h$ is well-defined and
\[
\|\mathcal{T}^y h\|_{1,\kappa} = w_\kappa(y) \bigl\|\mathcal{T}_{\langle \kappa \rangle}^y \tfrac{h}{w_\kappa}\bigr\|_{L_1(r^{\langle\kappa\rangle\!})} \leq w_\kappa(y) \bigl\|\tfrac{h}{w_\kappa}\bigr\|_{L_1(r^{\langle\kappa\rangle\!})} = w_\kappa(y) \|h\|_{1,\kappa}. \vspace{2pt}
\]

\textbf{(c)} Using part (a), we compute
\begin{align*}
\| h * g \|_{1,\kappa} & \leq \int_a^b \int_a^b |(\mathcal{T}^x h)(\xi)|\, |g(\xi)|\, r(\xi) d\xi \, w_\kappa(x) r(x) dx \\
& = \int_a^b \int_a^b |(\mathcal{T}^\xi h)(x)| \, w_\kappa(x) r(x) dx \, |g(\xi)|\, r(\xi) d\xi \\
& \leq \|h\|_{L_{1,\kappa}} \int_a^b |g(\xi)|\, w_\kappa(\xi) r(\xi) d\xi \, = \, \|h\|_{1,\kappa} \ccdot \|g\|_{1,\kappa}. \qedhere
\end{align*}
\end{proof}

\begin{corollary} \label{cor:shypPDEwl_L1kappa_banachalg}
The Banach space $L_{1,\kappa}$, equipped with the convolution multiplication $h \cdot g \equiv h * g$, is a commutative Banach algebra without identity element.
\end{corollary}

\begin{proof}
Proposition \ref{prop:shypPDEwl_gentransl_L1kappacont}(c) shows that the Whittaker convolution defines a binary operation on $L_{1,\kappa}$ for which the norm is submultiplicative. Since the trivialization property \eqref{eq:shypPDEconvfun_trivializ} extends (by continuity) to all $h, g \in L_{1,\kappa}$, the commutativity and associativity of the $\mathcal{L}$-convolution in the space $L_{1,\kappa}$ is a consequence of the uniqueness of the $\mathcal{L}$-transform (Proposition \ref{prop:shypPDEwl_gentransl_L1kappacont}(a)).

Suppose now that there exists $\mathrm{e} \in L^{1,\kappa}$ such that $h * \mathrm{e} = h$ for all $h \in L_{1,\kappa}$. Then
\[
(\mathcal{F}h)(\lambda) (\mathcal{F}\mathrm{e})(\lambda) = \bigl(\mathcal{F}(h * \mathrm{e})\bigr)(\lambda) = (\mathcal{F}h)(\lambda) \qquad \text{for all } h \in L_{1,\kappa} \text{ and } \lambda \geq \sigma^2.
\]
Clearly, this implies that $(\mathcal{F}\mathrm{e})(\lambda) = 1$ for all $\lambda \geq \sigma^2$. But we know that $\widehat{\delta_a} \equiv 1$, so it follows from (the proof of) Proposition \ref{prop:shypPDE_Lfourmeas_props}(ii) that $\mathrm{e}(x) r(x) dx = \delta_a(dx)$, which is absurd. This shows that the Banach algebra has no identity element.
\end{proof}

An interesting fact about the $\mathcal{L}$-transform and the convolution Banach algebra $(L_{1,\kappa},*)$ is that they admit the following analogue of the well-known Wiener-Lévy theorem on integral equations with difference kernel (compare with \cite[p.\ 164]{krein1962}):

\begin{theorem}[Wiener-Lévy type theorem] \label{thm:shypPDE_Lfour_wiener}
Let $f \in L_{1,\kappa}$ ($-\infty < \kappa \leq \sigma^2$) and $\varrho \in \mathbb{C}$. The following assertions are equivalent:
\begin{enumerate} [itemsep=0pt,topsep=4pt]
\item[\textbf{(i)}] $\varrho + (\mathcal{F}f)(\lambda) \neq 0$ for all $\lambda \in \varPi_\kappa$ (including $\lambda = \infty$);
\item[\textbf{(ii)}] There exists a unique function $g \in L_{1,\kappa}$ such that
\begin{equation} \label{eq:shypPDE_Lfour_wiener}
{1 \over \varrho + (\mathcal{F} f)(\lambda)} = \varrho + (\mathcal{F} g)(\lambda) \qquad (\lambda \in \varPi_\kappa).
\end{equation} 
\end{enumerate}
\end{theorem}

The proof of Theorem \ref{thm:shypPDE_Lfour_wiener} depends on the following lemma, according to which the function $w_\lambda$ is the unique (suitably bounded) solution of the functional equation determined by the corresponding product formula. Below we write $\Delta_\lambda = \sqrt{\lambda - \sigma^2}$ (where the principal branch of the square root is taken); this notation allows us to write the Laplace-type representation \eqref{eq:shypPDE_Wlaplacerep} as $w_{\lambda}(x) = \int_\mathbb{R} \cos(s \Delta_\lambda) \nu_x(ds)$.

\begin{lemma} \label{lem:shypPDEwl_prodform_functeq}
Let $-\infty < \kappa \leq \sigma^2$. Assume that $\vartheta:[a,b) \longrightarrow \mathbb{C}$ is a Borel measurable function such that there exists $C>0$ for which
\begin{equation} \label{eq:shypPDEwl_prodform_functeq_essbound}
\bigl|\vartheta(x)\bigr| \leq C\, w_{\kappa}(x) \qquad \text{for almost every } x \in [a,b)
\end{equation}
and that $\vartheta(x)$ is a nontrivial solution of the functional equation
\begin{equation} \label{eq:shypPDEwl_prodform_functeq}
\vartheta(x) \, \vartheta(y) = \int_{[a,b)} \vartheta(\xi)\, \bm{\nu}_{x,y}(d\xi) \qquad \text{for almost every } x,y \in [a,b).
\end{equation}
Then $\vartheta(x) = w_{\lambda}(x)$ for some $\lambda \in \varPi_\kappa := \bigl\{\lambda \in \mathbb{C} \bigm| |\mathrm{Im}\, \Delta_{\lambda}| \leq \mathrm{Im}\, \Delta_{\kappa}\bigr\}$.
\end{lemma}

\begin{proof}
For $t > 0$ and $x \in (a,b)$, let $\{T_t\}$ be the Feller semigroup generated by the Sturm-Liouville operator $\ell$ and write $(T_t h)(x) = \int_{[a,b)} h(\xi) p_{t,x}(d\xi)$, where $\{p_{t,x}\}_{t> 0, x \in (a,b)}$ is the family of transition kernels.

Assume for the moment that $0 < \kappa \leq \sigma^2$, so that $\vartheta \in L_\infty(r)$. We know from Corollary \ref{cor:shypPDE_prodformreg_measrep} that the measure $\bm{\nu}_{t,x,y}(d\xi) = q_t(x,y,\xi)\, r(\xi) d\xi$ of the time-shifted product formula is given by $\bm{\nu}_{t,x,y} = p_{t,x} * \delta_y$; therefore, using the functional equation \eqref{eq:shypPDEwl_prodform_functeq} we may compute
\begin{equation} \label{eq:shypPDEwl_prodform_functeq_pf2}
\bigl(T_t (\mathcal{T}^y\vartheta)\bigr)\!(x) = \int_{[a,b)\!} \vartheta(\xi)\, \bm{\nu}_{t,x,y}(d\xi) = \int_{[a,b)\!} (\mathcal{T}^y\vartheta)(\xi) \, p_{t,x}(d\xi) = \vartheta(y) \int_{[a,b)\!} \vartheta(\xi)\, p_{t,x}(d\xi) = \vartheta(y) \, (T_t \vartheta)(x)
\end{equation}
where we also used the fact that the transition kernels $\{p_{t,x}\}_{t> 0, x \in (a,b)}$ are absolutely continuous (indeed, by Proposition \ref{prop:shypPDE_L_fellergen_tpdf} we have $p_{t,x}(d\xi) = p(t,x,\xi) r(\xi)d\xi$, where $p(t,x,\xi)$ is given by \eqref{eq:shypPDE_fellersgp_fundsol_spectr}). On the other hand, by \cite[Corollary 4.4]{mckean1956} the Feller semigroup $\{T_t\}$ is such that
\[
{\partial \over \partial t} (T_t h)(x) = -\ell_x (T_t h)(x) \qquad \text{for all } h \in L_\infty(r) \qquad \bigl(t > 0, \; x \in (a,b)\bigr),
\]
and we thus have (noting that $\mathcal{T}^y\vartheta \in L_\infty(r)$, cf.\ Lemma \ref{lem:shypPDEwl_gentransl_Lpcont})
\begin{equation} \label{eq:shypPDEwl_prodform_functeq_pf3}
\ell_x \bigl(T_t (\mathcal{T}^y\vartheta)\bigr)\!(x) = -{\partial \over \partial t} \bigl(T_t (\mathcal{T}^y\vartheta)\bigr)\!(x) = \ell_y \bigl(T_t (\mathcal{T}^x\vartheta)\bigr)\!(y) 
\end{equation}
where the second equality holds because $\bigl(T_t (\mathcal{T}^y\vartheta)\bigr)\!(x) = \bigl(T_t (\mathcal{T}^x\vartheta)\bigr)\!(y)$. Combining \eqref{eq:shypPDEwl_prodform_functeq_pf2} and \eqref{eq:shypPDEwl_prodform_functeq_pf3}, we deduce that \[
\ell_y[\vartheta(y) \, (T_t \vartheta)(x)] = -{\partial \over \partial t}[\vartheta(y) \, (T_t \vartheta)(x)] = -{\partial \over \partial t}[\vartheta(x) \, (T_t \vartheta)(y)] = \ell_x[\vartheta(x) \, (T_t \vartheta)(y)]
\]
for all $t > 0$ and almost every $x, y \in (a,b)$, and therefore
\begin{equation} \label{eq:shypPDEwl_prodform_functeq_pf4}
{-{\partial \over \partial t} (T_t \vartheta)(x) \over (T_t \vartheta)(x)} = {\ell_x \vartheta(x) \over \vartheta(x)} = {\ell_y \vartheta(y) \over \vartheta(y)} = \lambda
\end{equation}
for some constant $\lambda \in \mathbb{C}$. From the last equality (which holds for almost every $y$) it follows that $\vartheta(y) = c_1 w_\lambda(y) + c_2 \bm{u}_\lambda(y)$, where $\bm{u}_\lambda$ is a solution of $\ell(u) = \lambda u$ linearly independent of $w_\lambda$ and $c_1,c_2 \in \mathbb{C}$ are constants. In particular, $\vartheta$ is continuous (in fact, by \eqref{eq:shypPDEwl_prodform_functeq_essbound} we have $\vartheta \in \mathrm{C}_0[a,b)$), and by continuity the functional equation \eqref{eq:shypPDEwl_prodform_functeq} holds for all $x,y \in [a,b)$; moreover, if we let $y_0$ be such that $\vartheta(y_0) \neq 0$, we see that
\[
\vartheta(x) = {1 \over \vartheta(y_0)} \int_{[a,b)\!} \vartheta(\xi)\, \bm{\nu}_{x,y_0}(d\xi) \longrightarrow {1 \over \vartheta(y_0)} \int_{[a,b)\!} \vartheta(\xi)\, \delta_{y_0}(d\xi) = 1 \qquad \text{as } x \downarrow a.
\]

In order to show that $\vartheta(x) = w_\lambda(x)$, by Lemma \ref{lem:entrpf_ode_wsol} it only remains to prove that $\lim_{x \downarrow a} \vartheta^{[1]}(x) = 0$. We know that $\lim_{t \downarrow 0}(T_t\vartheta)(x) = \vartheta(x)$ and, by \eqref{eq:shypPDEwl_prodform_functeq_pf4}, ${\partial \over \partial t} (T_t \vartheta)(x) = -\lambda (T_t \vartheta)(x)$, hence
\[
(T_t\vartheta)(x) = e^{-\lambda t}\vartheta(x) \qquad (t \geq 0,\; x \in (a,b))
\]
and therefore
\[
(\mathcal{R}_\eta \vartheta)(x) \equiv \int_0^\infty e^{-\eta t} (T_t\vartheta)(x) dt = {\vartheta(x) \over \lambda + \eta} \qquad (\eta > 0,\; x \in (a,b))
\]
where $\{\mathcal{R}_\eta\}_{\eta > 0}$ is the resolvent of the Feller semigroup $\{T_t\}_{t \geq 0}$. Since $\vartheta \in \mathrm{C}_0[a,b)$, it follows from the definition of generator and resolvent of the Feller semigroup that $(\mathcal{R}_\eta \vartheta)(x)$ belongs to $\mathcal{D}_\mathcal{L}^{(0)}$ (see \cite{fukushima2014}), so we conclude that $\vartheta^{[1]}(x) = (\lambda + \eta) (\mathcal{R}_\eta \vartheta)^{[1]}(x) \longrightarrow 0$ as $x \downarrow a$, as desired.

Finally, suppose that $\kappa \leq 0$ and choose $\kappa_0 < \kappa$. Recalling Lemma \ref{lem:shypPDE_modification}, it is easily seen that the function $\vartheta^{\langle\kappa_0\rangle\!}(x) := {\vartheta(x) \over w_{\kappa_0}(x)}$ is bounded almost everywhere by $w_{\kappa-\kappa_0}^{\langle\kappa_0\rangle\!}(x) = {w_\kappa(x) \over w_{\kappa_0}(x)} \in \mathrm{C}_0[a,b)$ and satisfies the functional equation
\begin{equation} \label{eq:shypPDEwl_prodform_functeq_pf5}
\vartheta^{\langle\kappa_0\rangle\!}(x) \, \vartheta^{\langle\kappa_0\rangle\!}(y) = \int_{[a,b)} \vartheta^{\langle\kappa_0\rangle\!}(\xi)\, \bm{\nu}_{x,y}^{\langle\kappa_0\rangle\!}(d\xi) \qquad \text{for almost every } x,y \in [a,b)
\end{equation}
where, as seen above, $\bm{\nu}_{x,y}^{\langle\kappa_0\rangle\!}(d\xi) = {w_{\kappa_0}(\xi) \over w_{\kappa_0}(x) w_{\kappa_0}(y)} \bm{\nu}_{x,y}(d\xi)$. Using the proof given for the case $0 < \kappa \leq \sigma^2$ (where we replace the associated Sturm-Liouville operator by $\ell^{\langle\kappa_0\rangle}$, etc.), we deduce that $\vartheta^{\langle\kappa_0\rangle\!}(x) = w_{\lambda_0}^{\langle\kappa_0\rangle\!}(x)$ for some $\lambda_0 \in \mathbb{C}$. Consequently, $\vartheta(x) = w_\lambda(x)$ for some $\lambda \in \mathbb{C}$.

It only remains to show that $\lambda \in \varPi_\kappa$. Indeed, taking into account the Laplace-type representation for $w_{\lambda}(x)$, \eqref{eq:shypPDEwl_prodform_functeq_essbound} holds if and only if 
\[
\biggl|\int_\mathbb{R} \cos(s \Delta_\lambda) \nu_x(ds)\biggr| \leq C \int_\mathbb{R} \cos(s \Delta_\kappa) \nu_x(ds)\qquad \text{for almost every } x \in [a,b)
\]
and clearly this takes place if and only if $\lambda \in \varPi_\kappa$.
\end{proof}

As a corollary, we obtain a characterization of the multiplicative linear functionals on the Banach algebra $(L_{1,\kappa},*)$:

\begin{corollary} \label{cor:shypPDEwl_L1kappa_functhomomorph}
Let $-\infty < \kappa \leq \sigma^2$. Let $J: L_{1,\kappa} \longrightarrow \mathbb{C}$ be a linear functional satisfying
\[
J(h*g) = J(h) \ccdot J(g) \qquad \text{for all } h,g \in L_{1,\kappa}.
\]
Then $J(h) = \int_{[a,b)} h(\xi) w_\lambda(\xi) r(\xi) d\xi$ for some $\lambda \in \varPi_\kappa := \bigl\{\lambda \in \mathbb{C} \bigm| |\mathrm{Im}\, \Delta_{\lambda}| \leq \mathrm{Im}\, \Delta_{\kappa}\bigr\}$.
\end{corollary}

\begin{proof}
By the well-known theorem on duality of $L_p$ spaces (e.g.\ \cite[Theorem 4.4.1]{bogachev2007}), 
\[
J(h) = \int_a^b h(\xi) \, \vartheta(\xi) \, r(\xi) d\xi
\]
where ${\vartheta \over w_\kappa} \in L_\infty(r)$, i.e., \eqref{eq:shypPDEwl_prodform_functeq_essbound} holds. Since $J(h*g) = J(h) \ccdot J(g)$, for $h,g \in L_{1,\kappa}$ we have
\begin{align*}
\int_a^b h(\xi) \, \vartheta(\xi) \, r(\xi) d\xi \cdot \int_a^b g(\xi) \, \vartheta(\xi) \, r(\xi) d\xi & = \int_a^b \int_a^b (\mathcal{T}^\xi h)(y)\, g(y) r(y) dy \: \vartheta(\xi) r(\xi) d\xi \\
& = \int_a^b \int_a^b (\mathcal{T}^y h)(\xi) \, \vartheta(\xi) r(\xi) d\xi \: g(y) r(y) dy \\
& = \int_a^b \int_a^b (\mathcal{T}^y \vartheta)(\xi) \, h(\xi) r(\xi) d\xi \: g(y) r(y) dy
\end{align*}
where the last equality follows from the commutativity of the $\mathcal{L}$-convolution, cf.\ Corollary \ref{cor:shypPDEwl_L1kappa_banachalg}. (The commutativity easily extends to $h \in L_{1,\kappa}$ and ${\vartheta \over w_\kappa} \in L_\infty(r)$ via a continuity argument.) Since $h$ and $g$ are arbitrary, 
\[
\vartheta(x) \, \vartheta(y) = (\mathcal{T}^y \vartheta)(x) \equiv \int_{[a,b)} \vartheta(\xi)\, \bm{\nu}_{x,y}(d\xi) \qquad \text{for almost every } x,y \in [a,b)
\]
and the conclusion follows from Lemma \ref{lem:shypPDEwl_prodform_functeq}.
\end{proof}

\begin{proof}[Proof of Theorem \ref{thm:shypPDE_Lfour_wiener}]
The proof of \emph{(i) $\!\implies\!$ (ii)} is entirely analogous to the proof of Theorem 15.15 of \cite{yakubovichluchko1994}, appealing to Corollaries \ref{cor:shypPDEwl_L1kappa_banachalg} and \ref{cor:shypPDEwl_L1kappa_functhomomorph} in place of the analogous results for the Kontorovich-Lebedev convolution. 

The proof of \emph{(ii) $\!\implies\!$ (i)} is straightforward: since the integral defining $(\mathcal{F} g)(\lambda)$ converges absolutely whenever $\lambda \in \varPi_\kappa$ (this follows from \eqref{eq:shypPDEwl_L1kappa_orderrel} and the condition $g \in L_{1,\kappa}$), it is clear that if for some $\lambda_0 \in \varPi_\kappa$ we have $\varrho + (\mathcal{F} f)(\lambda_0) = 0$, then \eqref{eq:shypPDE_Lfour_wiener} will fail at $\lambda = \lambda_0$, regardless of the choice of $g \in L_{1,\kappa}$.
\end{proof}

\subsection{Application to convolution integral equations}

In this final subsection it will be shown that the Wiener-Lévy type theorem \eqref{thm:shypPDE_Lfour_wiener} can be used to derive an existence and uniqueness result for integral equations of the second kind belonging to the following class of $\mathcal{L}$-convolution equations:

\begin{definition}
The integral equation
\begin{equation} \label{eq:shypPDE_2ndkindinteq}
h(x) + \int_a^b h(y) \, J(x,y)\, dy = \psi(x),
\end{equation}
where $\psi$ and $J$ are known functions and $h$ is to be determined, is said to be an \emph{$\mathcal{L}$-convolution equation} if $J(x,y) = (\mathcal{T}^{x\!} f)(y)r(y)$ for some function $f \in L_{1,\sigma^2}$. In other words, \eqref{eq:shypPDE_2ndkindinteq} is an $\mathcal{L}$-convolution equation if it can be represented as
\begin{equation} \label{eq:shypPDE_convinteq}
h(x) + (h*f)(x) = \psi(x)
\end{equation}
with $f \in L_{1,\sigma^2}$.
\end{definition}

Suppose that $\psi, f \in L_{1,\kappa}$ (where $-\infty < \kappa \leq \sigma^2$) and consider the $\mathcal{L}$-convolution equation \eqref{eq:shypPDE_convinteq}. Applying the $\mathcal{L}$-Fourier transform, we obtain
\begin{equation} \label{eq:shypPDE_convinteq_transf}
(\mathcal{F} h)(\lambda) \bigl[ 1 + (\mathcal{F} f)(\lambda) \bigr] = (\mathcal{F} \psi)(\lambda) \qquad (\lambda \in \varPi_\kappa).
\end{equation}
As seen in Theorem \ref{thm:shypPDE_Lfour_wiener}, the expression between square brackets can be written as $[1 + (\mathcal{F}g)(\lambda)]^{-1}$ (with $g \in L_{1,\kappa}$) if and only if $1 + (\mathcal{F}f)(\lambda) \neq 0$ for all $\lambda \in \varPi_\kappa$; whenever this is so, for $\lambda \in \varPi_\kappa$ we get $(\mathcal{F} h)(\lambda) = (\mathcal{F} \psi)(\lambda) \bigl[ 1 + (\mathcal{F} g)(\lambda) \bigr]$, meaning that
\begin{equation} \label{eq:shypPDE_convinteq_sol}
h(x) = \psi(x) + (\psi * g)(x) = \psi(x) + \int_a^b \psi(y) \, J_g(x,y) \, dy
\end{equation}
where $J_g(x,y) = (\mathcal{T}^{x\!} g)(y) r(y)$. Summing up, we have proved:

\begin{theorem} \label{thm:shypPDE_convinteq_existuniq}
Assume that $J(x,y) = (\mathcal{T}^{x\!} f)(y)r(y)$. 

If $1 + (\mathcal{F}f)(\lambda) \neq 0$ for all $\lambda \in \varPi_\kappa$ (including $\lambda = \infty$), then, for each given $\psi \in L_{1,\kappa}$, the integral equation \eqref{eq:shypPDE_2ndkindinteq} admits a unique solution $h \in L_{1,\kappa}$; moreover, this solution can be written in the form \eqref{eq:shypPDE_convinteq_sol} for some function $g \in L_{1,\kappa}$.

Conversely, if $1 + (\mathcal{F}f)(\lambda_0) = 0$ for some $\lambda_0 \in \varPi_\kappa$, then there exists no $g \in L_{1,\kappa}$ satisfying the integral equation \eqref{eq:shypPDE_2ndkindinteq}. 
\end{theorem}

As an interesting particular case, we deduce the existence and uniqueness of solution for integral equations involving the density $q_t(x,y,\xi)$ of the time-shifted product formula (cf.\ Proposition \ref{prop:shypPDE_expprodform}):

\begin{corollary}
For each fixed $t > 0$, $a < x < b$ and $\psi \in L_1(r)$, the integral equation 
\[
h(y) + \int_a^b h(\xi) \, q_t(x,y,\xi) r(\xi) d\xi = \psi(y)
\]
has a unique solution $h \in L_1(r)$ which can be written in the form \eqref{eq:shypPDE_convinteq_sol} for some function $g \in L_1(r)$.
\end{corollary}

\begin{proof}
Let us justify that this result is obtained by setting $f = f_{t,x} := p(t,x,\cdot)$ in the statement of Theorem \ref{thm:shypPDE_convinteq_existuniq}. Notice first that by Lemma \ref{lem:shypPDE_exptriplepf_probmeas} we have $f_{t,x} \in L_{1,0} \equiv L_1(r)$. Moreover, we have $(\mathcal{F}f_{t,x})(\lambda) = e^{-t\lambda} w_\lambda(x)$ (cf.\ \eqref{eq:shypPDE_fellersgp_fundsol_spectr}), thus $1 + (\mathcal{F}f_{t,x})(0) = 2$ and
\[
|1 + (\mathcal{F}f_{t,x})(\lambda)| \geq 1 - e^{-t\,\mathrm{Re}\,\lambda} |w_\lambda(x)| > 0, \qquad \lambda \in \varPi_0 \setminus \{0\}
\]
(this is easily seen to hold by setting $\lambda = \tau^2 + \sigma^2$ and recalling the estimate \eqref{eq:shypPDE_Wbound}). Recalling that
\[
\bigl(\mathcal{F}(\mathcal{T}^y f_{t,x})\bigr)(\lambda) = (\mathcal{F}f_{t,x})(\lambda) \, w_\lambda(y) = e^{-t\lambda} w_\lambda(x) w_\lambda(y) = \bigl(\mathcal{F} q_t(x,y,\cdot)\bigr)(\lambda)
\]
where we used \eqref{eq:shypPDE_transl_Fident}, we see that $(\mathcal{T}^y f_{t,x})(\xi) = q_t(x,y,\xi)$, so that the corollary is a particular case of the theorem.
\end{proof}

\section*{Acknowledgements}

The first and third authors were partly supported by CMUP (UID/MAT/00144/2019), which is funded by Fundação para a Ciência e a Tecnologia (FCT) (Portugal) with national (MCTES) and European structural funds through the programs FEDER, under the partnership agreement PT2020, and Project STRIDE -- NORTE-01-0145-FEDER-000033, funded by ERDF -- NORTE 2020. The first author was also supported by the grant PD/BD/135281/2017, under the FCT PhD Programme UC|UP MATH PhD Program. The second author was partly supported by FCT/MCTES through the project CEMAPRE -- UID/MULTI/00491/2013.

\renewcommand{\bibname}{References} 
\begin{small}

\end{small}

\end{document}